\documentclass[onefignum,onetabnum]{siamart250211}

\usepackage{mathtools,amsmath,amssymb,mathrsfs,amsfonts}
\usepackage{epsfig}
\usepackage{leftidx}
\usepackage{graphicx}
\usepackage{amsfonts}
\usepackage{color,epstopdf}
\usepackage{cite}
\usepackage{graphicx,wrapfig,tabularx}
\usepackage{flafter}
\usepackage{fancyhdr}
\usepackage{stmaryrd}
\usepackage{graphicx,subfigure}
\usepackage{multicol,multirow,booktabs}
\usepackage{booktabs,threeparttable}
\usepackage[center]{caption2}
\usepackage{epstopdf}
\usepackage [latin1]{inputenc}
\usepackage{multirow}
\usepackage{enumerate}
\usepackage{enumitem}
\usepackage{algpseudocode,algorithm,algorithmicx}
\usepackage{makecell}
\ifpdf
  \DeclareGraphicsExtensions{.eps,.pdf,.png,.jpg}
\else
  \DeclareGraphicsExtensions{.eps}
\fi

\newsiamremark{remark}{Remark}
\newsiamremark{hypothesis}{Hypothesis}
\crefname{hypothesis}{Hypothesis}{Hypotheses}
\newsiamthm{claim}{Claim}

\newtheorem{example}{Example}[section]

\def\ck{{\scriptstyle{K}}}
\def\cm{{\scriptstyle{M}}}

\def\Ck{{{\mathtt{K}}}}
\newcommand{\ze}{\mathrm{e}}

\newcommand{\myvec}[1]{\boldsymbol{#1}}
\newcommand{\zd}{\,\mathrm{d}}

 \newcommand{\absb}[1]{\big|#1\big|}

 \newcommand{\bra}[1]{\left(#1\right)}
\newcommand{\brab}[1]{\big(#1\big)}
\newcommand{\braB}[1]{\Big(#1\Big)}
\newcommand{\brat}[1]{(#1)}
\newcommand{\kbra}[1]{\left[#1\right]}
\newcommand{\kbrab}[1]{\big[#1\big]}

\newcommand{\myinner}[1]{\left\langle#1\right\rangle}
\newcommand{\myinnert}[1]{\langle#1\rangle}
\newcommand{\myinnerb}[1]{\big\langle#1\big\rangle}
\newcommand{\myinnerB}[1]{\Big\langle#1\Big\rangle}
\newcommand{\mynorm}[1]{\left\|#1\right\|}
\newcommand{\mynormb}[1]{\big\|#1\big\|}

\newcommand{\mynormt}[1]{\|#1\|}

\usepackage{cite}

\headers{Long-time stability of IERK methods}{H.-L. Liao, X.-M. Wang, X.-P. Wang and C. Wen}

\title{Long-time stability of implicit-explicit Runge-Kutta methods for two-dimensional incompressible flows\thanks{Submitted to the editors \today.
		\funding{This work is supported by the National Natural Science Foundation of China under grants 12471383, 12271237, 11731006 and 12501551, Basic Research Program of Jiangsu Province under grant BK20252027 and BK20251395, and	Ministry of Education Key Laboratory of NSLSCS under grant 202501.}}}

\author{ Hong-lin Liao\thanks{Corresponding author. ORCID 0000-0003-0777-6832. School of Mathematics, Nanjing University of Aeronautics and Astronautics, Nanjing 211106, China;
		Key Laboratory of Mathematical Modeling and High Performance Computing of Air Vehicles (NUAA), MIIT, Nanjing 211106, China. (\email{liaohl@nuaa.edu.cn}, \email{liaohl@csrc.ac.cn})}
		\and 
		Xiaoming Wang\thanks{School of Mathematical Sciences, Eastern Institute of Technology, Ningbo, Ningbo 315200, China. (\email{wxm.math@outlook.com})}
		\and
		Xuping Wang\thanks{School of Science and Engineering, The Chinese University of Hong Kong, Shenzhen 518172, China. (\email{wangxuping@cuhk.edu.cn}).}
		\and
		Cao Wen\thanks{School of Mathematics, Nanjing University of Aeronautics and Astronautics, Nanjing 211106, China. (\email{wencao@nuaa.edu.cn})}
		}

\usepackage{amsopn}

\begin{document}
  
\maketitle
  
\begin{abstract}
	High-order adaptive time-stepping algorithms are of significant practical value and theoretical interest for accelerating long-time fluid-flow simulations and resolving complex dynamical behaviors. While several high-order implicit-explicit schemes have been proposed in the literature, their long-time stability properties remain largely unexplored.				
	We develop a family of long-time stable implicit-explicit Runge-Kutta (IERK) methods, up to fourth-order temporal accuracy, for the two-dimensional incompressible Navier-Stokes equations in vorticity-stream function formulation. By combining a convolution-type H\"{o}lder inequality with a damping-type multistage Gr\"{o}nwall inequality, we establish a unified analytical framework that proves long-time stability in both the $L^2$ and $H^1$ norms. A key component of the analysis is a mathematical-induction argument that ensures stage-wise boundedness of the vorticity in the $H^\delta$ norm for some $\delta>0$.
	
	To the best of our knowledge, this is the first work to establish large-time stability results for high-order IERK algorithms for the two-dimensional incompressible Navier-Stokes equations. 
	Our IERK schemes employ stiffly accurate diagonally implicit Runge-Kutta  approximations for the linear diffusive term together with explicit Runge-Kutta approximations for the nonlinear advection term. By exploiting the specific structure of the Navier-Stokes model, we derive a reduced set of order conditions-requiring only 5 and 11 conditions for the third- and fourth-order methods, respectively, in contrast to the classical 6 and 18-allowing the construction of a parameterized family of efficient schemes. These IERK methods are particularly well suited for adaptive time-stepping, as they permit significantly enlarged step sizes in actual computations.
	
	Also, a new strategy of adaptive time-stepping with local delay and local backtrack is designed to accurately resolve the small-scale chaotic or high-frequency quasi-periodic  behaviors and efficiently accelerate the large-scale low-frequency periodic motions. Numerical experiments are provided to corroborate the theoretical findings.
	\\[1ex]   
	\textsc{Keywords:} two-dimensional incompressible flow, implicit-explicit Runge-Kutta methods, long-time stability, damping-type multi-stage Gr\"{o}nwall inequality,  simplified  order conditions
	\\[1ex]
	\emph{AMS subject classifications}: 65M12, 65M50, 76B47, 76D05  
\end{abstract}
  
\section{Introduction}
\setcounter{equation}{0}

We construct  high-order  stable implicit-explicit 
Runge-Kutta (IERK) approximations to accelerate the long-time simulations of incompressible Navier-Stokes equation (INSE) in the vorticity-stream function form \cite{ConstantinFoias:1988, FoiasManleyRosaTemam:2001,HillSuli:2000,Tachim-Medjo:1996,Temam:1997}
\begin{align}
&\,	\partial_t\omega+\myvec{u}\cdot\nabla \omega=\nu\Delta\omega+g,\label{cont: INSE-vorticity}\\
	&\,-\Delta\psi=\omega,\quad
	\myvec{u}=(\partial_y\psi,-\partial_x\psi),\label{cont: INSE-velocity}
\end{align}
where $\omega$ is the vorticity, $\psi$ is the stream function, $\myvec{u}:=\bra{u,v}$ is the velocity, $g$ is a given external force and $\nu$ 
is the coefficient of the kinematic viscosity. The divergence-free condition is always fulfilled in the form \eqref{cont: INSE-velocity} due to the fact $\nabla\cdot\myvec{u}=\partial_x\partial_y\psi-\partial_y\partial_x\psi=0$. 
For simplicity, we consider periodic boundary condition so that the domain is a 2-D torus $\mathbb{T}^2$  and assume that all functions have mean zero
over the torus.
It is known that the solution is analytic (Gevrey class regular \cite{FoiasTemam:1998JFA}) in space for a smooth source term $g$.  Therefore, spectral methods are clearly a natural choice for spatial discretization \cite{ShenTangWang:2011Spectral}. We apply Fourier pseudo-spectral method to approximate the spatial operators $\partial_x$, $\partial_y$, $\nabla$ and $\Delta$, as described in the next section, with the associated discrete operators (matrices) $\mathcal{D}_x$, $\mathcal{D}_y$, $\nabla_h$ and $\Delta_h$, respectively. Of course, there is also other spatial discretization method such as discontinuous Galerkin method\cite{LiuShu:2000, NguyenPeraireCockburn: 2011}. These works focus on the influence of high-order spatial discretization methods on the dynamic behaviors of INSE with high Reynolds numbers, which is not the focus of this work.  

The INSE in the vorticity-stream function form \eqref{cont: INSE-vorticity}-\eqref{cont: INSE-velocity} exhibits excellent long-time stability. In particular, the enstrophy $\frac{1}{2}\|\omega\|^2$ and the palinstrophy remain uniformly bounded for all $t>1$ for square integrable initial data, so the system possesses a global attractor (including large-scale low-frequency quasi-periodic motions and abundant small-scale  chaotic behaviors or high-frequency turbulent fluctuations) and invariant measures provided that the source term $g$ is square integrable in space and time independent \cite{ConstantinFoias:1988, FoiasManleyRosaTemam:2001,Temam:1997}. Consequently, in practical numerical simulations it is essential that numerical schemes designed for the dissipative system should inherit the dissipative nature of the original model and be able to capture the long-time dynamical behaviors accurately in the sense of convergence to the global attractor and invariant measures, cf.  \cite{ChenGunzburgerSunWang:2013,ChenGunzburgerSunWang:2016,ChengWang:2016,ChengWangxiaoming:2008,GottliebToneWangWangWirosoetisno:2012,HillSuli:2000,StuartHumphries:1998,ToneWirosoetisno:2006,Wang:2010,Wang:2012,Wang:2016}.  In the case when the long time dynamics is of interest, it is of great importance to develop efficient and long time stable numerical schemes, leading to the development of linearly semi-implicit or implicit-explicit schemes where the nonlinear term is treated explicitly for computational efficiency in order to avoid inner iterations at each time level.  For instance, \cite{GottliebToneWangWangWirosoetisno:2012} investigated the long time stability of the implicit-explicit Euler scheme 
\begin{align}\label{scheme: BDF1-vorticity}
	D_1\omega^{n}+\myvec{u}^{n-1}\cdot\nabla_h \omega^{n-1}=\nu\Delta_h\omega^{n}+g^{n},
\end{align}
where $D_1\omega^{n}:=(\omega^{n}-\omega^{n-1})/\tau$. It was proven that the first-order method is long time stable in the $L^2$ and $H^1$ norms, and that the global attractor as well as the invariant measures of the scheme converge to those of the INSE at vanishing time step. In addition, one of the authors of this work  investigated the second-order implicit-explicit schemes \cite{Wang:2012} by combining a second order backward-differentiation
for the diffusion term and an explicit second-order Adams-Bashforth/Gear's extrapolation treatment for the nonlinear advection term.
\begin{align}\label{scheme: BDF2-vorticity}
	D_2\omega^{n}+(2\myvec{u}^{n-1}-\myvec{u}^{n-2})\cdot\nabla_h (2\omega^{n-1}-\omega^{n-2})=\nu\Delta_h\omega^{n}+g^{n},
\end{align}
where  $D_2\omega^{n}:=(\frac32\omega^{n}-2\omega^{n-1}+\frac12\omega^{n-2})/\tau$
is the second-order backward differentiation formula. The uniform-in-time $L^2$, $H^1$ and $H^2$
bounds of the numerical solution was established in \cite{Wang:2012} so that a statistical convergence
becomes available. The same technique was applied to the double diffusive model in \cite{ToneWangWirosoetisno:2015}. To accelerate the long-time simulations of \eqref{cont: INSE-vorticity}-\eqref{cont: INSE-velocity}, high-order time approximations are desirable. By employing a  skew symmetric form for the nonlinear advection term, the so-called Temam technique, at the discrete level,
\begin{align}\label{discrete: skew symmetric form}
	\mathtt{c}_h(\myvec{u},\omega):=\frac{1}2\myvec{u}\cdot\nabla_h \omega+\frac{1}2\nabla_h \cdot(\myvec{u}\omega),
\end{align}
Cheng and Wang \cite{ChengWang:2016} studied the following high-order implicit-explicit multistep schemes proposed in \cite{GottliebWang:2012}, up to fourth-order accuracy, for the vorticity equation \eqref{cont: INSE-vorticity},
\begin{align}\label{scheme: Adams-Moulton-vorticity}
	\frac{\omega^{n}-\omega^{n-1}}{\tau}+\sum_{j=0}^{n-1}b_{n-j}^{(k)}\mathtt{c}_h(\myvec{u}^{j},\omega^{j})=
	\nu\sum_{j=0}^{n}d_{n-j}^{(k)}\Delta_h\omega^{j}+\tilde{g}^{n}\quad\text{for $k=2,3,4$,}
\end{align}
where $\tilde{g}^n=\frac{1}{\tau}\int_{t_{n-1}}^{t_{n}} g(\cdot,\tau) d\tau$ and $b_{j}^{(k)}$ are the  Adams-Bashforth coefficients, while $d_{j}^{(k)}$ are the special Adams-Moulton coefficients on a stretched stencil such that 
\begin{align}\label{scheme: first-coefficient dominant}
	d_{0}^{(k)}>\sum_{j=0}^{n-1}\absb{d_{n-j}^{(k)}}\quad\text{for $k=2,3,4$.}
\end{align} 
By taking advantage of the first-coefficient domination  condition \eqref{scheme: first-coefficient dominant}, the long-time stability in the $L^2$ and $H^m$ $(m\ge1)$ norms was established for the implicit-explicit multistep schemes \eqref{scheme: Adams-Moulton-vorticity} by a mathematical induction to the boundedness of vorticity in the $H^\delta$ norm for some $\delta>0$. 

Another venue to enhance efficiency is time-adaptivity so that relatively large time step can be utilized when the evolution of the solution is relatively slow, and small time step is used when the evolution is relatively fast, cf. Figures \ref{fig: ex1-Radau1}-\ref{fig: ex1-Radau1-tau=0.1}. While time-adaptivity alone is relatively straightforward, the concurrent maintenance of long time stability of the solution to the scheme is highly nontrivial.  Indeed, no long-time stable variable-step version of the high-order implicit-explicit multistep schemes \eqref{scheme: BDF2-vorticity} and \eqref{scheme: Adams-Moulton-vorticity} are available so far.
Therefore, implicit-explicit time approximations allowing variable time-step sizes, such as the non-uniform implicit-explicit multi-step methods and Runge-Kutta (RK) methods, are highly desirable for the INSE \eqref{cont: INSE-vorticity}-\eqref{cont: INSE-velocity}.

In recent years, some variable-step multi-step algorithms  were developed and analyzed for incompressible fluid dynamics, see \cite{ArchillaNovo:2022,BoisneaultDubuisPicasso:2023, DeCariaSchneier:2021,WangWangMao:2022, ChuWangWangZhang: 2023} and the references therein. However, due to the non-A-stable property of high-order variable-step multi-step schemes and the additional degrees of freedom from historical time-steps, the stability and convergence analysis of such high-order time discretizations are always theoretically challenging, also see
\cite{LiLiao:2022,LiaoZhang:2021}. 
The existing theoretical results always focus on the stability and convergence on a fixed time interval, while the long-time stability analysis (especially, the asymptotic decaying estimates of solution) for high-order variable-step multi-step schemes seems rare (maybe, missing) in the literature. 
This issue will be  addressed in a forthcoming report; and in this article, we aim to establish the long-time stability of a class of IERK methods for the INSE \eqref{cont: INSE-vorticity}-\eqref{cont: INSE-velocity}. 

The current work is inspired by some recent developments \cite{ChenHuangLiaoYi:2026cicp,LiaoWangWen:2024IERK,LiaoWangWen:2024JCP,LiaoTangWangZhou:2025R-IERK,WangZhaoLiao:2025} on the construction and stability analysis of computationally efficient RK methods preserving the original energy dissipation law of nonlinear gradient flows. The single-step nature of RK methods makes them particularly suitable for designing time adaptive algorithms for nonlinear stiff equations, cf. \cite{AscherRuuthSpiteri:1997,BoscarinoPareschiRusso:2017,BoscarinoRusso:2009,CalvoFrutosNovo:2001,CardoneJackiewiczSanduZhang:2014MMA,DimarcoPareschi:2013,FuTangYang:2024,GuesmiGrotteschiStiller:2023,IzzoJackiewicz:2017,KanevskyCarpenterGottliebHesthaven:2007,KennedyCarpenter:2003,LiuZou:2006,PareschiRusso:2005} and the references therein. Meanwhile, to the best of our knowledge, this is the first time such long-time stability of IERK methods, up to fourth-order accuracy,  is addressed for the INSE.  The main theoretical tools include a convolution-type H\"{o}lder inequality (see Lemma \ref{lemma: bound quadratic form}), a damping-type  Gr\"{o}nwall inequality for multi-stage methods (see Lemma \ref{lemma: decaying gronwall multi-stage RK}) and a unified discrete energy argument (cf. \cite{LiaoTangWangZhou:2025R-IERK,WangZhaoLiao:2025}) of the internal stability   with a complete mathematical induction to the stage-wise boundedness of vorticity in the $H^\delta$ norm. 
Unlike our recent works \cite{LiaoWangWen:2024IERK,LiaoTangWangZhou:2025R-IERK} where only schemes of order less than four are explicitly developed, 
a class of fourth-order IERK methods is designed by simplifying the order conditions, see Table \ref{table: order condition} and more details in Section \ref{appendix: order conditions} of the supplementary material.

Consider a nonuniform time mesh $0=t_0<t_1<\cdots<t_N=T$ with the time-step $\tau_n=t_n-t_{n-1}$ for $n\ge1$ and the maximum time step $\tau_{\max} =\max_{n\geq 1} \tau_n$.	
For an $s$-stage IERK method, let $\omega_h^{n,i}$, $\psi_h^{n,i}$ and $\myvec{u}_h^{n,i}$ be the numerical approximation   of  the vorticity $\omega$, the stream function $\psi$ and the velocity $\myvec{u}$, respectively, at the time stage $t_{n,i}:=t_{n-1}+c_{i}\tau_n$ with the abscissas $c_1:=0$, $c_i>0$ for $2\le i\le s-1$, and $c_{s}:=1$.  Always, the number of implicit stages is denoted by $s_{\mathrm{I}}:=s-1$. 
To integrate the INSE \eqref{cont: INSE-vorticity}-\eqref{cont: INSE-velocity} from $t_{n-1}$ ($n\ge1$) to the next grid point $t_n$, we consider the following $s$-stage stiffly-accurate IERK method:  with the stage starting values $\myvec{u}^{n,1}_h:=\myvec{u}^{n-1}_h$ and $\omega^{n,1}_h:=\omega^{n-1}_h$ for $n\ge1$, to seek $\myvec{u}^{n}_h:=\myvec{u}^{n,s}_h$ and $\omega^{n}_h:=\omega^{n,s}_h$ such that
\begin{align}	
	&\,\omega_h^{n,i+1}-\omega_h^{n,1}
	+\tau_n\sum_{j=1}^{i}\hat{a}_{i+1,j}\mathtt{c}_h(\myvec{u}_h^{n,j},\omega_h^{n,j})
	=
	\nu\tau_n\sum_{j=2}^{i+1}a_{i+1,j}\brab{\Delta_h\omega_h^{n,j}+g^{n,j}},\label{scheme: general IERK-vorticity}\\
	&\,-\Delta_h\psi_h^{n,i+1}=\omega_h^{n,i+1},\quad
	\myvec{u}_h^{n,i+1}=\brat{\mathcal{D}_y\psi_h^{n,i+1},-\mathcal{D}_x\psi_h^{n,i+1}},\label{scheme: general IERK-velocity}
\end{align}
for $1\le i\le s_{\mathrm{I}}$. 
At the time stage $t_{n,i+1}$, once the velocity $\myvec{u}_h^{n,j}$ and the vorticity $\omega_h^{n,j}$ $(1\le j\le i)$ are available, we compute the vorticity  $\omega_h^{n,i+1}$ via the IERK procedure \eqref{scheme: general IERK-vorticity}, obtain the stream function $\psi_h^{n,i+1}$ and the velocity $\myvec{u}_h^{n,i+1}$ via the Poisson equation in \eqref{scheme: general IERK-velocity}. The associated Butcher tableaux of \eqref{scheme: general IERK-vorticity} reads
\begin{align*}
	&\begin{array}{c|c}
		\mathbf{c} & A \\
		\hline\\[-8pt] 	& \mathbf{b}^T
	\end{array}
	=	\begin{array}{c|ccccc}
		c_{1} & 0 &  &  &  &   \\
		c_{2} & 0& a_{22} &  &  &   \\
		c_{3} & 0 & a_{32} & a_{33} &  &   \\
		\vdots & \vdots & \vdots & \ddots & \ddots &  \\[2pt]
		c_{s} & 0 & a_{s,2} &  \cdots  & a_{s,s-1}   & a_{s,s} \\[2pt]
		\hline  & 0 & a_{s,2} &  \cdots  & a_{s,s-1}   & a_{s,s}
	\end{array},\begin{array}{c|c}
		\mathbf{c} & \widehat{A} \\
		\hline\\[-8pt] 	& \hat{\mathbf{b}}^T
	\end{array}
	=\begin{array}{c|ccccc}
		c_{1} & 0 &  &  &  &   \\
		c_{2} & \hat{a}_{21} & 0 &  &  &   \\
		c_{3} & \hat{a}_{31} & \hat{a}_{32} & 0 &  &   \\
		\vdots & \vdots & \vdots & \ddots & \ddots &  \\[2pt]
		c_{s} & \hat{a}_{s,1} & \hat{a}_{s,2} &  \cdots  & \hat{a}_{s,s-1}   & 0 \\[1pt]
		\hline\\[-8pt]  & \hat{a}_{s,1} & \hat{a}_{s,2} &  \cdots  & \hat{a}_{s,s-1}   & 0
	\end{array}.
\end{align*}

\begin{table}[htb!]\centering
	\begin{threeparttable}
		\centering 
		\renewcommand\arraystretch{1.3}
		\belowrulesep=0pt\aboverulesep=0pt
		\caption{Order conditions for IERK methods up to fourth-order accuracy.}
		\label{table: order condition}
		\vspace{2mm}
		\begin{tabular}{c|cc|c}
			\toprule 
			\multirow{2}*{Order} & \multicolumn{2}{c|}{Stand-alone conditions} & Coupling condition \\
			\cmidrule{2-4}
			& Implicit part & Explicit part & \\
			\midrule 
			1 & $\mathbf{b}^{T} \mathbf{1}=1$ & $\hat{\mathbf{b}}^{T} \mathbf{1} = 1$ & - \\[2pt]
			\midrule 2 & $\mathbf{b}^{T} \mathbf{c} = \tfrac1{2}$ & $\hat{\mathbf{b}}^{T} \mathbf{c} = \tfrac1{2}$ & - \\[2pt]
			\midrule 3 &  & $\hat{\mathbf{b}}^{T} \mathbf{c}^{.2} = \tfrac1{3}$ & \\
			& $\mathbf{b}^{T} A  \mathbf{c} = \tfrac1{6}$ & $\hat{\mathbf{b}}^{T} \widehat{A}  \mathbf{c} =\tfrac1{6}$ & {$\mathbf{b}^{T} \widehat{A}  \mathbf{c} = \tfrac1{6}, \; \hat{\mathbf{b}}^{T} A  \mathbf{c} = \tfrac1{6}$} \\[2pt]
			\midrule 
			& & $ \hat{\mathbf{b}}^{T} \mathbf{c}^{.3} = \frac{1}{4}$  &  \\[2pt]
			&  &  {$ \hat{\mathbf{b}}^{T} [\mathbf{c} \odot (\widehat{A}  \mathbf{c})] = \frac{1}{8}$}  & $\widehat{\mathbf{b}}^{T} [\mathbf{c} \odot (A \mathbf{c})] = \frac{1}{8}$ \\[2pt]
			4 &  & {$\hat{\mathbf{b}}^{T} \widehat{A} \mathbf{c}^{.2} = \frac{1}{12}$} &   $\mathbf{b}^{T} \widehat{A} \mathbf{c}^{.2} = \frac{1}{12}$, $\hat{\mathbf{b}}^{T} A \widehat{A}  \mathbf{c} + \mathbf{b}^{T} \widehat{A}^2  \mathbf{c} = \frac{1}{12}$  \\[2pt]
			& $\mathbf{b}^{T} A^{2}  \mathbf{c} =\frac{1}{24}$ &  $\hat{\mathbf{b}}^{T} \widehat{A}^{2}  \mathbf{c} = \frac{1}{24}$  &{ $\mathbf{b}^{T} \widehat{A} A \mathbf{c} + \hat{\mathbf{b}}^{T} A^2 \mathbf{c} = \frac{1}{12}$} \\[2pt]
			& & & { $\mathbf{b}^{T} A \widehat{A}  \mathbf{c} = \frac{1}{24}, \ \hat{\myvec{b}}^\top \hat{A} A \myvec{c} = \frac{1}{24}$} \\
			\bottomrule
		\end{tabular}
		\footnotesize
		\tnote{* For the vectors $\mathbf{x}$ and $\mathbf{y}$, $\mathbf{x}\odot\mathbf{y}:=(x_1y_1,x_2y_2,\cdots,x_sy_s)^T$ and $\mathbf{x}^{.m}:=\mathbf{x}\odot\mathbf{x}^{.(m-1)}$ for $m>1$.}
	\end{threeparttable}
\end{table}

In our scheme, the stiff linear dissipation term $\nu\Delta\omega$ is approximated by a $s$-stage stiffly-accurate diagonally implicit Runge-Kutta methods with the coefficient matrix $A$, the abscissa vector $\mathbf{c}=A\mathbf{1}$ and the vector of weights $\mathbf{b}:=A^T\mathbf{e}_s$, while the nonlinear convective term $\myvec{u}\cdot\nabla \omega$ is approximated, in the skew symmetric form \eqref{discrete: skew symmetric form}, 
by the $s$-stage  explicit Runge-Kutta methods with the strictly lower triangular coefficient matrix $\widehat{A}$, the abscissa vector $\hat{\mathbf{c}}=\widehat{A}\mathbf{1}$ and the vector of weights $\hat{\mathbf{b}}:=\widehat{A}^T\mathbf{e}_s$, where $\mathbf{1}=(1,1,\cdots,1)_{s-1}^T$ and $\mathbf{e}_s=(0,0,\cdots,1)_{s-1}^T$. We also impose the canopy node condition, $\hat{\mathbf{c}}={\mathbf{c}}$ or $A\mathbf{1}=\widehat{A}\mathbf{1}$,
such that the resulting IERK procedure \eqref{scheme: general IERK-vorticity} 
for the vorticity equation \eqref{cont: INSE-vorticity} is consistent at all stages. 
Note that, the IERK procedure \eqref{scheme: general IERK-vorticity} is an ARS-type IERK \cite{FuTangYang:2024,IzzoJackiewicz:2017,LiuZou:2006} named after \cite{AscherRuuthSpiteri:1997}, or called Radau-type IERK \cite{ChenHuangLiaoYi:2026cicp,LiaoWangWen:2024IERK,LiaoTangWangZhou:2025R-IERK,WangZhaoLiao:2025} since the implicit part employs the diagonally implicit Runge-Kutta method of Radau-type. 

As mentioned earlier, another novel contribution is a set of simplified order conditions for the IERK procedure \eqref{scheme: general IERK-vorticity}, see Table \ref{table: order condition} and the detail derivations in Section \ref{appendix: order conditions} of the supplementary material. Compared with the classical order conditions \cite{AscherRuuthSpiteri:1997,IzzoJackiewicz:2017,LiaoWangWen:2024IERK}, the number of third-order conditions has been reduced from 6 to 5, and the number of fourth-order conditions has been significantly reduced from 18 to 11. 
Obviously, the simplified order conditions make it easier for us to construct a formal class of fourth-order IERK methods. The idea of our simplification is very simple: consider a special ordinary differential problem $y'=Ly+g(y,t)$ with a linear (abstract) operator $L$ but not the traditional model $y'=f(y,t)+g(y,t)$ in the derivation of order conditions. Practically, to avoid the inner iterations at each stage or maintain the whole computational efficiency, the implicit part of IERK method always contains only linear terms, such as the implicit part of IERK procedure \eqref{scheme: general IERK-vorticity}. 

The remainder of the paper is organized as follows. Next section describes the Fourier pseudo-spectral discretization and presents some technical lemmas, including the convolution-type H\"{o}lder inequality for the quadratic form and the damping-type  Gr\"{o}nwall inequality for multi-stage methods. The  unified discrete energy arguments of the internal stability  with a complete mathematical induction to the stage-wise boundedness of vorticity in the $H^\delta$ norm are addressed in Section \ref{sec: stability}, 
while some new parameterized IERK methods up to fourth-order time accuracy are presented in Section \ref{sec: IERK methods} of the supplementary material. 
Extensive numerical experiments confirming both the accuracy and long-time behaviors are reported in Section \ref{sec: experiments}, in which we also suggest a new adaptive time-stepping (ATS) algorithm, called  adaptive time-stepping with local delay and local backtrack (ATS-LDLB), to accurately resolve the small-scale chaotic or high-frequency quasi-periodic  behaviors and efficiently accelerate the large-scale low-frequency periodic motions.

\section{Some preliminaries and technical lemmas}
\setcounter{equation}{0}

\subsection{Fourier pseudo-spectral approximation}
Set the domain $\Omega=(0,L)^2$
and consider $h_x=h_y=h:=L/M$ in each direction
for an even positive integer $M$.
Let $\Omega_{h}:=\big\{\myvec{x}_{h}=(ih,jh)\,|\,1\le i,j \le M\big\}$
and put
$\bar{\Omega}_{h}:=\Omega_{h}\cup\partial{\Omega}$.
Denote the space of $L$-periodic grid functions
$\mathbb{V}_{h}:=\{v\,|\,v=\bra{v_h}\; \text{is $L$-periodic for}\; \myvec{x}_h\in\bar{\Omega}_h\}.$
For a periodic function $v(\myvec{x})$ on $\bar{\Omega}$,
let $P_M:L^2(\Omega)\rightarrow \mathscr{F}_M$
be the standard $L^2$ projection operator onto the space $\mathscr{F}_M$,
consisting of all trigonometric polynomials of degree up to $M/2$,
and $I_M:L^2(\Omega)\rightarrow \mathscr{F}_M$
be the trigonometric interpolation operator \cite{ShenTangWang:2011Spectral},
i.e.,
\[
\bra{P_Mv}(\myvec{x})=\sum_{\ell,m =- M/2}^{M/2-1}
\widehat{v}_{\ell,m}e_{\ell,m}(\myvec{x}),\quad
\bra{I_Mv}(\myvec{x})=\sum_{\ell,m=- M/2}^{M/2-1}
\widetilde{v}_{\ell,m}e_{\ell,m}(\myvec{x}),
\]
where the complex exponential basis function
$e_{\ell,m}(\myvec{x}):=e^{\mathrm{i}\nu\bra{\ell x+my}}$ with $\nu=2\pi/L$.
The coefficients $\widehat{v}_{\ell,m}$
refer to the standard Fourier coefficients of function $v(\myvec{x})$,
and the
pseudo-spectral coefficients $\widetilde{v}_{\ell,m}$ are determined such that $\bra{I_Mv}(\myvec{x}_h)=v_h$. 
The Fourier pseudo-spectral first and second order derivatives of $v_h$ are given by
\[
\mathcal{D}_xv_h:=\sum_{\ell,m= -M/2}^{M/2-1}
\bra{\mathrm{i}\nu\ell}\widetilde{v}_{\ell,m}
e_{\ell,m}(\myvec{x}_h),\quad
\mathcal{D}_x^2v_h:=\sum_{\ell,m = -M/2}^{M/2-1}
\bra{\mathrm{i}\nu\ell}^2\widetilde{v}_{\ell,m}
e_{\ell,m}(\myvec{x}_h).
\]
The operators $\mathcal{D}_y$
and $\mathcal{D}_y^2$ can be defined similarly.
We define the discrete gradient and Laplacian
 by $\nabla_hv_h := (\mathcal{D}_xv_h,\mathcal{D}_yv_h)^T$ and $\Delta_hv_h :=\brat{\mathcal{D}_x^2+\mathcal{D}_y^2}v_h,$ respectively.

For grid functions $v,w\in\mathbb{V}_{h}$,
define the inner product
$\myinner{v,w}:=h^2\sum_{\myvec{x}_h\in\Omega_{h}}v_h \bar{w}_h,$
and the associated $L^{2}$ norm $\mynorm{v}:=\sqrt{\myinner{v,v}}$.
Also, we will use  $\mynorm{v}_{\infty}=\max_{\myvec{x}_h\in\Omega_{h}}|v_h|$,  $$\mynormb{\nabla_hv}:=\sqrt{h^2\sum_{\myvec{x}_h\in\Omega_{h}}|\nabla_hv_h|^2}\quad
\text{and}\quad  \mynormb{\Delta_hv}:=\sqrt{h^2\sum_{\myvec{x}_h\in\Omega_{h}}|\Delta_hv_h|^2}.$$
One has the Green's formulas,
$\myinner{-\Delta_hv,w}=\myinner{\nabla_hv,\nabla_hw}$ and
$\myinnert{\Delta_h^2v,w}=\myinner{\Delta_hv,\Delta_hw}$,
see \cite{ShenTangWang:2011Spectral,ChengWangWiseYue:2016Weakly} for more details.
Here and hereafter, any subscripted $\ck$, such as $\ck_\psi$, $\ck_g$, $\ck_\Omega$, ${\ck}_{\Omega1}$, ${\ck}_{\Omega2}$, ${\ck}_{\Omega3}$, ${\ck}_{\Omega,\delta}$, $\ck_{p1}$, $\ck_{p2}$, ${\ck}_{p,\delta}$ and so on,
denotes a fixed constant; while any subscripted $\Ck$, such as $\Ck_\omega$ and $\Ck_\psi$, denotes a generic positive constant, not necessarily
the same at different occurrences. The appeared constants may be dependent on the given data
and the solutions $\omega$ and $\psi$ but are always independent of the spatial length $h$ and the time-step size $\tau_n$.

\subsection{Estimates in Fourier collocation spectral space}	
It is recognized that aliasing errors in the nonlinear term impose a significant challenge for the numerical analysis of Fourier pseudospectral scheme. To address this issue, we introduce a continuous extension of a grid function and a Fourier collocation interpolation operator.

Let $v_h\in\mathbb{V}_{h}$ be a periodic grid function defined on the 2D mesh $\Omega_{h}$.  Its \emph{continuous extension} (or Fourier collocation interpolant) is the $\mathscr{F}_{M}$-trigonometric polynomial
\begin{align}\label{continuous extension formula}
	v_{\mathrm{c}}(\myvec{x}):=(I_{M}v_h)(\myvec{x})=\sum_{\ell,m=-M/2}^{M/2-1}\widetilde{v}_{\ell,m}\,e_{\ell,m}(\myvec{x}),\quad \myvec{x}\in\Omega,
\end{align}
where $I_M$ denotes the trigonometric interpolation operator and the \emph{pseudo-spectral coefficients} $\widetilde{v}_{\ell,m}$ are determined by requiring exactness on the grid:
\begin{align*}
	\widetilde{v}_{\ell,m}=\frac{1}{M^{2}}\sum_{i,j=0}^{M-1}v_h(x_i,y_j)\,e_{-\ell,-m}(x_i,y_j),\qquad
	|\ell|,|m|\le M/2.
\end{align*}
Consequently $v_{\mathrm{c}}(\myvec{x}_h)=v_h(\myvec{x}_h)$ at every mesh point $\boldsymbol{x}_h\in\Omega_h$.
For any periodic continuous function $v$ (possibly containing wavelengths beyond the mesh resolution), the same operator $I_M$ produces its collocation interpolant
\begin{align*}
	(I_M v)(\myvec{x})=\sum_{\ell,m=-M/2}^{M/2-1}\widetilde{v}_{\ell,m}\,e_{\ell,m}(\myvec{x}),
\end{align*}
with coefficients $\widetilde{v}_{\ell,m}$ given by the discrete formula above with $v_h(x_i,y_j)$ replaced by $v(x_i,y_j)$.
Note that $\widetilde{v}_{\ell,m}$ may not be the true Fourier coefficients of $v$, due to the truncation and aliasing errors.

We also use  the standard seminorms and norms in the Sobolev space $H^{m}(\Omega)$ for $m\ge0$.
Let $C^{\infty}(\Omega)$ be a set of infinitely differentiable $L$-periodic functions defined on $\Omega$,
and $H_{per}^{m}(\Omega)$ be the closure of $C^{\infty}(\Omega)$ in $H^{m}(\Omega)$,
endowed with the semi-norm $|\cdot|_{H_{per}^m}$ and the norm $\mynorm{\cdot}_{H_{per}^{m}}$. 
For simplicity, denote further $|\cdot|_{H^m}:=|\cdot|_{H_{per}^m}$, $\mynorm{\cdot}_{H^{m}}:=\mynorm{\cdot}_{H_{per}^{m}}$ and $\mynorm{\cdot}_{L^{2}}:=\mynorm{\cdot}_{H^{0}}$.

\begin{lemma} \cite{ChengWang:2016} \label{lemma: collocation interpolation}
		For any $v\in\mathcal{P}_{2N}$ in the $d$ dimension, it holds that
	$$\mynormb{I_Mv}_{H^k}\le 2^{d/2}\mynormb{v}_{H^k}\quad\text{for $k\in\mathbb{Z}$, $k\ge0$.}$$
\end{lemma}

This lemma devotes to overcome a key difficulty associated with the $H^k$ norm bound of the nonlinear term
obtained by collocation interpolation. As is seen later, Lemma \ref{lemma: collocation interpolation}, combined with the H\"{o}lder's inequality, the Sobolev embedding and the elliptic regularity, will be used in the proofs of Lemma \ref{lemma: nonlinear convection bound} and  Lemma \ref{lemma: nonlinear convection H2 bound}.

Note that, the numerical velocity from \eqref{scheme: general IERK-velocity} is divergence-free at each stage,
\begin{align}\label{discrete: divergence-free}
	\nabla_h\cdot\myvec{u}^{n,i+1}=&\,\mathcal{D}_x\mathcal{D}_y\psi^{n,i+1}-\mathcal{D}_y\mathcal{D}_x\psi^{n,i+1}=0\quad\text{for $1\le i\le s_{\mathrm{I}}$.}
\end{align} 
The skew symmetric form \eqref{discrete: skew symmetric form} of the nonlinear convection term is well-known for its ability to overcome the difficulty associated with the aliasing errors appearing in
pseudospectral approximations. Actually, it makes the nonlinear term orthogonal to
the vorticity field in the $L^2$ space,
\begin{align}\label{discrete: skew symmetric orthogonality}
	2\myinnerb{\mathtt{c}_h(\myvec{u},\omega),\omega}=\myinnerb{\myvec{u}\cdot\nabla_h \omega,\omega}-\myinnerb{\myvec{u}\omega,\nabla_h \omega}=0.
\end{align}
In addition, we denote $\myvec{u}_c^{n,j}$, $\myvec{\omega}_c^{n,j}$, and $\myvec{\psi}_c^{n,j}$
as the continuous projection
of $\myvec{u}^{n,j}$, $\omega^{n,j}$, and $\psi^{n,j}$
, respectively, with the projection formula given by \eqref{continuous extension formula}. It is
clear that $\myvec{u}_c^{n,j}, \myvec{\omega}_c^{n,j}, \myvec{\psi}_c^{n,j}\in \mathscr{F}_M$ and the kinematic equation $-\Delta \myvec{\psi}_c^{n,j}= \myvec{\omega}_c^{n,j}, \myvec{u}_c^{n,j}=(\partial_y\myvec{\psi}_c^{n,j},-\partial_x\myvec{\psi}_c^{n,j})$
is
satisfied at the continuous level. Because of these kinematic equations, an application
of elliptic regularity shows that
\begin{align}\label{elliptic regularity}
	&\|\psi^{n,j}\|_{H^{m+2}_h} \le \ck_\psi \|\omega^{n,j}\|_{H^{m}_h}, \quad
	\|\psi^{n,j}\|_{H^{m+2+\alpha}_h} \leq \ck_\psi \|\omega^{n,j}\|_{H^{m+\alpha}_h},\\ \notag
	&\|\myvec{\psi}_c^{n,j}\|_{H^{m+2}} \leq \ck_\psi \|\myvec{\omega}_c^{n,j}\|_{H^{m}}, \quad
	\|\myvec{\psi}_c^{n,j}||_{H^{m+2+\alpha}} \leq \ck_\psi \|\myvec{\omega}_c^{n,j}||_{H^{m+\alpha}},
\end{align}
in which we normalize the stream function with $\int_{\Omega}\psi^{n,j}d\myvec{x}=0$. Meanwhile, since all the profiles have mean zero over the domain,
\begin{align*}
	&\overline{\psi^{n,j}}=0,\quad\overline{\myvec{u}^{n,j}}=\brat{\overline{\mathcal{D}_y\psi^{n,j}},-\overline{\mathcal{D}_x\psi^{n,j}}}=0,\quad \overline{\omega^{n,j}}=\overline{\Delta_h \psi^{n,j}}=0,\\
	&\overline{\myvec{\psi}_c^{n,j}}=0,\quad\overline{\myvec{u}^{n,j}_c}=\brat{\overline{\partial_y\myvec{\psi}_c^{n,j}},-\overline{\partial_x\myvec{\psi}_c^{n,j}}}=0,\quad \overline{\myvec{\omega}_c^{n,j}}=\overline{\Delta \myvec{\psi}_c^{n,j}}=0.
\end{align*}
It is clear that all the Poincar\'{e} inequality and elliptic regularity can be applied.

\subsection{Technical lemmas for IERK methods}

We define the lower triangular coefficient matrices for the implicit and explicit parts,
$$A_{\mathrm{I}}:=\brab{a_{i+1,j+1}}_{i,j=1}^{s_{\mathrm{I}}}\quad\text{and}\quad A_{\mathrm{E}}:=\brab{\hat{a}_{i+1,j}}_{i,j=1}^{s_{\mathrm{I}}}\,.$$
Also, let $E_{s_{\mathrm{I}}}:=(1_{i\ge j})_{s_{\mathrm{I}}\times s_{\mathrm{I}}}$ be the lower triangular matrix full of element 1, and let $I_{s_{\mathrm{I}}}$ be the  identity matrix of the same size as $A_{\mathrm{I}}$. 
In the subsequent analysis, we always define the stage difference operator  $\delta_{\tau}w^{n,i+1}:=w^{n,i+1}-w^{n,i}$ and use the following stage difference form of \eqref{scheme: general IERK-vorticity}, 
\begin{align}\label{scheme: general IERK-vorticity-differential}    
	\delta_{\tau}\omega_h^{n,i+1}+\tau_n\sum_{j=1}^{i}\underline{\hat{a}}_{i+1,j}
	\mathtt{c}_h(\myvec{u}_h^{n,j},\omega_h^{n,j})
	=\nu\tau_n\sum_{j=1}^{i}\underline{a}_{i+1,j+1}\brab{\Delta_h\omega_h^{n,j+1}+g^{n,j+1}}
\end{align}
for $1\le i\le s_{\mathrm{I}}$, where the difference coefficients $\underline{a}_{i+1,j+1}$ and $\underline{\hat{a}}_{i+1,j}$ can be defined by
\begin{align}\label{scheme: IERK-coefficient matrix}    
	(\underline{a}_{i+1,j+1})_{s_{\mathrm{I}}\times s_{\mathrm{I}}}:=E_{s_{\mathrm{I}}}^{-1}A_{\mathrm{I}}\quad\text{and}\quad 
	(\underline{\hat{a}}_{i+1,j})_{s_{\mathrm{I}}\times s_{\mathrm{I}}}:=E_{s_{\mathrm{I}}}^{-1}A_{\mathrm{E}}.
\end{align}

For the two difference coefficient matrices defined by \eqref{scheme: IERK-coefficient matrix}, we have the following result, where $\mathcal{S}(A):=(A+A^T)/2$ for a given matrix $A$.

\begin{lemma} \label{lemma: bound quadratic form}	
	Assume that $\mathcal{S}(E_{s_{\mathrm{I}}}^{-1}A_{\mathrm{I}})$ has a finite minimum eigenvalue $\lambda_{\mathrm{I}}>0$, while the spectral
	norms of $E_{s_{\mathrm{I}}}^{-1}A_{\mathrm{I}}$ and $E_{s_{\mathrm{I}}}^{-1}A_{\mathrm{E}}$ are bounded by  positive constants $\sigma_{\mathrm{I}}$ and $\sigma_{\mathrm{E}}$, respectively. For any time-space sequences $\{v_h^{i},u_h^i \mid i \geq 1\}$, it holds that, 
	\begin{align*}
		(i)\;\;&\,\sum_{i=1}^{k} \sum_{j=1}^{i} \underline{a}_{i+1,j+1} \myinnerb{v^{j}, v^{i}} \geq \lambda_{\mathrm{I}}\sum_{i=1}^{k}\mynormb{v^{i}}^2\quad\text{for $1\le k\le s_{\mathrm{I}}$,}\\
		(ii)\;\;&\,\sum_{i=1}^{k} \sum_{j=1}^{i}\underline{a}_{i+1,j+1} \myinnerb{v^{j}, u^{i}} \leq
		\sigma_{\mathrm{I}}\braB{\sum_{i=1}^k\mynormb{v^{i}}^2}^{\frac12}
		\braB{\sum_{i=1}^k\mynormb{u^{i}}^2}^{\frac12} \quad\text{for $1\le k\le s_{\mathrm{I}}$,}\\
		(iii)\;\;&\,\sum_{i=1}^{k} \sum_{j=1}^{i}\underline{\hat{a}}_{i+1,j} \myinnerb{v^{j}, u^{i}} \leq
		k^{\frac{1}{2}-\frac{1}{\mu}}\sigma_{\mathrm{E}}\braB{\sum_{i=1}^k\mynormb{v^{i}}^\alpha}^{\frac1{\alpha}}
		\braB{\sum_{i=1}^k\mynormb{u^{i}}^\beta}^{\frac1{\beta}} \quad\text{for $1\le k\le s_{\mathrm{I}}$,}
	\end{align*}
	where $\mu:=\max\{\alpha,\beta\}$ for $\alpha, \beta \geq 1$ and $\frac1\alpha+\frac1\beta=1$.
\end{lemma}
\begin{proof} By the definitions in \eqref{scheme: IERK-coefficient matrix} of the matrix $E_{s_{\mathrm{I}}}^{-1}A_{\mathrm{I}}$, the first inequality is obvious according to the Cauchy's interlacing theorem for a bordered Hermitian matrix, see \cite[Theorem 4.3.17]{HornJohnson:2013book}. As for the third result, we denote $w^i_h=\sum_{j=1}^i\underline{\hat{a}}_{i+1,j} v^i_h $ and apply the  Cauchy-Schwarz inequality to get
	\begin{align}\label{lemmaProof-bound quadratic form equation}
		& \,\sum_{i=1}^{k} \sum_{j=1}^{i}\underline{\hat{a}}_{i+1,j} \myinnerb{v^{j}, u^{i}}\leq	 \myinnerB{\braB{\sum_{i=1}^{k}(w^i)^2}^{\frac12},\braB{\sum_{i=1}^{k}(u^i)^2}^{\frac12}}   \\
		& \,\hspace{2cm}\le  \sigma_{\mathrm{E}} \myinnerB{\braB{\sum_{i=1}^{k}(v^i)^2}^{\frac12},\braB{\sum_{i=1}^{k}(u^i)^2}^{\frac12}}\notag\\ 
		&\,\hspace{2cm}\leq  \sigma_{\mathrm{E}}\braB{\sum_{i=1}^{k}\mynorm{v^i}^2}^{\frac12}\braB{\sum_{i=1}^{k}\mynorm{u^i}^2}^{\frac12}
		= \sigma_{\mathrm{E}}\braB{\sum_{i=1}^{k}\tilde{v}_i^2}^{\frac12}\braB{\sum_{i=1}^{k}\tilde{u}_i^2}^{\frac12},\notag
	\end{align} %
	where $\tilde{v}_i:=\mynormt{v^i}$ and $\tilde{u}_i:=\mynormt{u^i}$. According to the equivalent relationship \cite[(5.4.21)]{HornJohnson:2013book} of the $\ell_{p}$-norms of vector, $$\braB{\sum_{i=1}^{k}\tilde{v}_i^{p_2}}^{\frac{1}{p_2}}\leq \braB{\sum_{i=1}^{k}\tilde{v}_i^{p_1}}^{\frac{1}{p_1}}\leq k^{\frac{1}{p_1}-\frac{1}{p_2}} \braB{\sum_{i=1}^{k}\tilde{v}_i^{p_2}}^{\frac{1}{p_2}}\quad\text{ for $1\leq p_1 \leq p_2<\infty$,}$$ one can get the inequality (ii). Actually, for any $\beta\geq2$ with $\frac1\alpha+\frac1\beta=1$, one finds that
	\begin{align*}
		\braB{\sum_{i=1}^{k}\tilde{v}_i^{2}}^{\frac12}\leq \braB{\sum_{i=1}^{k}\tilde{v}_i^{\alpha}}^{\frac1\alpha} \quad\text{and}\quad \braB{\sum_{i=1}^{k}\tilde{u}_i^2}^{\frac12}\leq k^{\frac{1}{2}-\frac{1}{\beta}} \braB{\sum_{i=1}^{k}\tilde{u}_i^{\beta}}^{\frac1\beta}.
	\end{align*}
	Thus, the right hand side of \eqref{lemmaProof-bound quadratic form equation} can be bounded by
	\begin{align*}
		\braB{\sum_{i=1}^{k}\tilde{v}_i^2}^{\frac12} \braB{\sum_{i=1}^{k}\tilde{u}_i^2}^{\frac12}\leq k^{\frac{1}{2}-\frac{1}{\beta}} \braB{\sum_{i=1}^{k}\tilde{v}_i^{\alpha}}^{\frac1\alpha}\braB{\sum_{i=1}^{k}\tilde{u}_i^{\beta}}^{\frac1\beta} \quad\text{for $\beta\geq2$ and $\frac1\alpha+\frac1\beta=1$}.
	\end{align*}
	Similarly, one has
	\begin{align*}
		\braB{\sum_{i=1}^{k}\tilde{v}_i^2}^{\frac12} \braB{\sum_{i=1}^{k}\tilde{u}_i^2}^{\frac12}\leq k^{\frac{1}{2}-\frac{1}{\alpha}} \braB{\sum_{i=1}^{k}\tilde{v}_i^{\alpha}}^{\frac1\alpha}\braB{\sum_{i=1}^{k}\tilde{u}_i^{\beta}}^{\frac1\beta} \quad\text{for $\alpha\geq2$  and $\frac1\alpha+\frac1\beta=1$}.
	\end{align*}
	The second result can be similarly derived directly from \eqref{lemmaProof-bound quadratic form equation}. It completes the proof.
\end{proof}

The convolution-type H\"{o}lder inequality in Lemma \ref{lemma: bound quadratic form} (iii) and
the following discrete Gr\"{o}nwall inequality are useful in the long-time stability analysis.

\begin{lemma} \label{lemma: decaying gronwall multi-stage RK}
	For any integers $s>1$ and $N>1$, let $c_g>0$ be a constant and $v_{n,i}$ be the non-negative discrete function defined at the stage $t_{n,i}$ for the time level indexes $n=1,2,\cdots,N$  and the stage indexes $i=1,2,\cdots,s$ with a finite $T=t_{N}=t_{N,s}$. Also, let $v_n:=v_{n,s}=v_{n+1,1}$. There exists a positive constant $\xi$, independent of the time-step sizes $\tau_n$, such that
	\begin{equation}\label{lemProof: decaying gronwall assumption}
		(1+2\xi\tau_n)v_{n,i}\le v_{n,1}
		+c_g\tau_n\quad\text{for $n\ge1$ and $2\le i\le s$.}
	\end{equation} 
	If the maximum step-size $\tau_{\max}$ is small such that $\tau_{\max}\le 1/\xi$, it holds  that
	\begin{align*}
		v_{m,i}	\le e^{-\xi t_{m}}v_{0}+c_g/{\xi}\quad \text{for $m\ge1$ and $2\le i\le s$.}
	\end{align*}	
\end{lemma}

\begin{proof} For $\tau_{\max}\le 1/\xi$, one has $1+2\xi\tau_n\ge e^{\xi\tau_n}$. Thus the priori setting \eqref{lemProof: decaying gronwall assumption} gives
	\begin{equation}\label{lemProof: vn stage-level}
		v_{n,i}\le e^{-\xi\tau_n}v_{n,1}+e^{-\xi\tau_n}c_g\tau_n
	\end{equation} 
	One can take $i:=s$ and $n:=k$ in \eqref{lemProof: vn stage-level} to get
	\begin{equation*}
		v_{k}\le e^{-\xi\tau_k}v_{k-1}+e^{-\xi\tau_k}c_g\tau_k\quad\text{for $k\ge1$,}
	\end{equation*} 
	which yields
	\begin{align*}
		v_{k}\le&\, e^{-\xi\tau_k}v_{k-1}+e^{-\xi\tau_k}c_g\tau_k
		\le e^{-\xi\tau_k}(e^{-\xi\tau_{k-1}}v_{k-2}
		+e^{-\xi\tau_{k-1}}c_g\tau_{k-1})+e^{-\xi\tau_k}c_g\tau_k\\
		\le&\,\cdots
		\le e^{-\xi t_k}v_{0}+c_g\sum_{j=1}^{k}\tau_je^{-\xi(t_k-t_{j-1})}.
	\end{align*}
	Then we derive from \eqref{lemProof: vn stage-level} that 
	\begin{align*}
		v_{m,i}\le&\, e^{-\xi t_{m}}v_{0}+c_g\sum_{j=1}^{m}\tau_je^{-\xi(t_{m}-t_{j-1})}
		=e^{-\xi t_{m}}v_{0}+c_ge^{-\xi t_{m}}\sum_{j=1}^{m}\tau_je^{\xi t_{j-1}}\\
		\le&\, e^{-\xi t_{m}}v_{0}+c_ge^{-\xi t_{m}}\int_0^{t_m}e^{\xi t}\zd t
		=e^{-\xi t_{m}}v_{0}+\frac{c_g}{\xi}(1-e^{-\xi t_{m}}),
	\end{align*}
	which leads to the claimed  stage estimate and completes the proof.
\end{proof}


\section{Long time stability of IERK methods}\label{sec: stability}
\setcounter{equation}{0}

The long-time stability of the discrete system \eqref{scheme: general IERK-vorticity}-\eqref{scheme: general IERK-velocity} will be performed by a complete mathematical induction to the stage-wise boundedness of vorticity in the $H^{\delta}$ norm for some $\delta\in(0,\frac12]$, 
\begin{align}\label{estimate: H-delta-bound}
	\mynormb{\myvec{\omega}_c^{n,j}}_{H^{\delta}}\le \cm_0\quad\text{for $n\ge1$ and $1\le j\le s$,}
\end{align} 
where $\myvec{\omega}_c^{n,j}$ is the continuous extension of $\omega_h^{n,j}$ and the global-in-time constant $\cm_0>0$ will be determined later (we use the subscripted $\cm$, including $\cm_0$, $\cm_1$ and $\cm_2$, to
denote a fixed constant in the process of mathematical induction). Obviously, it holds for the starting case $n=1$ and $j=1$ since $\omega_h^{1,1}=\omega_h^{0}$. We put the following induction hypothesis
\begin{align}\label{estimate: H-delta-induction hypothesis}
	\mynormb{\myvec{\omega}_c^{n,j}}_{H^{\delta}}\le \cm_0\quad\text{for $1\le n\le m$ and $1\le j\le l$.}
\end{align}
It remains to verify $\mynormb{\myvec{\omega}_c^{m,l+1}}_{H^{\delta}}\le \cm_0$ by a three-step procedure:
\begin{itemize}
	\item \textbf{Step 1}: 	Based on the induction hypothesis \eqref{estimate: H-delta-induction hypothesis}, this step establishes the global-in-time boundedness $\cm_1$ of stage solutions $\omega_h^{m,l+1}$ in the $L^2$ norm by testing the stage difference form \eqref{scheme: general IERK-vorticity-differential}  with $2\omega_h^{n,i+1}$ with the help of the Sobolev embedding and the H\"{o}lder's inequalities;
	\item \textbf{Step 2}: With the help of the boundedness of $\mynormb{\omega^{m,l+1}}$ in Step 1, this step establishes the global-in-time boundedness $\cm_2$ of stage solutions $\omega_h^{m,l+1}$ in the $H^1$ semi-norm  by testing the stage difference form \eqref{scheme: general IERK-vorticity-differential}  with $-2\Delta_h\omega_h^{n,i+1}$;
	\item \textbf{Step 3}: Recover the stage boundedness  \eqref{estimate: H-delta-bound} in the $H^{\delta}$ norm by using the $L^2$ norm estimate in Step 1 and the $H^1$ semi-norm estimate in Step 2. 
\end{itemize}

\subsection{$L^2$ norm estimate}

Next lemma is helpful to the $L^2$ norm estimate of vorticity, 
see the proof in Section \ref{appendix: lemma3.1} of the supplementary material.

\begin{lemma} \label{lemma: nonlinear convection bound}
	Assume that the spectral norm of $E_{s_{\mathrm{I}}}^{-1}A_{\mathrm{E}}$ is bounded by a constant $\sigma_{\mathrm{E}}>0$. If  $\mynormb{\myvec{\omega}_c^{n,j}}_{H^{\delta}}\le \cm_0$ for $1\le n\le m$ and $1\le j\le l$,
	then there exists a positive constant $\ck_{\Omega}$ such that
	\begin{align*} 
		&\, 2\tau_n\sum_{i=1}^k\sum_{j=1}^{i}
		\underline{\hat{a}}_{i+1,j}\myinnerb{\mathtt{c}_h(\myvec{u}^{n,j},\omega^{n,j}),\omega^{n,i+1}}
		\le\frac{s^2\ck_{\Omega}^2\cm_0^2\sigma_{\mathrm{E}}^2}{8\epsilon}
		\tau_n\sum_{i=1}^k\mynormb{\delta_{\tau}\omega^{n,i+1}}^2\\
		&\,\hspace{1.8cm}+\epsilon\tau_n\sum_{i=1}^k\mynormb{\nabla_h\omega^{n,i+1}}^2\quad\text{for any $\epsilon>0$, $1\le k\le l\le s_{\mathrm{I}}$ and $1\le n\le m$.}
	\end{align*}		
\end{lemma}

Now testing the stage difference form \eqref{scheme: general IERK-vorticity-differential}  with $2\omega_h^{n,i+1}$
and summing $i$ from $i=1$ to $k$, one has 
\begin{align}\label{stability: IERK-L2 product}
	\mynormb{\omega^{n,k+1}}^2&\,-\mynormb{\omega^{n,1}}^2+\sum_{i=1}^k\mynormb{\delta_{\tau}\omega^{n,i+1}}^2\\
	&\,+2\nu\tau_n\sum_{i=1}^k\sum_{j=1}^{i}
	\underline{a}_{i+1,j+1}\myinnerb{\nabla_h\omega^{n,j+1},\nabla_h\omega^{n,i+1}}\notag\\
	&\,=-2\tau_n\sum_{i=1}^k\sum_{j=1}^{i}
	\underline{\hat{a}}_{i+1,j}\myinnerb{\mathtt{c}_h(\myvec{u}^{n,j},\omega^{n,j}),\omega^{n,i+1}}\notag\\
	&\,+2\tau_n\sum_{i=1}^k\sum_{j=1}^{i}
	\underline{a}_{i+1,j+1}\myinnerb{g^{n,j+1},\omega^{n,i+1}}.\notag
\end{align}
With the help of Lemma \ref{lemma: bound quadratic form} (i), the last term at the left hand side of \eqref{stability: IERK-L2 product} can be bounded by
\begin{align*} 
	J_{L1}^{n,k+1}:=\,2\nu\tau_n\sum_{i=1}^k\sum_{j=1}^{i}
	\underline{a}_{i+1,j+1}\myinnerb{\nabla_h\omega^{n,j+1},\nabla_h\omega^{n,i+1}}\ge\,2\lambda_{\mathrm{I}}\nu\tau_n\sum_{i=1}^k\mynormb{\nabla_h\omega^{n,i+1}}^2.
\end{align*}
Applying Lemma \ref{lemma: nonlinear convection bound} with $\epsilon:=\epsilon_1\lambda_{\mathrm{I}}\nu$, 
the first term at the right hand side of \eqref{stability: IERK-L2 product} is bounded by
\begin{align*} 
	J_{L2}^{n,k+1}
	\le&\,\frac{s^2\ck_{\Omega}^2\cm_0^2\sigma_{\mathrm{E}}^2}{8\lambda_{\mathrm{I}}\epsilon_1\nu}
	\tau_n\sum_{i=1}^k\mynormb{\delta_{\tau}\omega^{n,i+1}}^2
	+\epsilon_1\lambda_{\mathrm{I}}\nu\tau_n\sum_{i=1}^k\mynormb{\nabla_h\omega^{n,i+1}}^2.
\end{align*}
Applying the Poincar\'{e} inequality, $\mynormb{\omega^{n,i+1}}\le \ck_{p1}\mynormb{\nabla_h\omega^{n,i+1}}$, 
the second term at the right hand side of \eqref{stability: IERK-L2 product} can be bounded by
\begin{align*} 
	J_{L3}^{n,k+1}
	\leq&\,2\sigma_{\mathrm{I}}\ck_{p1}\tau_n\sqrt{\sum_{i=1}^k\mynormb{g^{n,i+1}}^2}
	\sqrt{\sum_{i=1}^k\mynormb{\nabla_h\omega^{n,i+1}}^2} \\
	\leq&\, \frac{\sigma_{\mathrm{I}}^2\ck_{p1}^2}{\lambda_{\mathrm{I}}\epsilon_2\nu}\tau_n\sum_{i=1}^k\mynormb{g^{n,i+1}}^2
	+\epsilon_2\lambda_{\mathrm{I}}\nu\tau_n\sum_{i=1}^k\mynormb{\nabla_h\omega^{n,i+1}}^2.
\end{align*}
By collecting the above estimates on $J_{Lj}^{n,k+1}$ $(j=1,2,3)$, it follows from \eqref{stability: IERK-L2 product} that
\begin{align}\label{stability: IERK-L2 product2}
	\mynormb{\omega^{n,k+1}}^2-&\,\mynormb{\omega^{n,1}}^2+\sum_{i=1}^k\mynormb{\delta_{\tau}\omega^{n,i+1}}^2
	+(2-\epsilon_1-\epsilon_2)\lambda_{\mathrm{I}}\nu\tau_n\sum_{i=1}^k\mynormb{\nabla_h\omega^{n,i+1}}^2
	\\
	\le&\,\frac{s^2\ck_{\Omega}^2\cm_0^2\sigma_{\mathrm{E}}^2}{8\lambda_{\mathrm{I}}\epsilon_1\nu}
	\tau_n\sum_{i=1}^k\mynormb{\delta_{\tau}\omega^{n,i+1}}^2
	+\frac{\sigma_{\mathrm{I}}^2\ck_{p1}^2}{\lambda_{\mathrm{I}}\epsilon_2\nu}\tau_n\sum_{i=1}^k\mynormb{g^{n,i+1}}^2.
	\notag
\end{align}
By taking $\epsilon_1=\epsilon_2=1/8$ and assuming the maximum step size 
\begin{align}\label{cond: maximum time-stepL2}
	\tau_{\max}\le \frac{\lambda_{\mathrm{I}}\nu}{s^2\ck_{\Omega}^2\cm_0^2\sigma_{\mathrm{E}}^2},
\end{align}
one can apply the priori setting $\mynormb{g^{n,i}}\le \ck_g$ to get
\begin{align}\label{stability: IERK-L2 product3}
	\mynormb{\omega^{n,k+1}}^2+\,\tfrac{7}{4}\lambda_{\mathrm{I}}\nu\tau_n\sum_{i=1}^k\mynormb{\nabla_h\omega^{n,i+1}}^2
	\le\,\mynormb{\omega^{n,1}}^2	
	+\frac{8s\sigma_{\mathrm{I}}^2\ck_{p1}^2\ck_g^2}{\lambda_{\mathrm{I}}\nu}\tau_n.
\end{align}
With the help of  the Poincar\'{e} inequality, 
it follows from \eqref{stability: IERK-L2 product3} that
\begin{align*}
	\mynormb{\omega^{n,k+1}}^2+\tfrac{7\lambda_{\mathrm{I}}\nu\tau_n}{4\ck_{p1}}\mynormb{\omega^{n,k+1}}^2
	\le\,\mynormb{\omega^{n,1}}^2+\frac{8s\sigma_{\mathrm{I}}^2\ck_{p1}^2\ck_g^2}{\lambda_{\mathrm{I}}\nu}\tau_n.
\end{align*}
The damping-type discrete Gr\"{o}nwall inequality in Lemma \ref{lemma: decaying gronwall multi-stage RK} 
(with $\xi:=\frac{7\lambda_{\mathrm{I}}\nu}{8\ck_{p1}}$) yields
\begin{align}\label{stability: IERK-L2 product4}
	\mynormb{\omega^{m,l+1}}^2
	\le e^{-\tfrac{7\lambda_{\mathrm{I}}\nu}{8\ck_{p1}}t_m}
	\mynormb{\omega^{0}}^2
	+\frac{64s\sigma_{\mathrm{I}}^2\ck_{p1}^3\ck_g^2}{7\lambda_{\mathrm{I}}^2\nu^2}:=\cm_1,
\end{align}
where $\cm_{1}$ is independent of the priori  bound $\cm_0$.

\subsection{$H^1$ norm estimate}

Next lemma handles the nonlinear convection term in the $H^1$ semi-norm estimate of vorticity, 
see the proof in Section \ref{appendix: lemma3.2} of the supplementary material.

\begin{lemma} \label{lemma: nonlinear convection H2 bound}
	Assume that the spectral norm of $E_{s_{\mathrm{I}}}^{-1}A_{\mathrm{E}}$ is bounded by a constant $\sigma_{\mathrm{E}}>0$. If $\mynormb{\myvec{\omega}_c^{m,l+1}}_{2}\le \cm_1$ and $\mynormb{\myvec{\omega}_c^{n,j}}_{H^{\delta}}\le \cm_0$ for $1\le n\le m$ and $1\le j\le l$,
	then there exists positive constants $\hat{\ck}_{\Omega}(\cm_0)$ and $\bar{\ck}_{\Omega}$ (independent of $\cm_0$) such that
	\begin{align*} 
		&\, 2\tau_n\sum_{i=1}^k\sum_{j=1}^{i}
		\underline{\hat{a}}_{i+1,j}\myinnerb{\mathtt{c}_h(\myvec{u}^{n,j},\omega^{n,j}),\Delta_h\omega^{n,i+1}}
		\le\frac{\hat{\ck}_{\Omega}(\cm_0)s^2\sigma_{\mathrm{E}}^2\tau_n}{16\epsilon}
		\sum_{i=1}^k\mynormb{\delta_{\tau}\nabla\omega^{n,i+1}}^2\\
		&\,\hspace{0.5cm}+\epsilon\tau_n\sum_{i=1}^k\mynormb{\Delta_h\omega^{n,i+1}}^2+\frac{\bar{\ck}_{\Omega}s^2\sigma_{\mathrm{E}}^2\tau_n}{16\epsilon}\quad\text{for any $\epsilon>0$, $1\le k\le l$ and $1\le n\le m$.}
	\end{align*}		
\end{lemma}

Now testing the stage difference form \eqref{scheme: general IERK-vorticity-differential}  with $-2\Delta_h\omega_h^{n,i+1}$
and summing $i$ from $i=1$ to $k$, one has 
\begin{align}\label{stability: IERK-H1 product}
	\mynormb{\nabla_h\omega^{n,k+1}}^2&\,-\mynormb{\nabla_h\omega^{n,1}}^2+\sum_{i=1}^k\mynormb{\delta_{\tau}\nabla_h\omega^{n,i+1}}^2\\
	&\,+2\nu\tau_n\sum_{i=1}^k\sum_{j=0}^{i}
	\underline{a}_{i+1,j+1}\myinnerb{\Delta_h\omega^{n,j+1},\Delta_h\omega^{n,i+1}}\notag\\
	&\,=2\tau_n\sum_{i=1}^k\sum_{j=1}^{i}
	\underline{\hat{a}}_{i+1,j}\myinnerb{\mathtt{c}_h(\myvec{u}^{n,j},\omega^{n,j}),\Delta_h\omega^{n,i+1}}\notag\\
	&\,-2\tau_n\sum_{i=1}^k\sum_{j=1}^{i}
	\underline{a}_{i+1,j+1}\myinnerb{g^{n,j+1},\Delta_h\omega^{n,i+1}}.\notag
\end{align}
With the help of Lemma \ref{lemma: bound quadratic form} (i), the last term at the left hand side of \eqref{stability: IERK-H1 product} can be bounded by
\begin{align*} 
	J_{H1}^{n,k+1}:=\,2\nu\tau_n\sum_{i=1}^k\sum_{j=1}^{i}
	\underline{a}_{i+1,j+1}\myinnerb{\Delta_h\omega^{n,j+1},\Delta_h\omega^{n,i+1}}
	\ge\,2\lambda_{\mathrm{I}}\nu\tau_n\sum_{i=1}^k\mynormb{\Delta_h\omega^{n,i+1}}^2.
\end{align*}
Applying Lemma \ref{lemma: nonlinear convection bound} with $\epsilon:=\epsilon_3\lambda_{\mathrm{I}}\nu$, 
the first term at the right hand side of \eqref{stability: IERK-H1 product} is bounded by
\begin{align*} 
	J_{H2}^{n,k+1}
	\le&\,\frac{\hat{\ck}_{\Omega}(\cm_0)s^2\sigma_{\mathrm{E}}^2\tau_n}{16\epsilon_3\lambda_{\mathrm{I}}\nu}
	\sum_{i=1}^k\mynormb{\delta_{\tau}\nabla\omega^{n,i+1}}^2
	+\epsilon_3\lambda_{\mathrm{I}}\nu\tau_n\sum_{i=1}^k\mynormb{\Delta_h\omega^{n,i+1}}^2
	+\frac{\bar{\ck}_{\Omega}s^2\sigma_{\mathrm{E}}^2\tau_n}{16\epsilon_3\lambda_{\mathrm{I}}\nu}
\end{align*}
Applying Lemma \ref{lemma: bound quadratic form} (ii),
the second term at the right hand side of \eqref{stability: IERK-H1 product} can be bounded by
\begin{align*} 
	J_{H3}^{n,k+1}
	\leq&\,2\sigma_{\mathrm{I}}\tau_n\sqrt{\sum_{i=1}^k\mynormb{g^{n,i+1}}^2}
	\sqrt{\sum_{i=1}^k\mynormb{\Delta_h\omega^{n,i+1}}^2} \\
	\leq&\, \frac{\sigma_{\mathrm{I}}^2}{\lambda_{\mathrm{I}}\epsilon_4\nu}\tau_n\sum_{i=1}^k\mynormb{g^{n,i+1}}^2
	+\epsilon_4\lambda_{\mathrm{I}}\nu\tau_n\sum_{i=1}^k\mynormb{\Delta_h\omega^{n,i+1}}^2.
\end{align*}
By collecting the above estimates on $J_{Hj}^{n,k+1}$ $(j=1,2,3)$, it follows from \eqref{stability: IERK-H1 product} that
\begin{align*}
	\mynormb{\nabla_h\omega^{n,k+1}}^2-&\,\mynormb{\nabla_h\omega^{n,1}}^2
	+(2-\epsilon_3-\epsilon_4)\lambda_{\mathrm{I}}\nu\tau_n\sum_{i=1}^k\mynormb{\Delta_h\omega^{n,i+1}}^2
	\notag\\
	\le&\,-\sum_{i=1}^k\mynormb{\delta_{\tau}\nabla_h\omega^{n,i+1}}^2+\frac{\hat{\ck}_{\Omega}(\cm_0)s^2\sigma_{\mathrm{E}}^2}{16\epsilon_3\lambda_{\mathrm{I}}\nu}
	\tau_n\sum_{i=1}^k\mynormb{\delta_{\tau}\nabla_h\omega^{n,i+1}}^2\\
	&\,+\frac{\ck_g^2s\sigma_{\mathrm{I}}^2\tau_n}{\epsilon_4\lambda_{\mathrm{I}}\nu}
	+\frac{\bar{\ck}_{\Omega}s^2\sigma_{\mathrm{E}}^2\tau_n}{16\epsilon_3\lambda_{\mathrm{I}}\nu},
\end{align*}
where the priori setting $\mynormb{g^{n,i}}\le \ck_g$ was used.
By taking $\epsilon_3=\epsilon_4=1/8$ and assuming 
\begin{align}\label{cond: maximum time-stepH1}
	\tau_{\max}\le \frac{2\lambda_{\mathrm{I}}\nu}{\hat{\ck}_{\Omega}(\cm_0)s^2\sigma_{\mathrm{E}}^2},
\end{align}
one gets
\begin{align}\label{stability: IERK-H1 product3}
	\mynormb{\nabla_h\omega^{n,k+1}}^2\!
	+\!\tfrac{7\lambda_{\mathrm{I}}
		\nu\tau_n}{4}\sum_{i=1}^k\mynormb{\Delta_h\omega^{n,i+1}}^2
	\!\le\!\mynormb{\nabla_h\omega^{n,1}}^2	
	\!+\!\frac{16\ck_g^2\sigma_{\mathrm{I}}^2s\!+\!\bar{\ck}_{\Omega}\sigma_{\mathrm{E}}^2s^2}
	{2\lambda_{\mathrm{I}}\nu}\tau_n.	
\end{align}
With the help of  the Poincar\'{e} inequality, 
$\mynormb{\nabla_h\omega^{n,i+1}}\le \ck_{p2}\mynormb{\Delta_h\omega^{n,i+1}}$, 
it follows from \eqref{stability: IERK-H1 product3} that
\begin{align*}
	\mynormb{\nabla_h\omega^{n,k+1}}^2+\tfrac{7\lambda_{\mathrm{I}}\nu\tau_n}{4\ck_{p2}}\mynormb{\nabla_h\omega^{n,k+1}}^2
	\le\,\mynormb{\nabla_h\omega^{n,1}}^2
	+\frac{16\ck_g^2\sigma_{\mathrm{I}}^2s+\bar{\ck}_{\Omega}\sigma_{\mathrm{E}}^2s^2}
	{2\lambda_{\mathrm{I}}\nu}\tau_n.
\end{align*}
The damping-type discrete Gr\"{o}nwall inequality in Lemma \ref{lemma: decaying gronwall multi-stage RK} 
(with $\xi:=\frac{7\lambda_{\mathrm{I}}\nu}{8\ck_{p2}}$) yields
\begin{align}\label{stability: IERK-H1 product4}
	\mynormb{\nabla_h\omega^{m,l+1}}^2
	\le\,e^{-\tfrac{7\lambda_{\mathrm{I}}\nu}{8\ck_{p2}}t_m}
	\mynormb{\nabla_h\omega^{0}}^2
	\,+\frac{64\ck_g^2\sigma_{\mathrm{I}}^2s+4\bar{\ck}_{\Omega}\sigma_{\mathrm{E}}^2s^2}
	{7\lambda_{\mathrm{I}}^2\nu^2}\ck_{p2}:=\cm_{2},
\end{align}
where $\cm_{2}$ (involving the constant $\bar{\ck}_{\Omega}$ defined in Lemma \ref{lemma: nonlinear convection H2 bound}) is dependent on  the upper bound $\cm_1$ of $L^2$ norm for the vorticity, see \eqref{stability: IERK-L2 product4},  but always independent of the priori  bound $\cm_0$.

\subsection{Stability from recovering of the priori $H^{\delta}$ assumption}
The $L^2$ norm estimate \eqref{stability: IERK-L2 product4} says that there exists a global-in-time constant $\cm_{1}$, independent of $\cm_0$, such that 
$\mynormb{\omega^{m,l+1}}\le \cm_{1}$. The $H^1$ semi-norm estimate \eqref{stability: IERK-H1 product4} says that there exists a global-in-time constant $\cm_{2}$, independent of $\cm_0$, such that 
$\mynormb{\nabla_h\omega^{m,l+1}}\le \cm_{2}$. Then
\begin{align*}
	\mynormb{\myvec{\omega}_c^{m,l+1}}_{H^{\delta}}
	\le&\, \ck_{\Omega,\delta}\mynormb{\myvec{\omega}_c^{m,l+1}}^{1-\delta}
	\mynormb{\myvec{\omega}_c^{m,l+1}}_{H^1}^{\delta}\\
	\le&\, \ck_{\Omega,\delta}\ck_{p,\delta}\mynormb{\myvec{\omega}_c^{m,l+1}}^{1-\delta}
	\mynormb{\nabla\myvec{\omega}_c^{m,l+1}}^{\delta}
	\le \ck_{\Omega,\delta}\ck_{p,\delta}\cm_{1}^{1-\delta}\cm_{2}^{\delta}.
\end{align*}
Thus by choosing $\delta=1/2$ and taking the constant $\cm_0:=\ck_{\Omega,\delta}\ck_{p,\delta}\sqrt{\cm_{1}\cm_{2}}$, we recover the priori $H^{\delta}$ assumption \eqref{estimate: H-delta-bound} for the case $n=m$ and $j=l+1$ under the  maximum time-step size restrictions \eqref{cond: maximum time-stepL2} and \eqref{cond: maximum time-stepH1}, that is,
\begin{align}\label{cond: maximum time-step}
	\tau_{\max}\le \frac{\lambda_{\mathrm{I}}\nu}{\sigma_{\mathrm{E}}^2s^2}\max\Big\{\frac{1}{\ck_{\Omega}^2\cm_0^2},\frac{2}{\hat{\ck}_{\Omega}(\cm_0)}\Big\},
\end{align}
where the constant $\hat{\ck}_{\Omega}(\cm_0)$, defined in Lemma \ref{lemma: nonlinear convection H2 bound}, is always dependent on the priori  bound $\cm_0$.
The induction for the $H^{\delta}$ norm boundedness \eqref{estimate: H-delta-bound} is completed.


As a by-product, we prove the following theorem, which 
says that the high-order IERK method \eqref{scheme: general IERK-vorticity}-\eqref{scheme: general IERK-velocity} is unconditionally stable when  $\nu$ is not too small.
\begin{theorem} \label{thm: IERK long-time stability}
	Assume that  $\mathcal{S}(E_{s_{\mathrm{I}}}^{-1}A_{\mathrm{I}})$ has a finite minimum eigenvalue $\lambda_{\mathrm{I}}>0$, the spectral
	norms of $E_{s_{\mathrm{I}}}^{-1}A_{\mathrm{I}}$ and $E_{s_{\mathrm{I}}}^{-1}A_{\mathrm{E}}$ are bounded by the positive constants $\sigma_{\mathrm{I}}$ and $\sigma_{\mathrm{E}}$, respectively.
	Let $\omega_h^{n,i}$ be the discrete solution of the IERK method \eqref{scheme: general IERK-vorticity}-\eqref{scheme: general IERK-velocity} with an initial data $\omega_0\in H^2$ and the exterior force $g\in L^{\infty}\brab{\mathbb{R}^+;L^2}$ bounded by $\ck_g$. Let $\myvec{\omega}_c^{n,j}$ is the continuous extension in space of $\omega_h^{n,j}$. If the the maximum step size $\tau_{\max}$ satisfies \eqref{cond: maximum time-step},  there exists a constant 
	$\cm_*$ such that 
	$$\mynormb{\myvec{\omega}_c^{n,j}}_{H^{1}}\le \cm_*=\cm_*\brab{\mynorm{\omega_0}_{H^2},\mynorm{g}_{L^2},\nu^{-1},s,\lambda_{\mathrm{I}}, \sigma_{\mathrm{I}},\sigma_{\mathrm{E}}}\quad\text{for $n\ge1$ and $1\le j\le s$.}$$
	In addition, the $L^2$ norm and $H^1$ semi-norm of the vorticity (equivalent to the $H^1$ and $H^2$ semi-norms of the velocity) has the following asymptotic decaying estimates
	\begin{align*}
		\mynormb{\omega^{n,j}}^2
		\le &\,e^{-\tfrac{7\lambda_{\mathrm{I}}\nu}{8\ck_{p1}}t_n}
		\mynormb{\omega^{0}}^2
		+\tfrac{64s\sigma_{\mathrm{I}}^2\ck_{p1}^3\ck_g^2}{7\lambda_{\mathrm{I}}^2\nu^2}:=\cm_1,\\
		\mynormb{\nabla_h\omega^{n,j}}^2
		\le &\,e^{-\tfrac{7\lambda_{\mathrm{I}}\nu}{8\ck_{p2}}t_n}
		\mynormb{\nabla_h\omega^{0}}^2
		+\tfrac{64\ck_g^2\sigma_{\mathrm{I}}^2s+4\bar{\ck}_{\Omega}\sigma_{\mathrm{E}}^2s^2}
		{7\lambda_{\mathrm{I}}^2\nu^2}\ck_{p2},
	\end{align*}	
	where $s$, $\lambda_{\mathrm{I}}$, $\sigma_{\mathrm{I}}$ and  $\sigma_{\mathrm{E}}$ are determined by the IERK method
	and the global-in-time constant $\bar{\ck}_{\Omega}$, defined in Lemma \ref{lemma: nonlinear convection H2 bound}, is also dependent on  the upper bound $\cm_1$ of $L^2$ norm for the vorticity.
\end{theorem}


	Note that, the above proof of Theorem \ref{thm: IERK long-time stability} can not be directly applied to derive the $H^2$ semi-norm estimate for the vorticity by testing the stage difference form \eqref{scheme: general IERK-vorticity-differential}  with $2\Delta_h^2\omega_h^{n,i+1}$. Actually, under the priori $H^{\delta}$ assumption \eqref{estimate: H-delta-bound},  the convolution term  $2\tau_n\sum_{i=1}^k\sum_{j=1}^{i}
	\underline{\hat{a}}_{i+1,j}\myinnerb{\mathtt{c}_h(\myvec{u}^{n,j},\omega^{n,j}),\Delta_h^2\omega^{n,i+1}}$
	of nonlinear convection can not be directly controlled by $\sum_{i=1}^k\mynormb{\delta_{\tau}\Delta_h\omega^{n,i+1}}^2+2\lambda_{\mathrm{I}}\nu\tau_n\sum_{i=1}^k\mynormb{\Delta_h\nabla_h\omega^{n,i+1}}^2 $. It seems that  a priori $H^{1+\delta}$ assumption on the stage boundedness of vorticity for some $\delta>0$ would be sufficient to establish the asymptotic decaying estimate in the $H^2$ semi-norm, cf. \cite{GottliebWang:2012}. That is, we need a new process of mathematical induction for the priori $H^{1+\delta}$ assumption and a new technical lemma (similar to Lemma \ref{lemma: nonlinear convection H2 bound}) to handle the following convolution term of nonlinear convection $$2\tau_n\sum_{i=1}^k\sum_{j=1}^{i}
	\underline{\hat{a}}_{i+1,j}\myinnerb{\nabla_h\mathtt{c}_h(\myvec{u}^{n,j},\omega^{n,j}),\Delta_h\nabla_h\omega^{n,i+1}},$$ while the technical details are left to interested readers. 		
	The long time stability in the $L^2$ and $H^m$ $(m\ge1)$ norms further leads to the convergence of the global attractors and invariant measures of the IERK scheme \eqref{scheme: general IERK-vorticity}-\eqref{scheme: general IERK-velocity} to those of the INSE at vanishing step sizes. Interested readers can referred to \cite{GottliebToneWangWangWirosoetisno:2012,Wang:2010,Wang:2012,Wang:2016} for details.

		In  Section \ref{sec: IERK methods} of the supplementary material, we apply the simplified order conditions in Table \ref{table: order condition} to construct three new parameterized IERK methods based on the requirement of Theorem \ref{thm: IERK long-time stability}, including the IERK(2,3;$c_2$), IERK(3,5;$a_{55}$) and IERK(4,7;$\hat{a}_{43}$) methods, where the notation IERK($p$,$s$;$\mu$) represents the  $p$-th order $s$-stage $\mu$-parameterized IERK method.

	\section{Numerical experiments}\label{sec: experiments}
	
	\subsection{Accuracy verification}
	\begin{example}\label{ex: WangCheng}
		Consider the INSE \eqref{cont: INSE-vorticity}-\eqref{cont: INSE-velocity} on  $\Omega=(0,2\pi)^2$ with the viscosity coefficient $\nu=0.5$. We prescribe an artificial time-dependent forcing term $g$ so that the system admits the exact solution $\psi=\frac12 \cos t\sin x\sin y, \ \omega=\cos t\sin x\sin y.$
		Always, the spatial operators are approximated by the Fourier pseudo-spectral method on a $128 \times 128$ uniform grid.
	\end{example}
	
	\begin{figure}[htb!]
		\centering
		\subfigure[IERK(2,3;0.35)]{
			\includegraphics[width=0.3\textwidth]{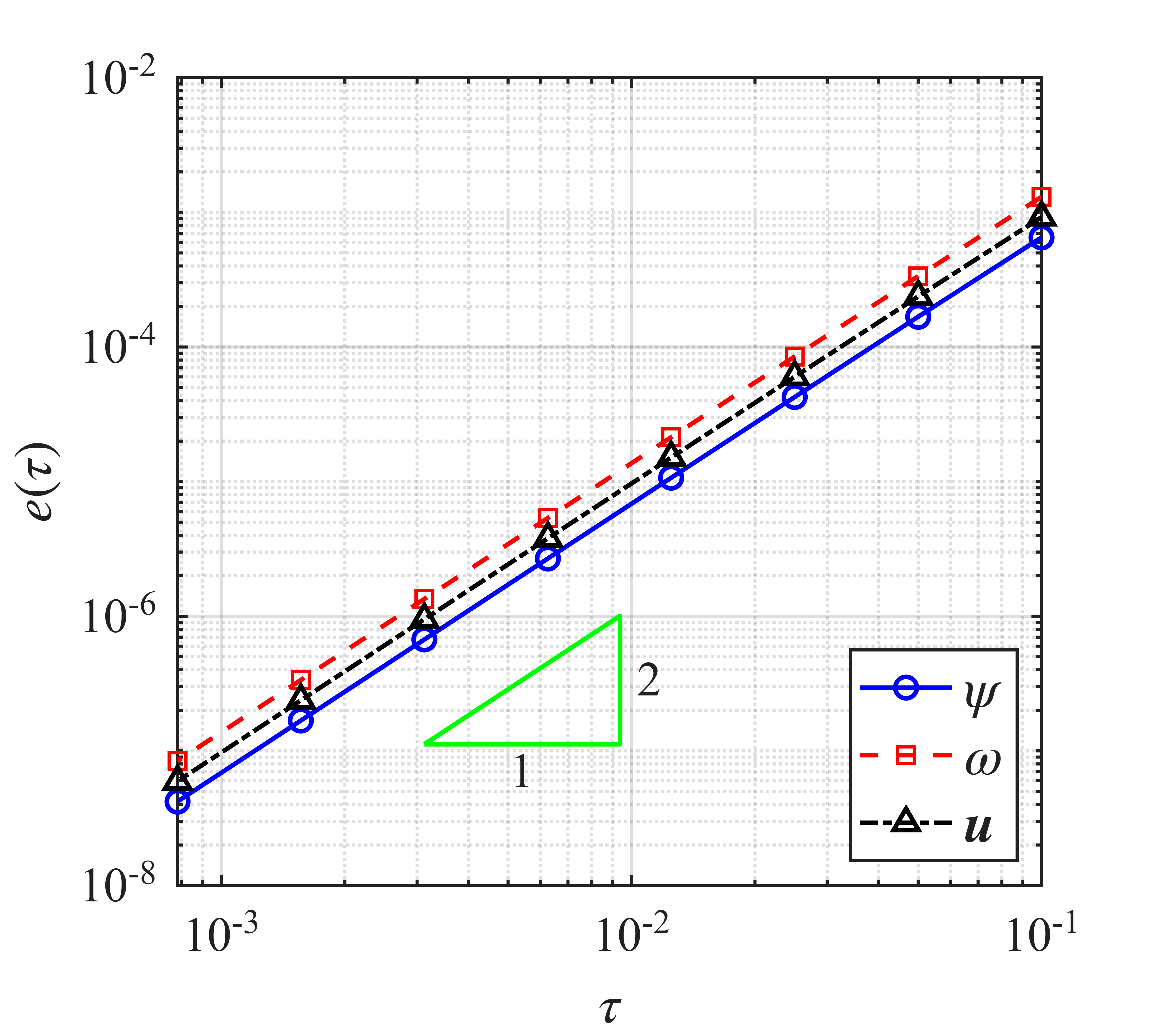}} 
		\subfigure[IERK(3,5;1.2)]{
			\includegraphics[width=0.3\textwidth]{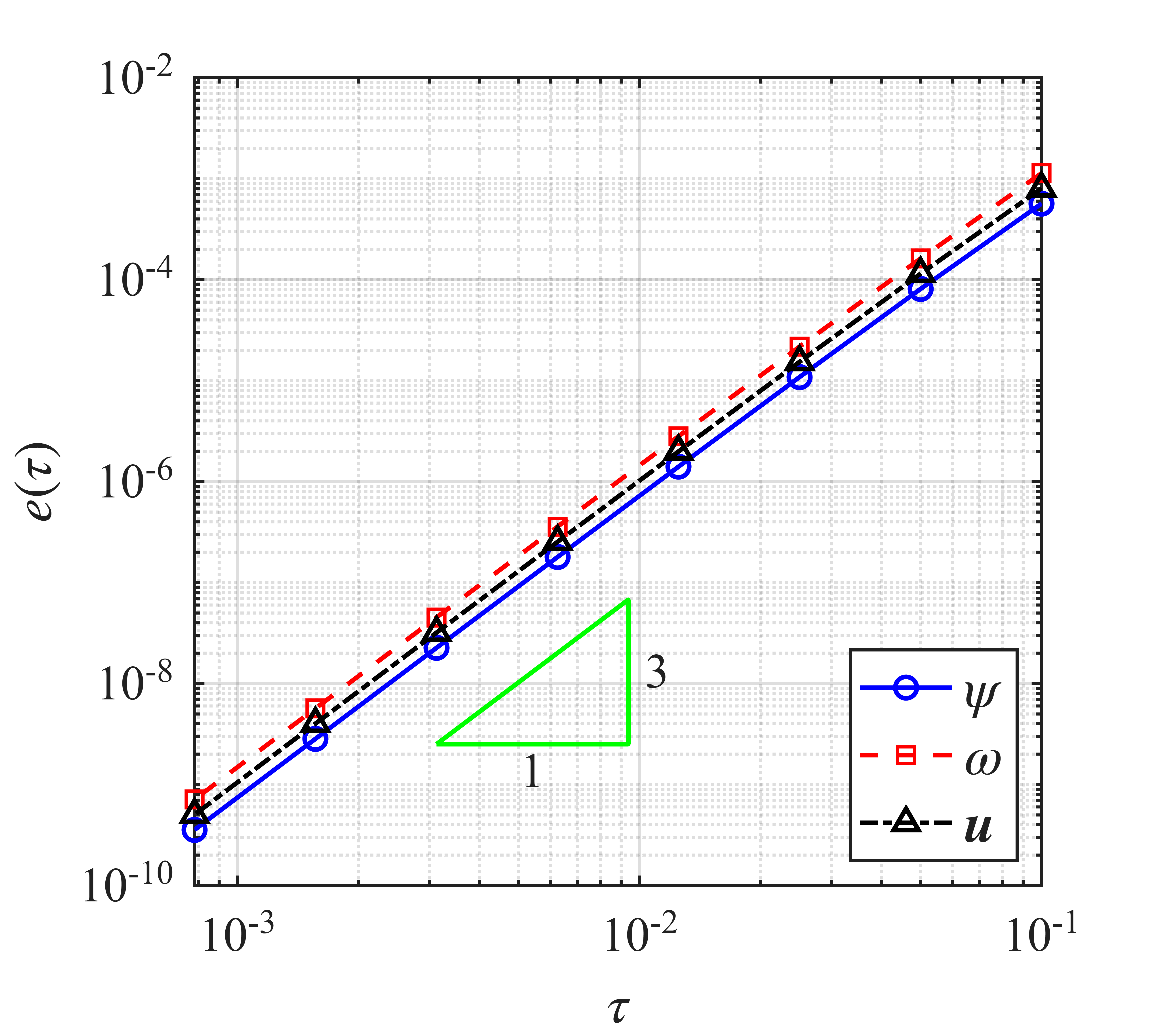}} 
		\subfigure[IERK(4,7;-0.8)]{
			\includegraphics[width=0.3\textwidth]{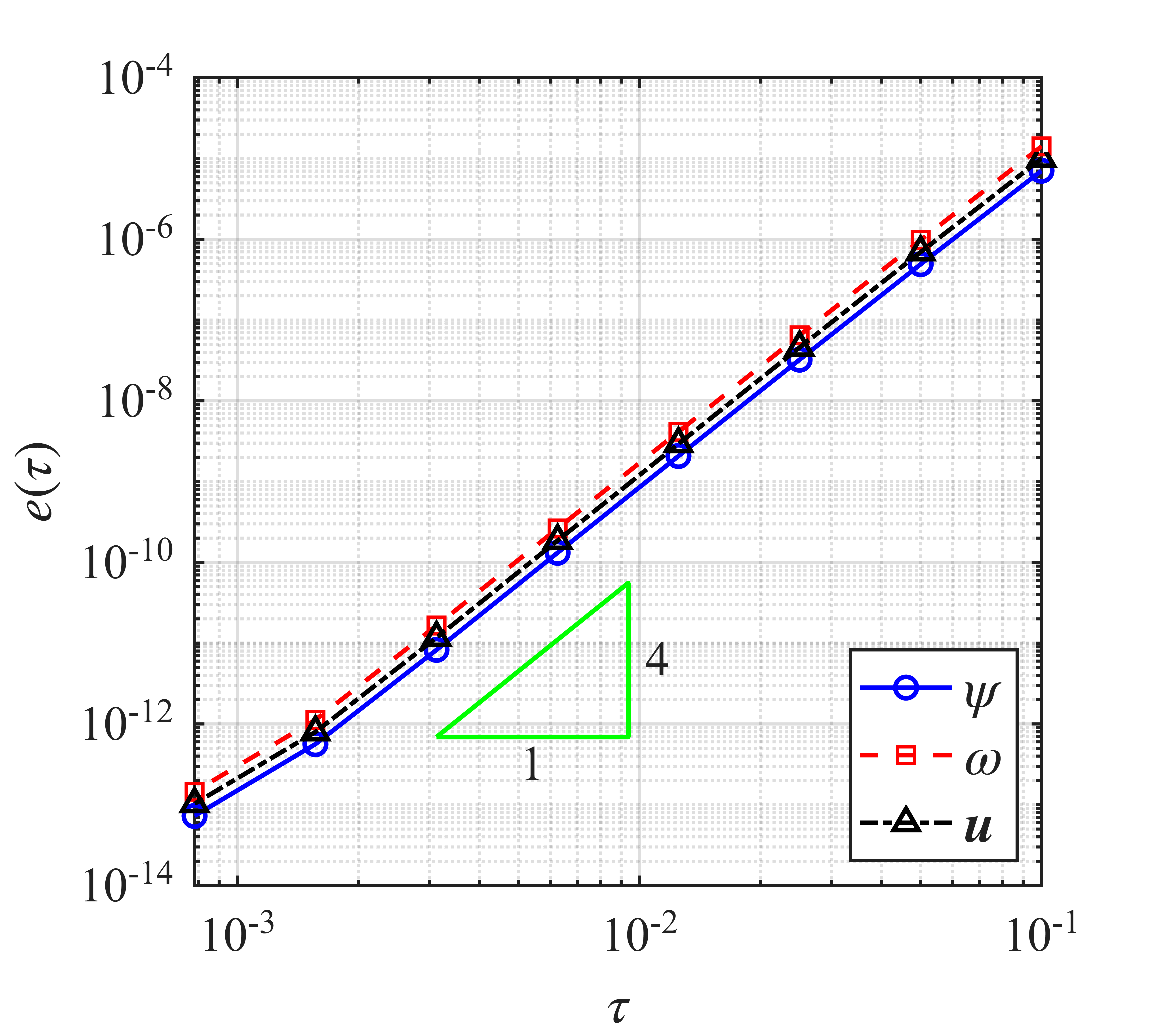}}
		\caption{Errors of IERK(2,3;0.35), IERK(3,5;1.2) and IERK(4,7;-0.8) schemes.}
		\label{fig: accuracy of Radau2-4}
	\end{figure}
	
		\begin{figure}[htb!]
		\centering
		\subfigure[IERK(2,3;$c_2$)]{
			\includegraphics[width=0.3\textwidth]{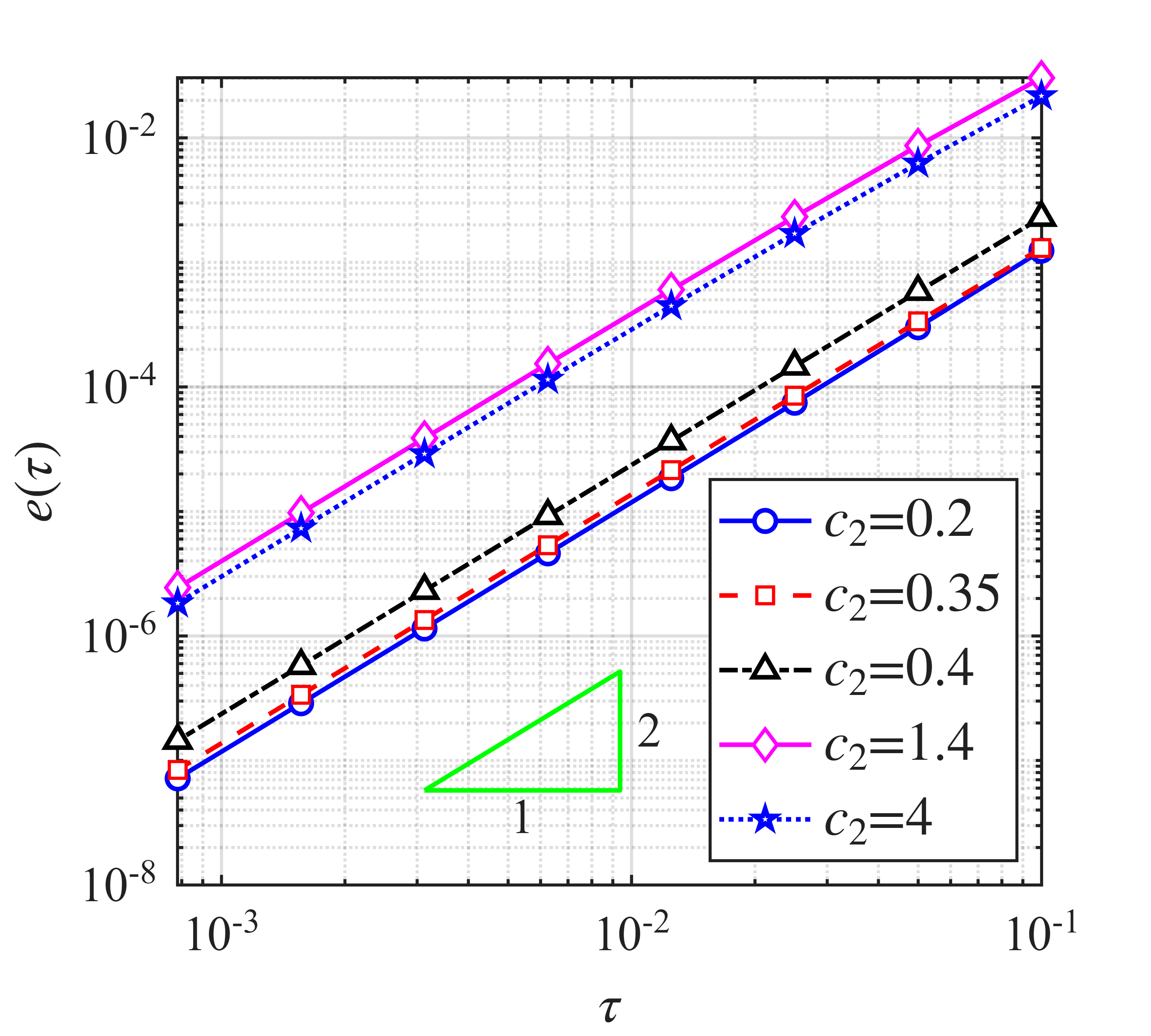}}
		\subfigure[IERK(3,5;$a_{55}$)]{
			\includegraphics[width=0.3\textwidth]{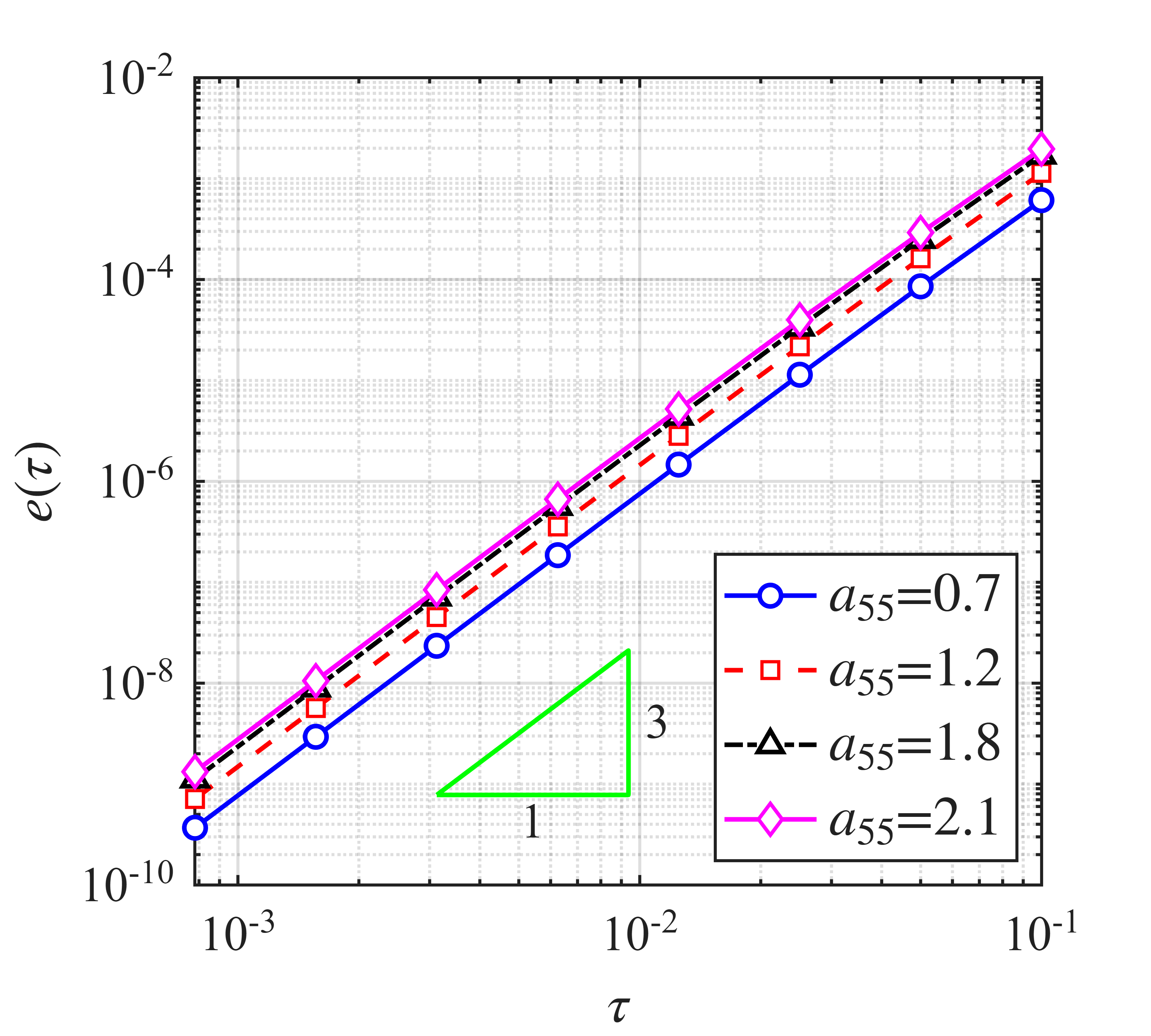}}
		\subfigure[IERK(4,7;$\hat{a}_{43}$)]{
			\includegraphics[width=0.3\textwidth]{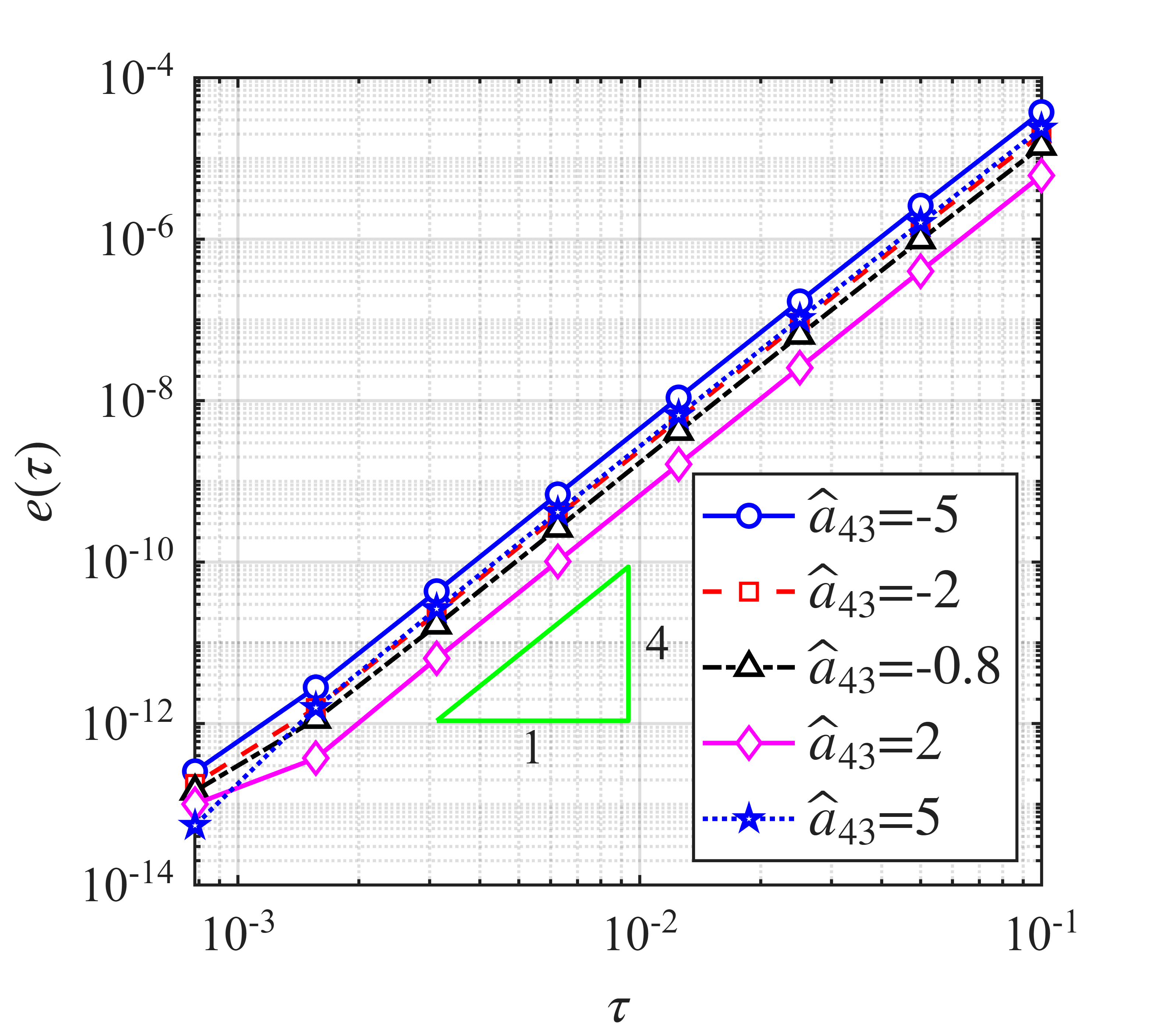}}
		\caption{Errors of IERK(2,3;$c_2$), IERK(3,5;$a_{55}$) and IERK(4,7;$\hat{a}_{43}$).}
		\label{fig: accuracy of other-IERK}
	\end{figure}
	
	We employ the IERK(2,3;0.35), IERK(3,5;1.2) and IERK(4,7;-0.8) as the representative second-, third-, and fourth-order methods, respectively. Numerical solutions are computed using time steps $\tau=2^{-k}/10$ for $0\leq k\leq 7$, and the errors are measured until the final time $T=1$. Figure \ref{fig: accuracy of Radau2-4} displays the $L^2$-norm errors of the vorticity, velocity and stream function. As expected, the three methods achieve the second-, third-, and fourth-order temporal convergence rates, respectively. In Figure \ref{fig: accuracy of other-IERK}, we further investigate the impact of the parameters by plotting the vorticity errors for the IERK(2,3;$c_2$), IERK(3,5;$a_{55}$) and IERK(4,7;$\hat{a}_{43}$) methods using different parameter values. It is observed that the choice of these method parameters may cause significant effects on the numerical precision. For this example, $c_2=0.2$, $a_{55}=0.7$ and $\hat{a}_{43}=2$ appear to be the best choice for their respective methods, as they yield the most accurate solutions. 
	Nonetheless, it remains mysterious to us. Maybe, the leading coefficients of truncation error vary with respect to these parameters.

	\subsection{Adaptive time-stepping strategies}
	
		\begin{figure}[htb!]
		\centering
		\subfigure[$\|\omega^{n}\|^{2}$ with $\tau_{\max}=1$]{
			\includegraphics[width=0.3\textwidth]{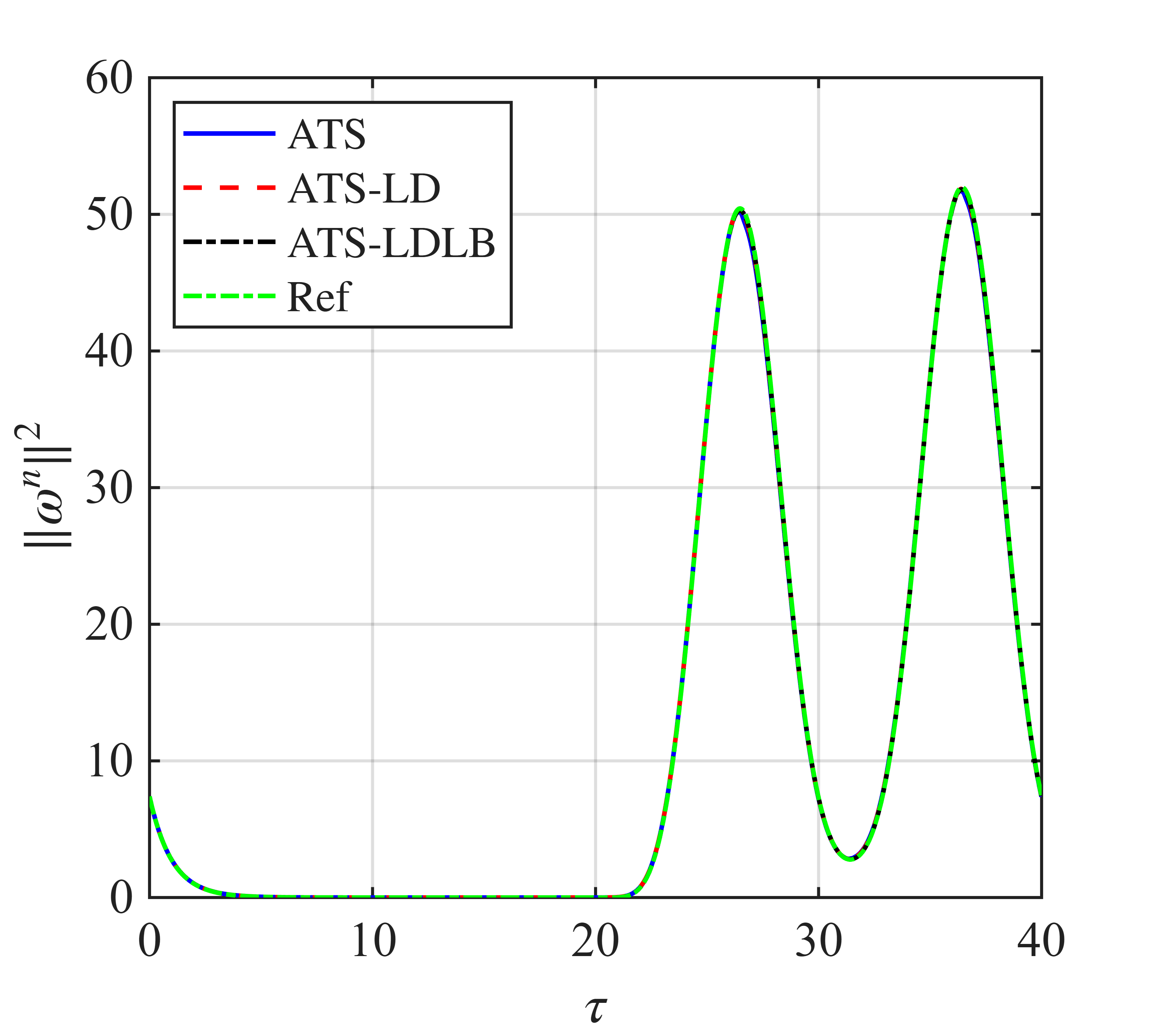}
		}
		\subfigure[$\tau_{n}$ with $\tau_{\max}=1$]{
			\includegraphics[width=0.3\textwidth]{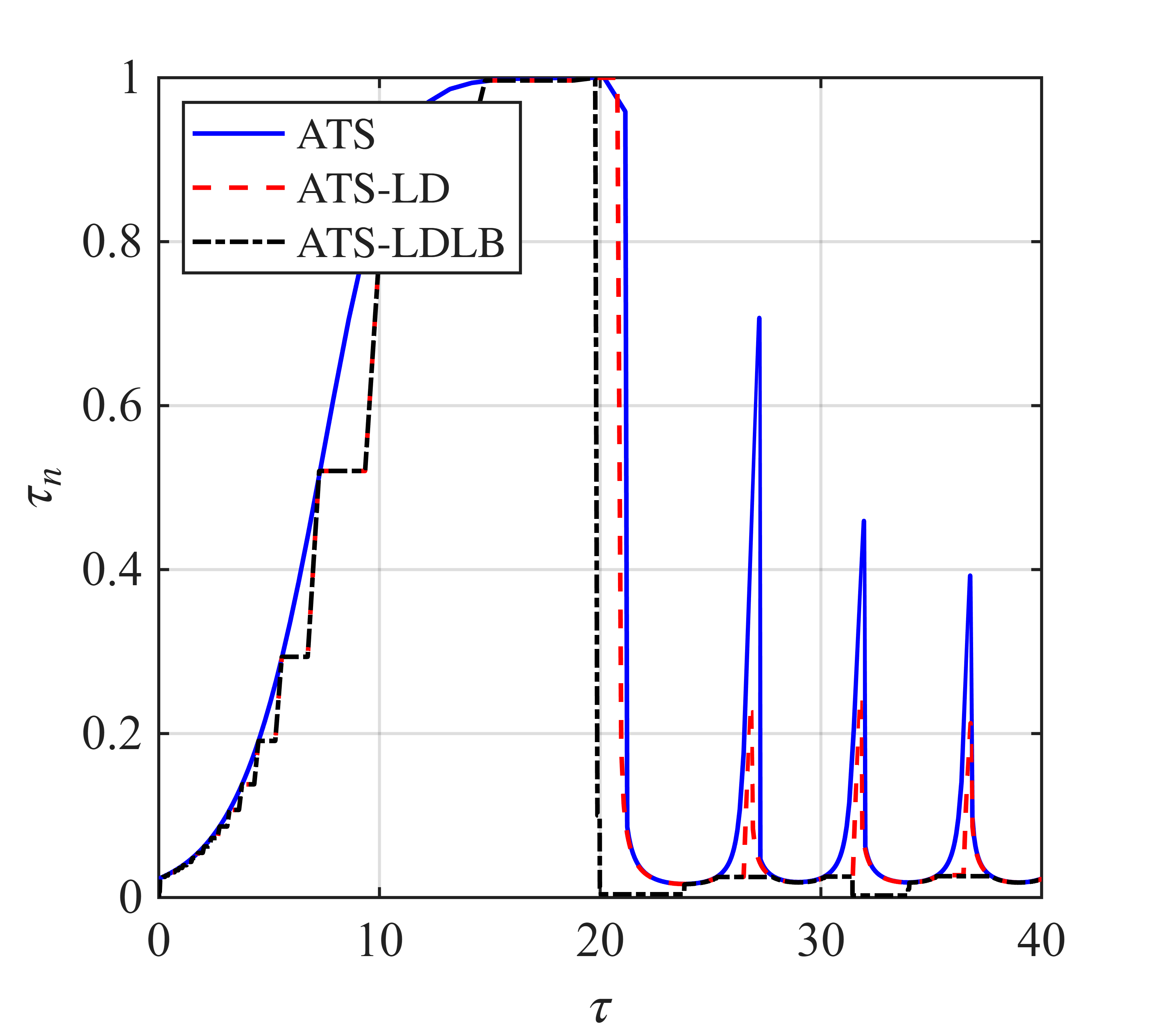}
		}
		\subfigure[$\|e^{n}_{\text{mix}}\|_{\infty}$ with $\tau_{\max}=1$]{
			\includegraphics[width=0.3\textwidth]{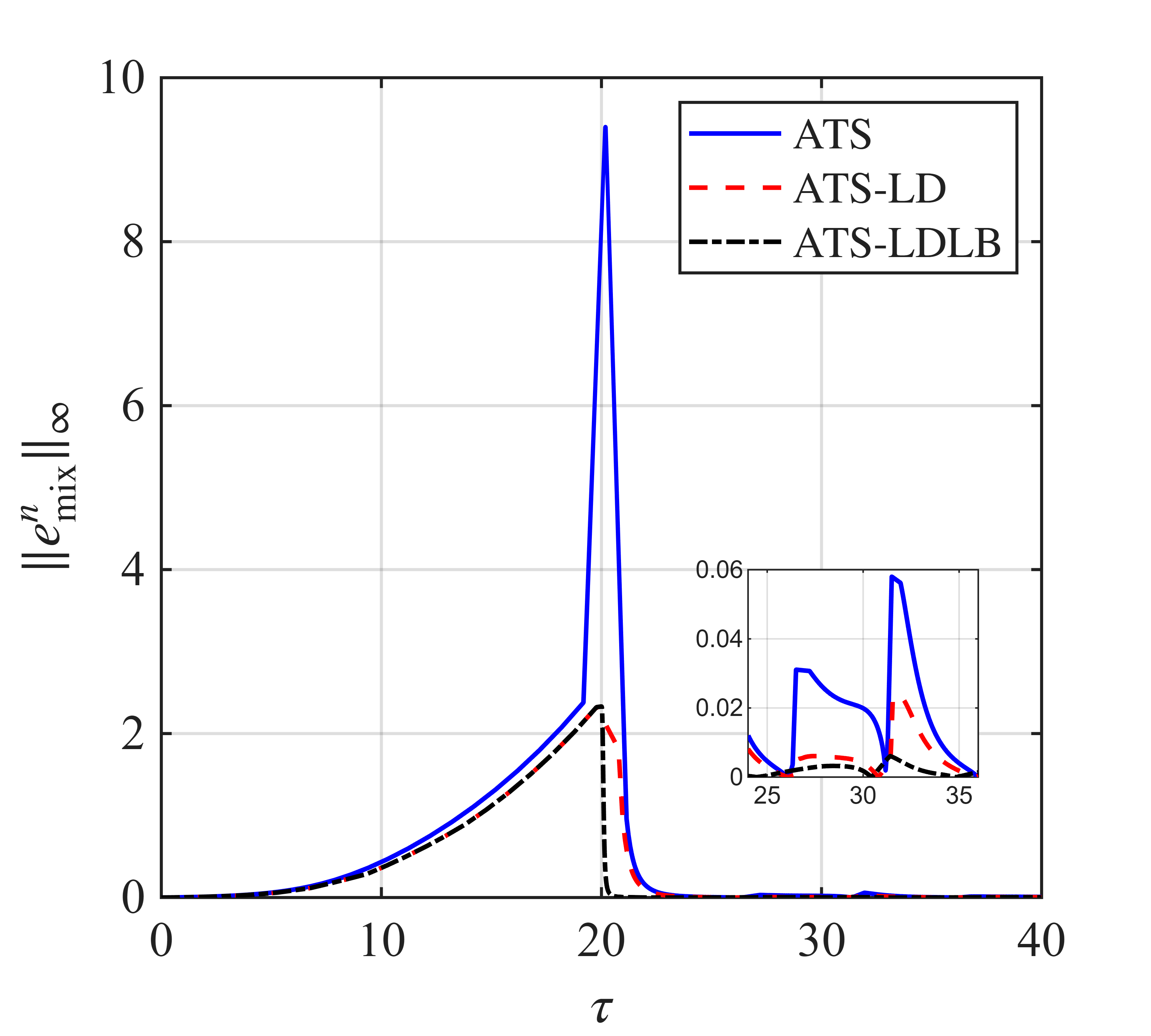}
		}
		\subfigure[$\|\omega^{n}\|^{2}$ with $\tau_{\max}=0.5$]{
			\includegraphics[width=0.3\textwidth]{IERK1_sin2T2_dt=0.5_enstrophy.png}
		}
		\subfigure[$\tau_{n}$ with $\tau_{\max}=0.5$]{
			\includegraphics[width=0.3\textwidth]{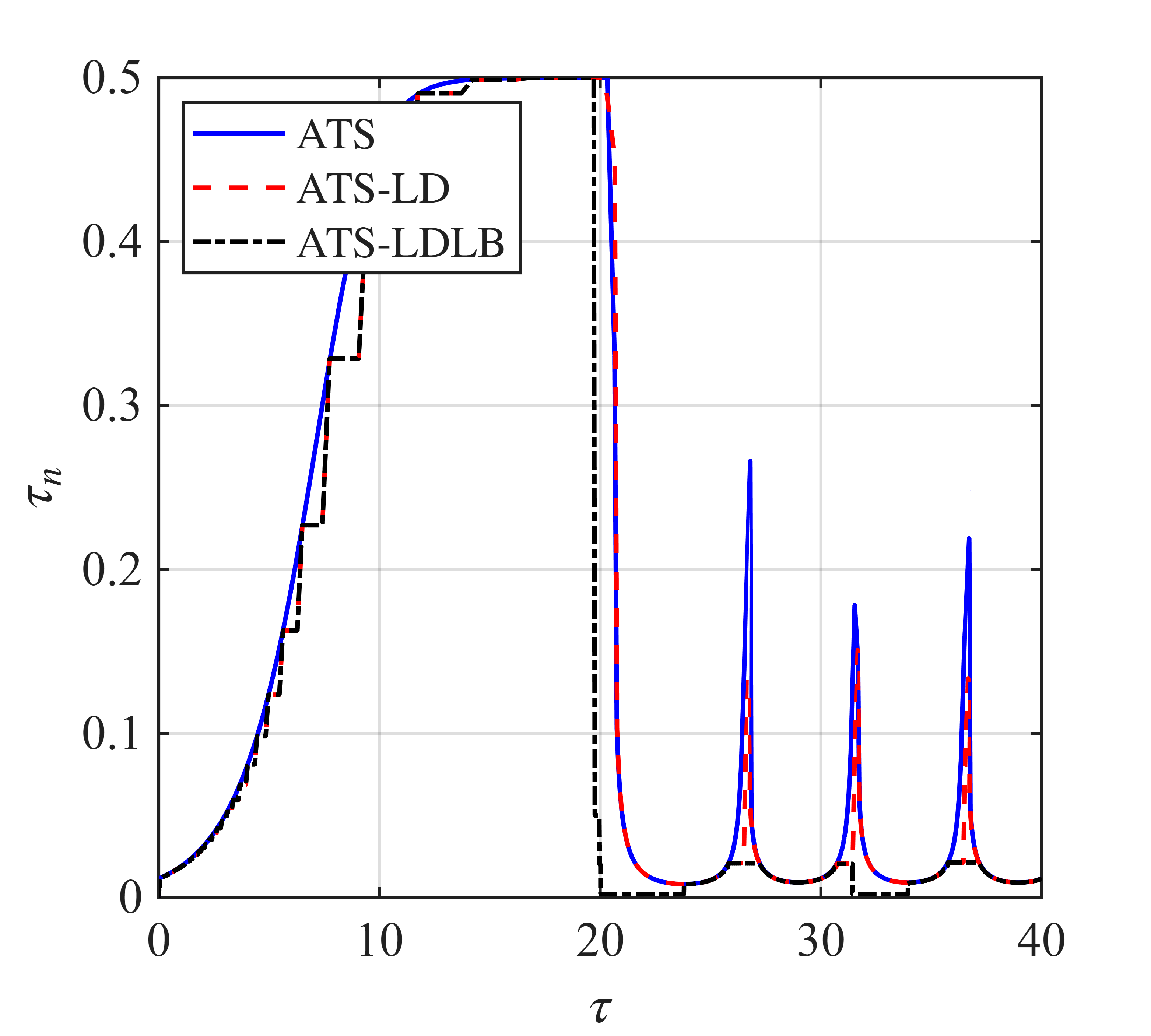}
		} 
		\subfigure[$\|e^{n}_{\text{mix}}\|_{\infty}$ with $\tau_{\max}=0.5$]{
			\includegraphics[width=0.3\textwidth]{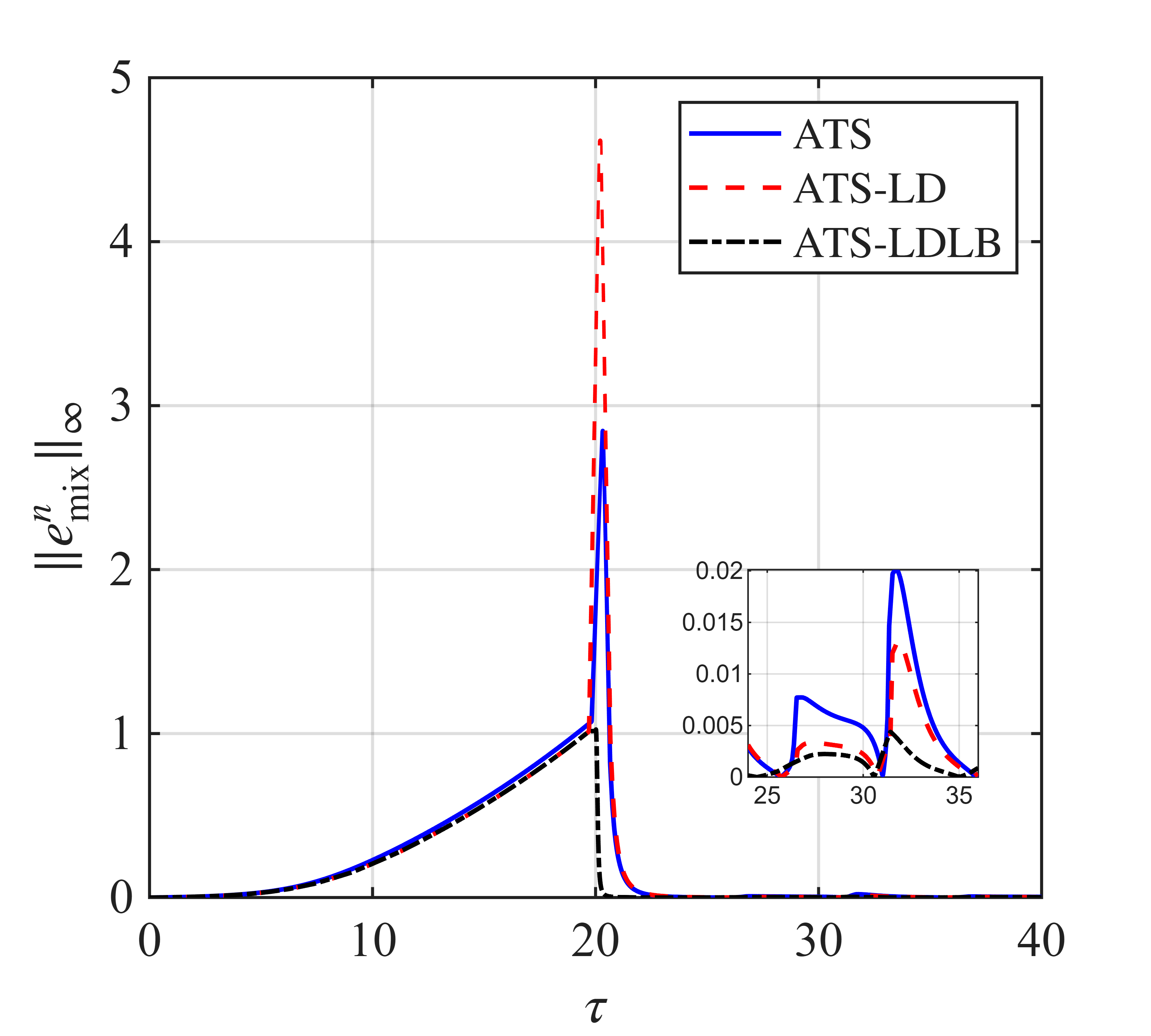}
		} 
		\caption{Numerical simulation of Example \ref{ex: ATS3comparison} using the first-order scheme \eqref{scheme: BDF1-vorticity}}
		\label{fig: ex1-Radau1}
	\end{figure}
	
	\begin{example}\label{ex: ATS3comparison}	
		Consider the INSE \eqref{cont: INSE-vorticity}-\eqref{cont: INSE-velocity} on the domain $\Omega = (0, 2\pi)^2$ with periodic boundary conditions and a viscosity coefficient $\nu = 0.5$. We introduce a  forcing term of separable form
		$	g(x,y,t) = g_1(x,y)\, g_2(t),$
		where the $T$-periodic weight function $g_2(t)$ is defined as $g_2(t) = 0$ for $t \in [\ell T,\, \ell T + T_1]$
		and $\forall\, \ell \in \mathbb{N}^+$.
		More specifically,  for the first periodic region and a positive integer $l_t$, we define
		\begin{equation}\label{eq:temporal-forcing}
			g_2(t) :=
			\begin{cases}
				0, & t \in [0, T_1], \\[4pt]
				\displaystyle \sin^2\!\brab{\tfrac{2\pi l_t (t - T_1)}{T-T_1}}, 
				& t \in [T_1, T].
			\end{cases}
		\end{equation}
		This forcing exhibits two distinct temporal regimes: a quiescent interval on $[0, T_1]$ and a highly oscillatory interval on $[T_1,T]$ when $l_t \gg 1$. 
		
		In the regime of large viscosity (i.e., small Reynolds number), and provided the quiescent interval is sufficiently long,  the long-time dynamics of the solution closely follow the slow-fast structure of the forcing, yielding an essentially time-periodic evolution. In particular, if $g_1$ is chosen as a Kolmogorov-type forcing
		$	g_1(x,y) = \sin(l_x x)$ for $l_x \in \mathbb{N}^+$,
		then the vorticity equation admits an exact solution $	\omega(x,y,t) = f(t)\, \sin(l_x x),$
		where $f(t)$ satisfies the ordinary differential equation,  $	\frac{df}{dt} + \nu l_x^2 f(t) = g_2(t).$
		Thus $f(t)$ can be explicitly expressed as 
		\begin{equation}
			f(t) =
			\begin{cases}
				\displaystyle f (0) \ze^{-\nu\ell_x^2 t}, & t \in [0, T_1], \\
				\displaystyle f (T_1) \ze^{-\nu\ell_x^2 (t-T_1) } + \ze^{-\nu\ell_x^2 (t-T_1) } \mathcal{J}_{2}\braB{\nu\ell_x^2, \tfrac{2\pi\ell_t}{T-T_1}, t-T_1}, & t \in (T_1, T],
			\end{cases}
		\end{equation}
		where $\mathcal{J}_2 (a, \omega, \Delta t) := \int_0^{\Delta t} \ze^{as} \sin^2 (\omega s) \zd s$.
		To ensure a time-periodic evolution such that $f(0) = f(T)$, the initial value is prescribed as
		\[
		f (0) = {\mathcal{J}_2 (\nu\ell_x^2, \tfrac{2\pi\ell_t}{T-T_1}, T - T_1) }/\brab{1 - \ze^{- \nu\ell_x^2 T}}.
		\]	
		For sufficiently large $\nu$, this exact solution is globally attracting, providing an excellent benchmark for evaluating the time-adaptivity of numerical schemes.  Always, the spatial operators are discretized by the Fourier pseudo-spectral method on a $128 \times 128$ uniform grid. To develop a practically effective approach of adaptive time-stepping for the IERK method \eqref{scheme: general IERK-vorticity}-\eqref{scheme: general IERK-velocity}, we set $l_t = 1$, $T_1 = 20$, and $T = 40$ to clearly distinguish the error behaviors and step-size response inside the oscillatory regime.
	\end{example}

	To resolve the non-equilibrium dynamical behaviors, including the high-frequency quasi-periodic behaviors, of the INSE \eqref{cont: INSE-vorticity}-\eqref{cont: INSE-velocity} subject to an external force $g$, we first introduce the ATS-LD strategy (Algorithm \ref{Adaptive-Time-Strategy-Delay} in Section \ref{appendix: ATS-LDLB strategy} of the supplementary material), which extends the traditional adaptive time-stepping approach (denoted as ATS). The traditional ATS determines the adaptive step size by
	\begin{align}\label{def: adaptive step size}
		\tau_{\text{ada}}:=\min\Big\{\max\Big\{\tau_{\min},\dfrac{\tau_{\max}}{\sqrt{1+\beta\|\partial_{\tau}\omega^n\|^2}}\Big\},r^{*} \tau_{n}\Big\},
	\end{align}
	where  $\beta$ is a predetermined sensitivity parameter controlling the adaptivity level and $r^{*}$ is the maximum step ratio predetermined by the user to control the potential growth of step-sizes.
	The  ATS strategy has been confirmed to be practically effective for resolving low-frequency behaviors 
	approaching an equilibrium state, cf. \cite{ChenHuangLiaoYi:2026cicp,LiaoTangWangZhou:2025R-IERK, LiaoZhang:2021}.

	We employ the implicit-explicit Euler method \eqref{scheme: BDF1-vorticity} to evaluate the performance of the ATS-LD strategy for three different maximum steps $\tau_{\max} = 1$, $\tau_{\max} = 0.5$ and $\tau_{\max} = 0.1$. The involved parameters for adaptive time-stepping strategies are set as follows: $\tau_{\min}=10^{-4},$ $\beta=1000,$
	$d_{\max} = 5,$ $\gamma_{\mathrm{tol}} = 10^{-3}$ and $r^{*} = 4.$
	To evaluate the numerical precision, we define the mixed errors $e^{n}_{\mathrm{mix}}$ on the spatial grid points:
	\[
	e^{n}_{\mathrm{mix}}:= 
	\begin{cases}
		\omega^n_{h} - \omega(x_h, y_h, t_n), & \text{if } |\omega(x_h, y_h, t_n)| \leq 10^{-8}, \\[3pt]
		\tfrac{\omega^n_{h} - \omega(x_h, y_h, t_n)}{\omega(x_h, y_h, t_n)}, & \text{otherwise}.
	\end{cases}
	\]
	In general, it represents the relative error of numerical solution; while we use the absolute error to measure the numerical precision in the regions where the magnitude of vorticity is very small.
	
	Figure~\ref{fig: ex1-Radau1} depicts the results for $\tau_{\max} = 1$ (top) and $\tau_{\max} = 0.5$ (bottom), in which ATS-LDLB represents an improved version of the ATS-LD algorithm and will be described later. We observe that the discrete enstrophies under different adaptive time-stepping strategies  match the reference curve perfectly, see Figs.~\ref{fig: ex1-Radau1} (a) and (d).

	Inside the oscillatory interval $[T_1, T]$, as shown in Figs.~\ref{fig: ex1-Radau1} (b) and (e), the adaptive steps of the traditional ATS method can reach 0.5 for the predetermined $\tau_{\max} = 0.5$, but interestingly, they always stay below 0.4 for the predetermined $\tau_{\max} = 1$. This behavior is inconsistent with the expected performance of a robust adaptive stepping strategy, because smaller step sizes (and thus smaller solution errors) would always expected when the predetermined value $\tau_{\max}$ is reduced. Indeed, Figs.~\ref{fig: ex1-Radau1} (c) and (f) show that ATS yields the maximum error about 0.06 under $\tau_{\max} = 0.5$, which is larger than the maximum error 0.04 observed under the larger setting $\tau_{\max} = 1$.  	
	In contrast, ATS-LD maintains consistently smaller step sizes inside the oscillatory interval compared to ATS. One can observe that the associated adaptive steps and the solution errors decrease monotonically as the predetermined value $\tau_{\max}$ reduces. Consequently, the new ATS-LD strategy is superior to  ATS since the former generates more accurate solutions inside the oscillatory interval $[T_1, T]$.

	However, ATS-LD performs worse than ATS under the setting $\tau_{\max} = 0.5$ near the on-off point of forcing at $T_1 = 20$, as seen in Figs.~\ref{fig: ex1-Radau1} (c) and (f). This degradation is likely due to the delay mechanism allowing the adaptive steps to increase to the maximum one $\tau_{\max}$ during the preceding decay phase ($t < 20$), resulting a numerical response lag at the onset of the oscillatory regime ($t=T_1$), which leads to large local errors.
	To remedy this issue, we propose an improved ATS-LD algorithm (see Algorithm \ref{Adaptive-Time-Strategy-Delay-Recompute} in Section \ref{appendix: ATS-LDLB strategy} of the supplementary material), called ATS-LDLB, by incorporating a \textbf{local backtrack mechanism} triggered when the vorticity gradient exceeds a variation threshold $\beta_{\text{thr}}$. One can observe from Figure~\ref{fig: ex1-Radau1} that ATS-LDLB (taking the variation threshold $\beta_{\text{thr}}=10$) effectively prevents the potential lags of step-size observed in ATS-LD especially near the on-off point of source (cf. the point $t=T_1$ of the exterior force \eqref{eq:temporal-forcing}), balancing the computational efficiency during the steady decay phase with the numerical accuracy near the transition point. Moreover, ATS-LDLB exhibits smoother variations of  steps inside the oscillatory interval.

	\begin{figure}[htb!]
		\centering
		\subfigure[$\tau_{n}$]{
			\includegraphics[width=0.35\textwidth]{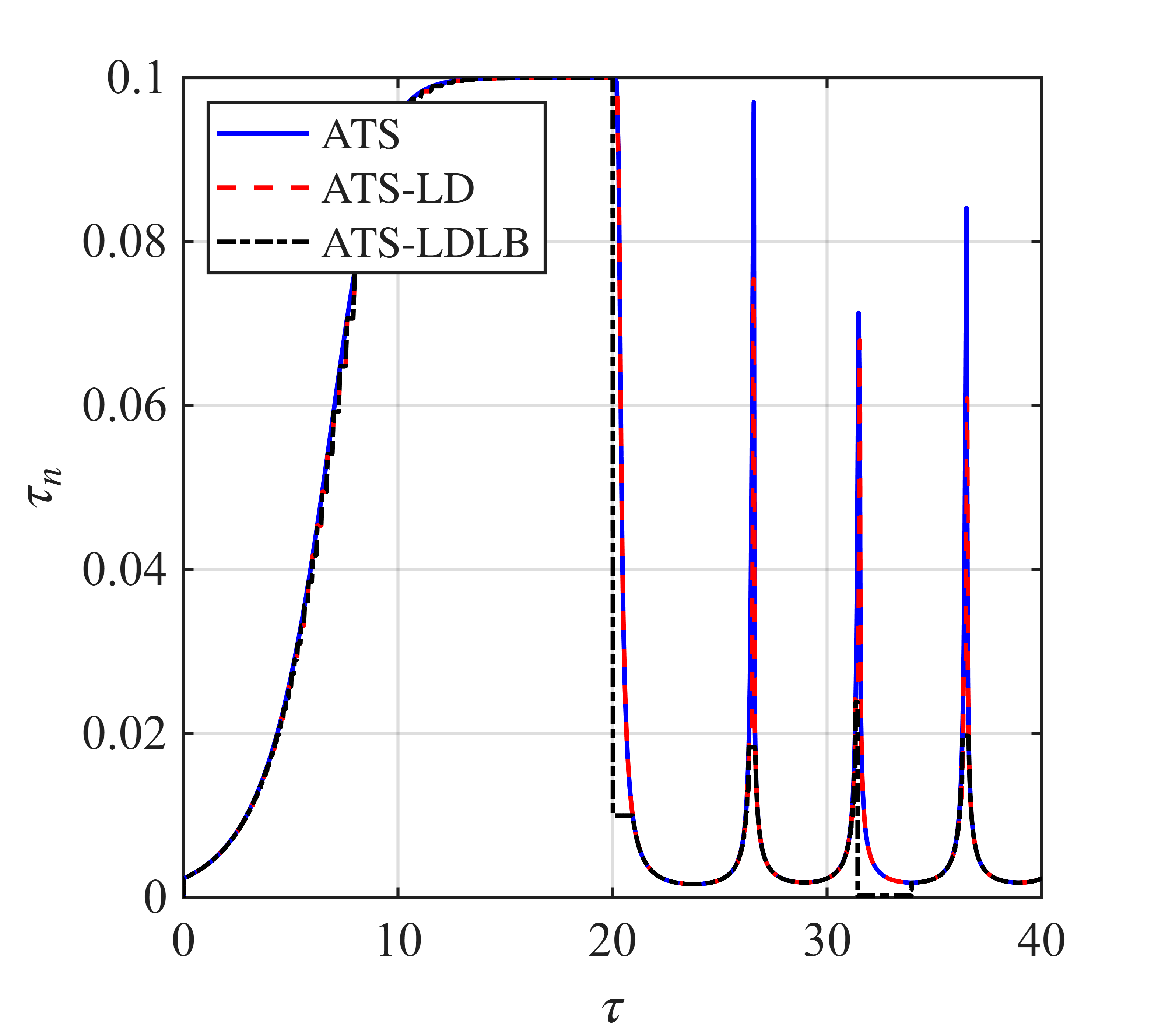}
		}
		\subfigure[$\|e^{n}_{\text{mix}}\|_{\infty}$]{
			\includegraphics[width=0.35\textwidth]{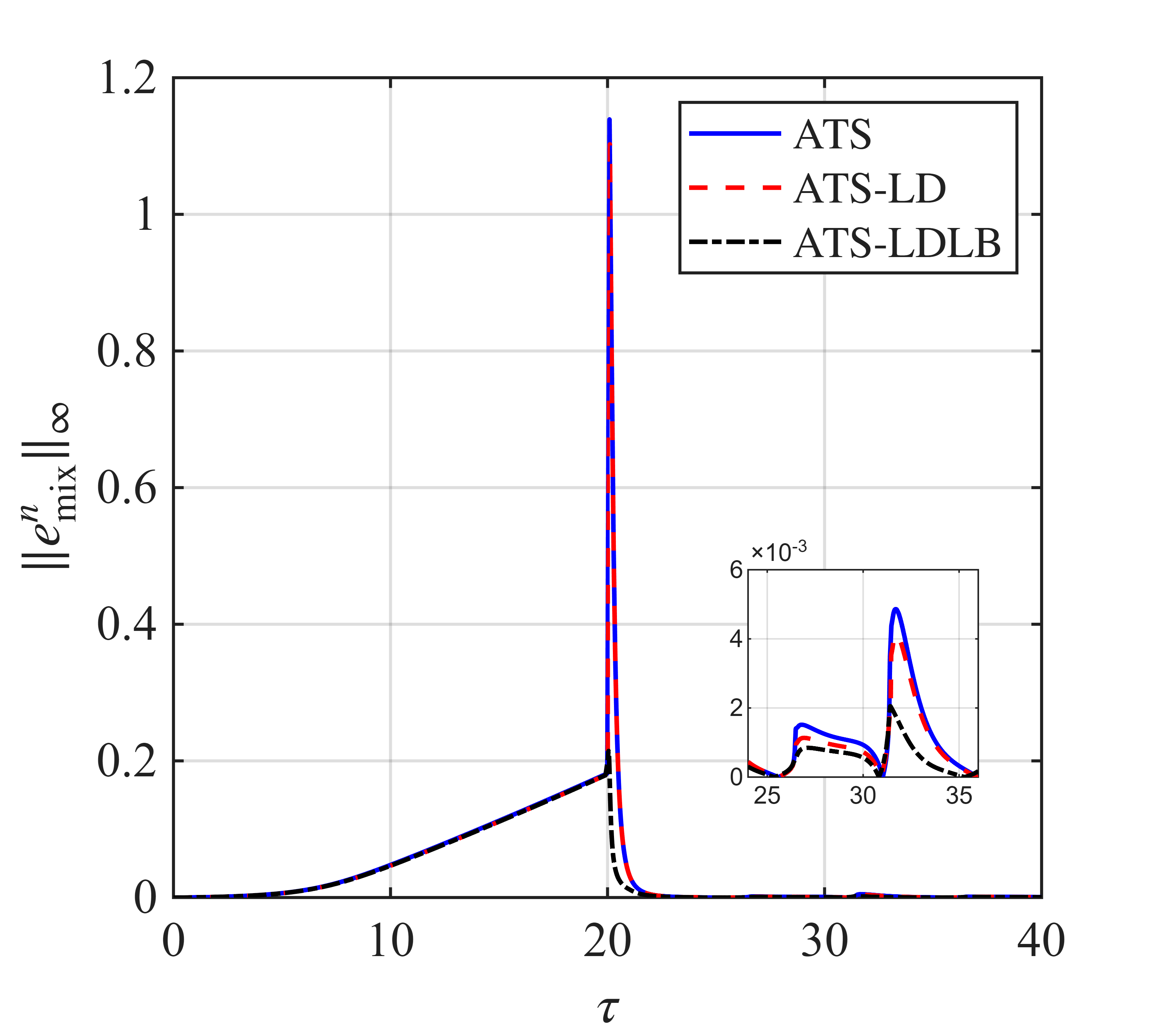}
		}
		\caption{Simulation of Example \ref{ex: ATS3comparison} using the first-order scheme \eqref{scheme: BDF1-vorticity} with $\tau_{\max}=0.1$.}
		\label{fig: ex1-Radau1-tau=0.1}
	\end{figure}

	Further numerical results for the case $\tau_{\max} = 0.1$ are presented in Figure~\ref{fig: ex1-Radau1-tau=0.1}. As expected, inside the oscillatory interval, ATS-LDLB again yields the smallest relative error, followed by ATS-LD, while ATS produces the largest error. Near the on-off point  ($t=T_1$) of forcing, the relative error of ATS-LD remains comparable to that of ATS, whereas ATS-LDLB demonstrates significantly superior performance. Correspondingly, ATS-LDLB rapidly reduces its step size without any delay effect observed in the results of ATS-LD near the critical transition point $t=T_1$.
	
	In summary, ATS-LDLB consistently delivers the best combination of accuracy and stability across all tested  values of the maximum time-step $\tau_{\max}$. It may therefore be a good candidate for 
	the non-equilibrium dynamical behaviors of \eqref{cont: INSE-vorticity}-\eqref{cont: INSE-velocity}. 
	
	\subsection{Adaptive time-stepping tests for IERK schemes} 
	
	Next, we evaluate the accuracy and stability of the parameterized  second-, third-, and fourth-order IERK schemes in Sections \ref{sec: IERK methods} coupled with the ATS-LDLB strategy for Example \ref{ex: ATS3comparison}. We will investigate the influence of method parameter choices on numerical performance, with a specific focus on the error behavior near the on-off point of external forcing.
	
	\begin{figure}[htb!]
		\centering
		\subfigure[$\tau_{n}$ with $\tau_{\max}=1$]{
			\includegraphics[width=0.3\textwidth]{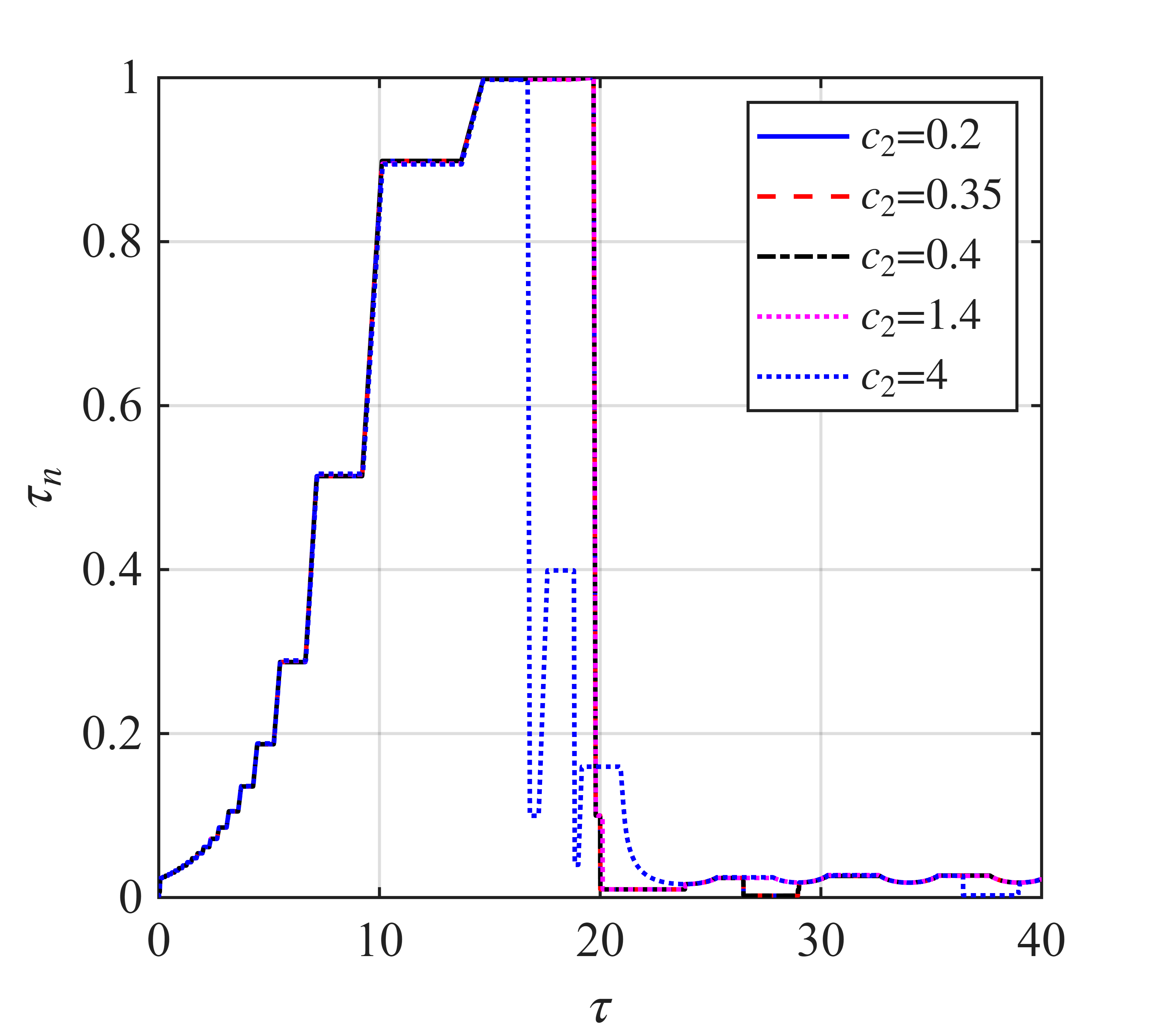}
		} 
		\subfigure[$\tau_{n}$ with $\tau_{\max}=0.5$]{
			\includegraphics[width=0.3\textwidth]{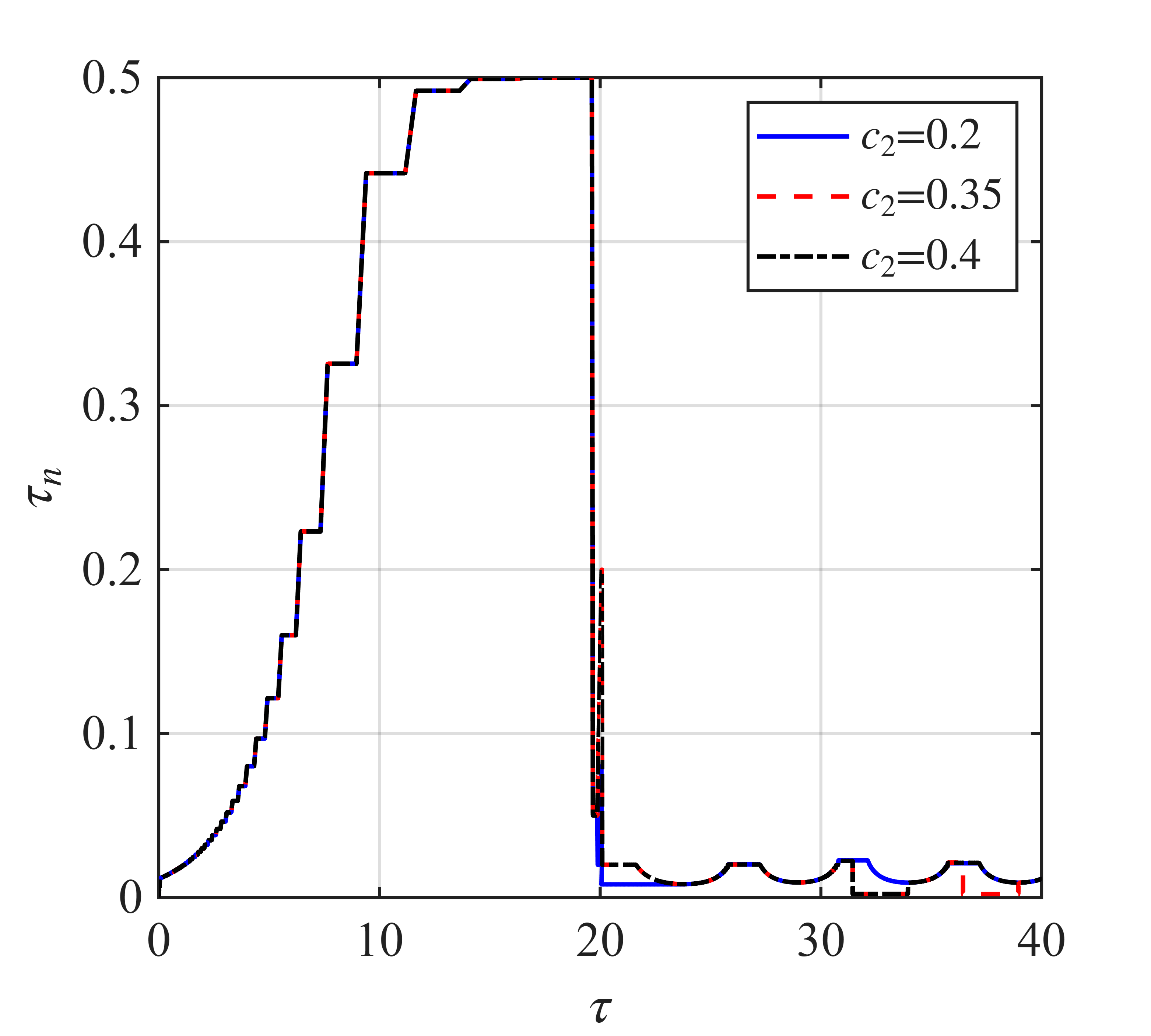}
		} 
		\subfigure[$\tau_{n}$ with $\tau_{\max}=0.1$]{
			\includegraphics[width=0.3\textwidth]{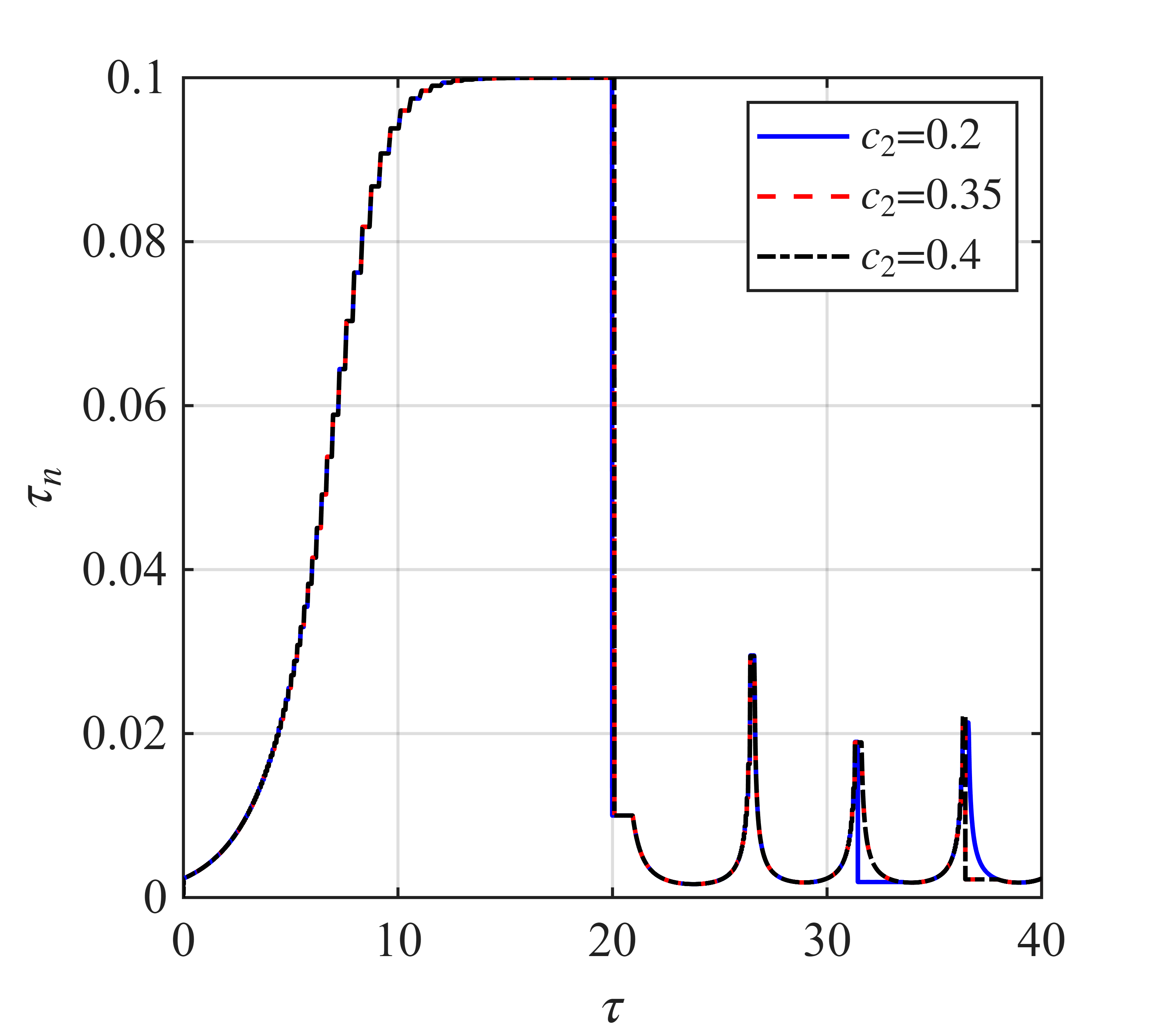}
		}
		\subfigure[$\|e^{n}_{\text{mix}}\|_{\infty}$ with $\tau_{\max}=1$]{
			\includegraphics[width=0.3\textwidth]{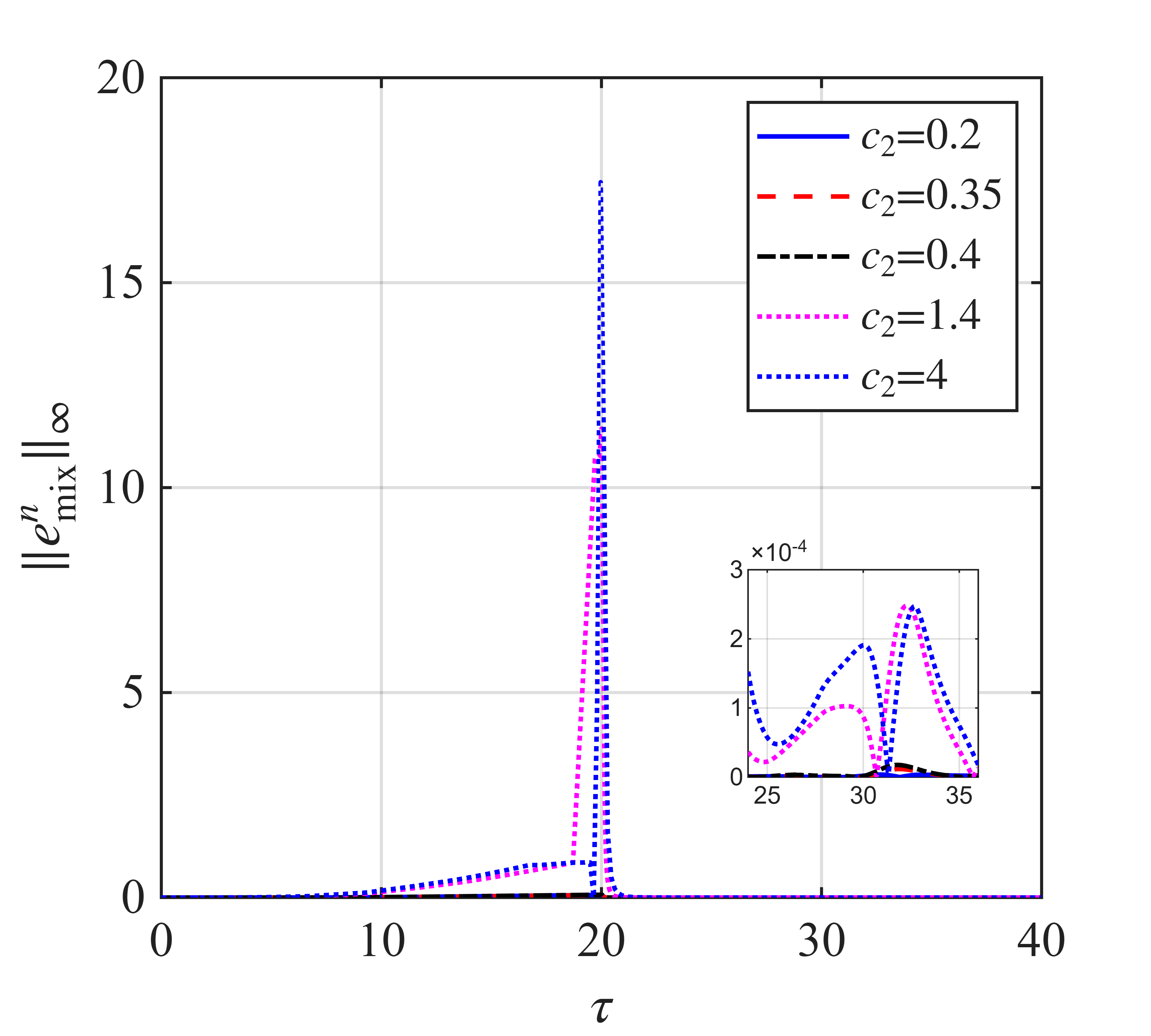}
		} 
		\subfigure[$\|e^{n}_{\text{mix}}\|_{\infty}$ with $\tau_{\max}=0.5$]{
			\includegraphics[width=0.3\textwidth]{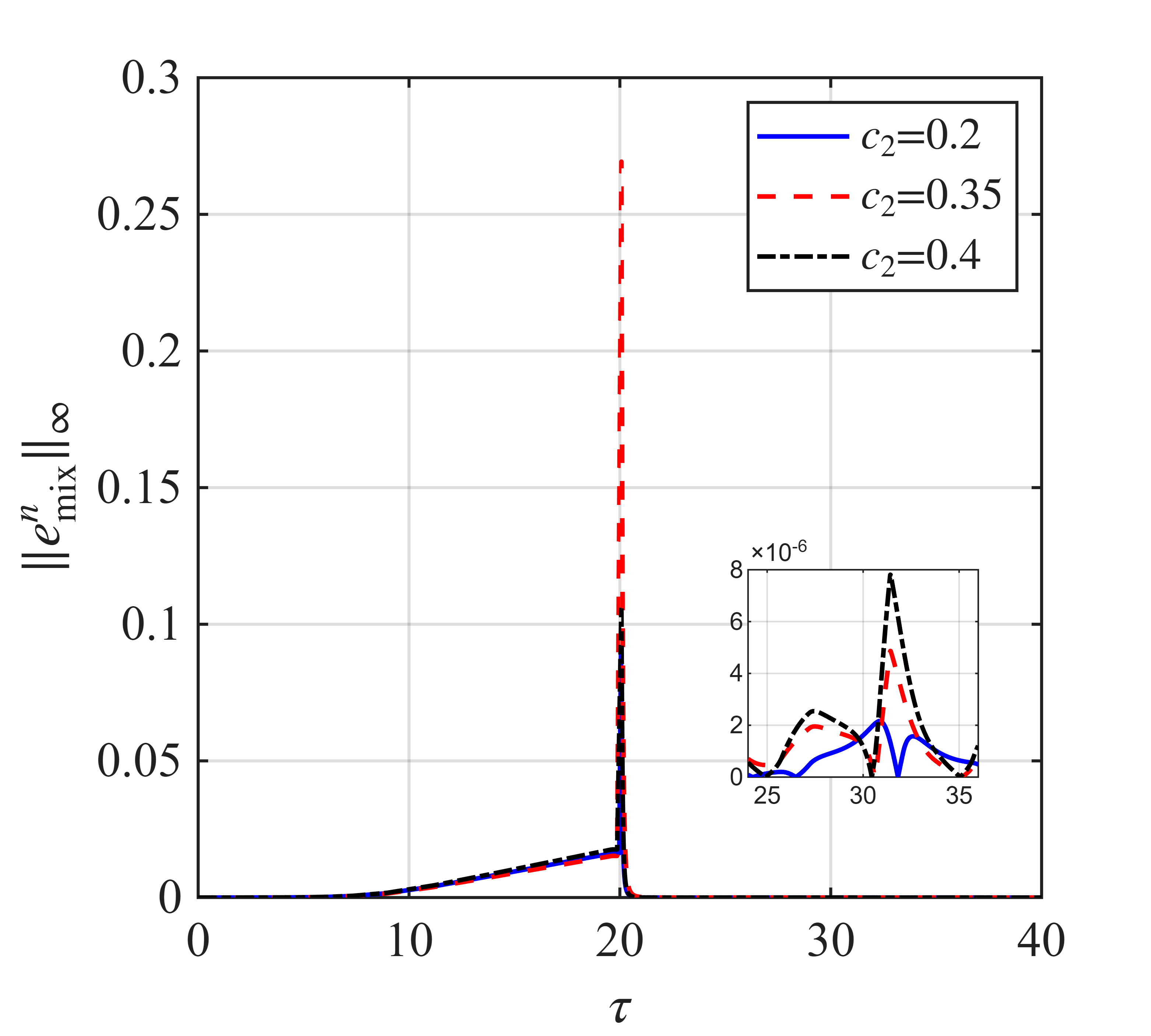}
		} 
		\subfigure[$\|e^{n}_{\text{mix}}\|_{\infty}$ with $\tau_{\max}=0.1$]{
			\includegraphics[width=0.3\textwidth]{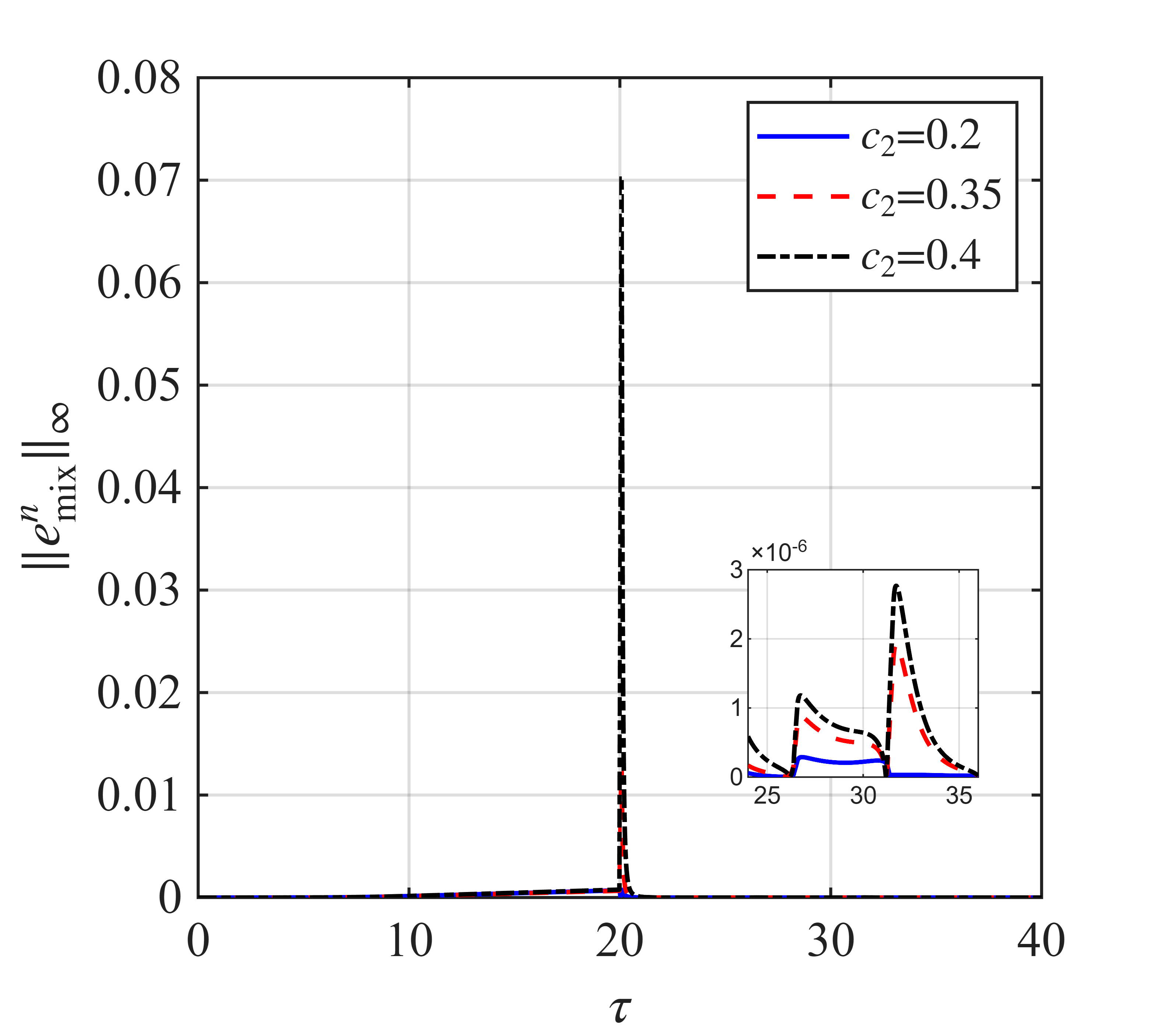}
		}
		\caption{Numerical simulations of Example \ref{ex: ATS3comparison} using IERK(2,3;$c_2$) schemes with ATS-LDLB strategy for $\tau_{\max}=1$ (left), $\tau_{\max}=0.5$ (middle), and $\tau_{\max}=0.1$ (right).}
		\label{fig: ex2-Radau2-para}
	\end{figure}
	
	\paragraph{Second-order IERK methods}
	We run the second-order IERK(2,3;$c_2$) methods for Example \ref{ex: ATS3comparison} with five different parameters $c_2 = 0.2, 0.35, 0.4, 1.4,$ and $4$, see Fig.~\ref{fig: ex2-Radau2-para}. The enstrophy curves for all tested parameters align well with the reference solution (plots omitted for brevity). For different $\tau_{\max}$ settings, the adaptive time-steps remain relatively smooth within the oscillatory interval $[T_1, T]$, as illustrated in Figs.~\ref{fig: ex2-Radau2-para} (a)--(c). Notably, the step size rapidly decreases at the forcing activation time $t = T_1$ to resolve the sudden transition of solution. However, the error curves in Figs.~\ref{fig: ex2-Radau2-para} (d)--(f) reveal that the mixed error $\|e^n_{\text{mix}}\|_{\infty}$ exhibits a significant peak at the transition point.
	
	Further inspections indicate that for the cases $c_2 = 1.4$ and $c_2 = 4$, the numerical errors near $t = T_1$ exceed $1.0$ regardless of the values of $\tau_{\max}$ (for clarity, the errors for $c_2 = 1.4$ and $c_2 = 4$ are only plotted for the case $\tau_{\max} = 1$). From the perspective of the scheme construction, the intermediate stages of the IERK(2,3;$c_2$) method require the forcing term at $t_n + c_2 \tau_n$. When $c_2 > 1$, the point falls outside the current time interval $[t_n, t_{n+1}]$, causing the scheme to ``pre-emptively'' perceive the activation of the forcing term before the transition actually occurs. This look-ahead sampling may disrupt the smoothness of numerical solution and introduces substantial numerical impulses at the interface. As a result, the adaptive strategy attempts to suppress these errors by drastically reducing the step size before reaching $t = T_1$, leading to undesirable  oscillations of adaptive steps, see Figure.~\ref{fig: ex2-Radau2-para} (a). In contrast, when $0 < c_2 \leq 1$, the errors are well-controlled within the oscillatory regime. At the moment of forcing activation, the error magnitudes vary slightly among the parameters; specifically, $c_2 = 0.35$ yields the smallest error for $\tau_{\max} = 1$, while $c_2 = 0.2$ performs best for the predetermined settings $\tau_{\max} = 0.5$ and $0.1$, as summarized in Table~\ref{table: Radau2-max-error}. 
	
	\begin{table}[htb!]
		\centering
		\begin{threeparttable}
			\centering 
			\renewcommand\arraystretch{1.3}
			\belowrulesep=0pt\aboverulesep=0pt
			\caption{Maximum errors $\max_{n} \|e^{n}_{\text{mix}}\|_{\infty}$ of IERK(2,3;$c_2$) for Example \ref{ex: ATS3comparison}.}
			\label{table: Radau2-max-error}
			\vspace{2mm}
			\begin{tabularx}{0.95\textwidth}{c@{\extracolsep{\fill}}ccccc}
				\toprule 
				& $c_{2}=0.2$ & $c_{2}=0.35$ & $c_{2}=0.4$ & $c_{2}=1.4$ & $c_{2}=4$ \\
				\midrule 
				$\tau_{\max}=1$ & 6.103e-2 & 5.759e-2 & 6.607e-2 & 11.44 & 17.45 \\[2pt]
				\midrule 
				$\tau_{\max}=0.5$ & 8.989e-2 & 2.693e-1 & 1.068e-1 & 3.412e-1 & 1.135  \\[2pt]
				\midrule 
				$\tau_{\max}=0.1$ & 7.632e-4 & 1.322e-2 & 7.040e-2 & 9.557e-1 & 6.262 \\
				\bottomrule
			\end{tabularx}
		\end{threeparttable}
	\end{table}
	
	\paragraph{Third- and Fourth-order IERK methods}
	
	We next evaluate the performance of the third-order IERK(3,5;$a_{55}$) schemes with $a_{55} \in \{0.7, 1.2, 1.8, 2.1\}$ and the fourth-order IERK(4,7;$\hat{a}_{43}$) schemes with $\hat{a}_{43} \in \{-5, -2, -0.8, 2, 5\}$ under the ATS-LDLB strategy. The resulting mixed error evolutions are presented in Figure~\ref{fig: ex2-Radau3&4-para}. For the sake of brevity, the plots of adaptive steps are omitted, as the influence of the method parameters is sufficiently characterized by the error behaviors.
	

	\begin{figure}[htb!]
		\centering
		\subfigure[$\tau_{\max}=1$]{
			\includegraphics[width=0.3\textwidth]{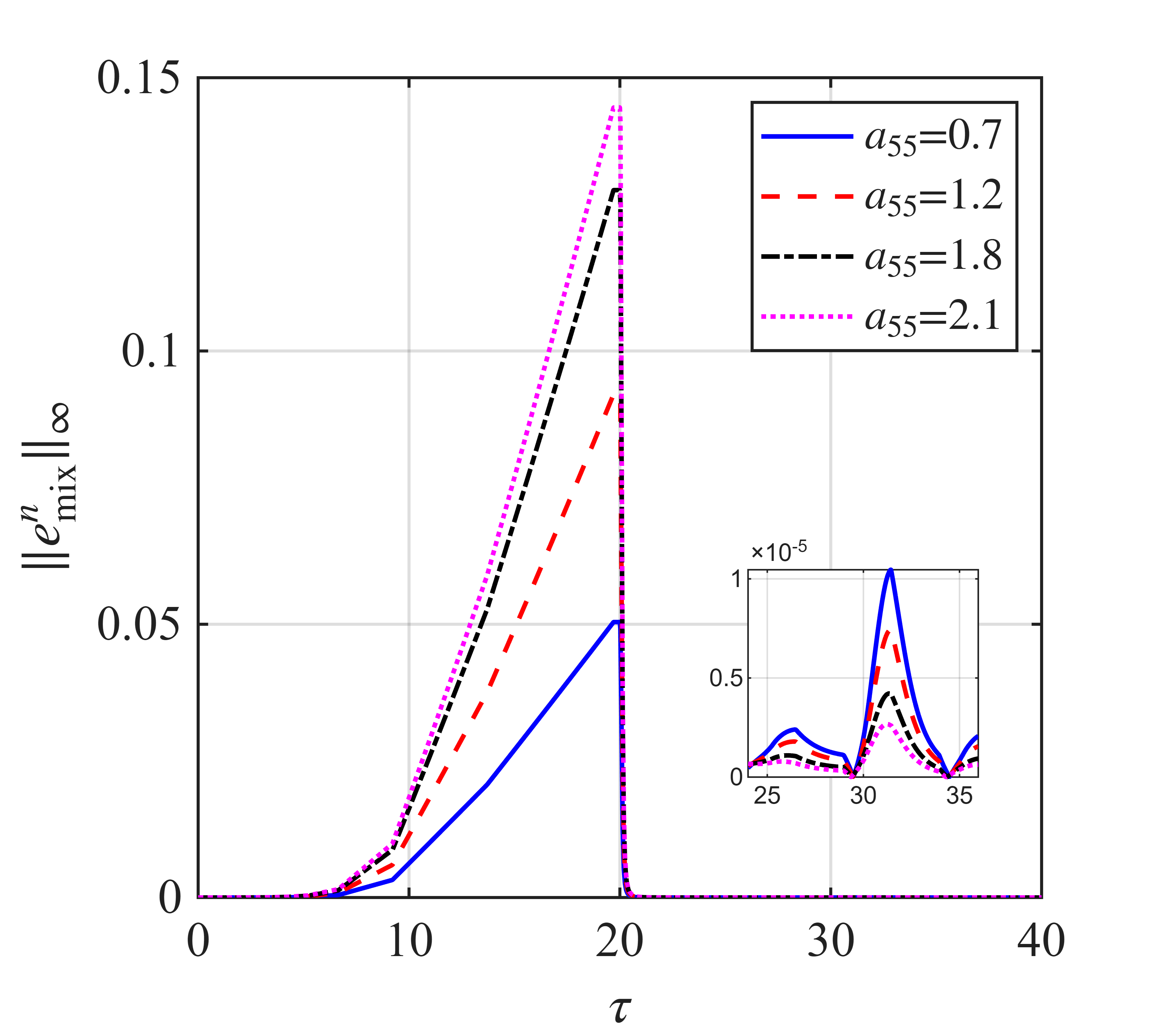}
		} 
		\subfigure[$\tau_{\max}=0.5$]{
			\includegraphics[width=0.3\textwidth]{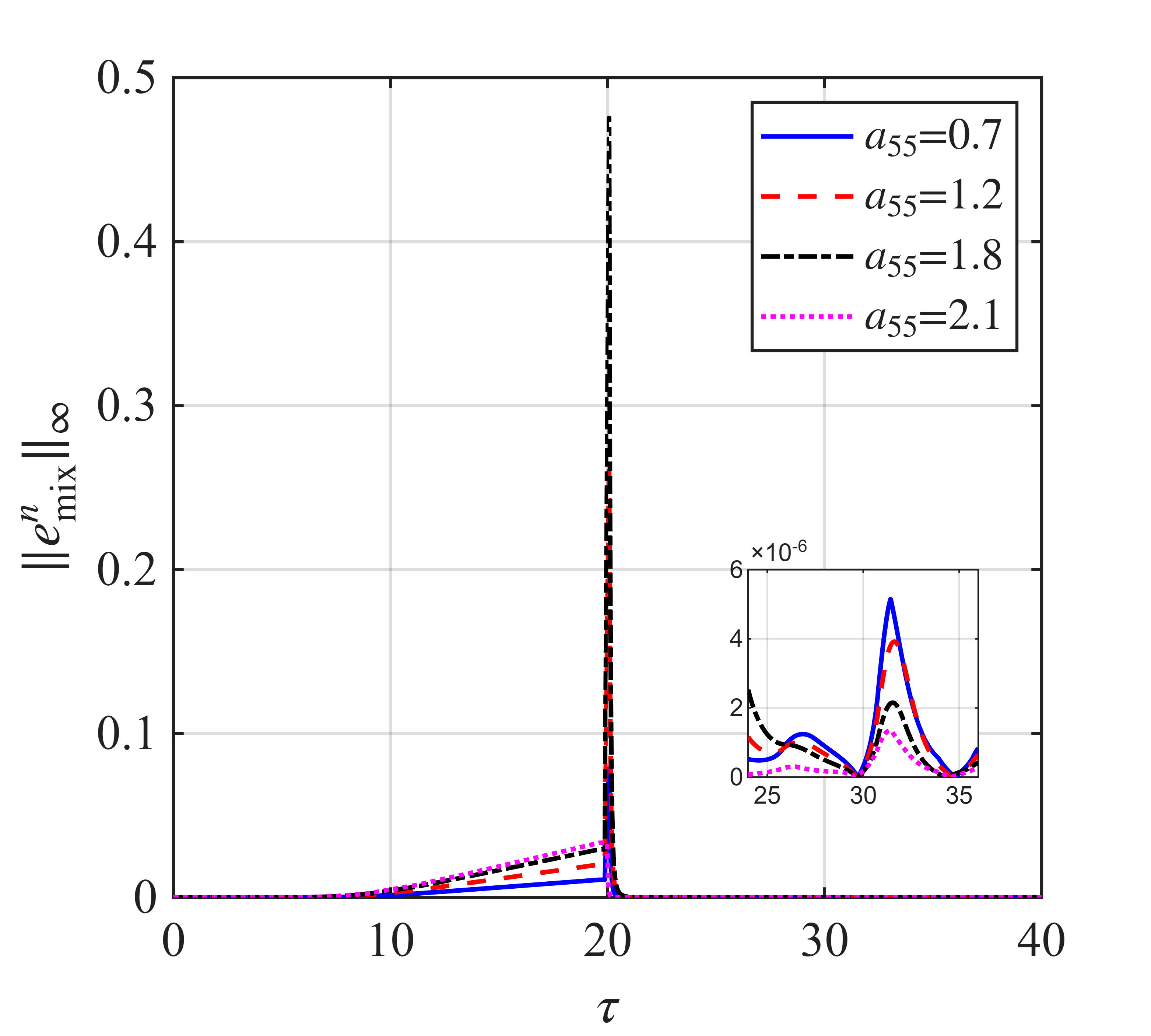}
		} 
		\subfigure[$\tau_{\max}=0.1$]{
			\includegraphics[width=0.3\textwidth]{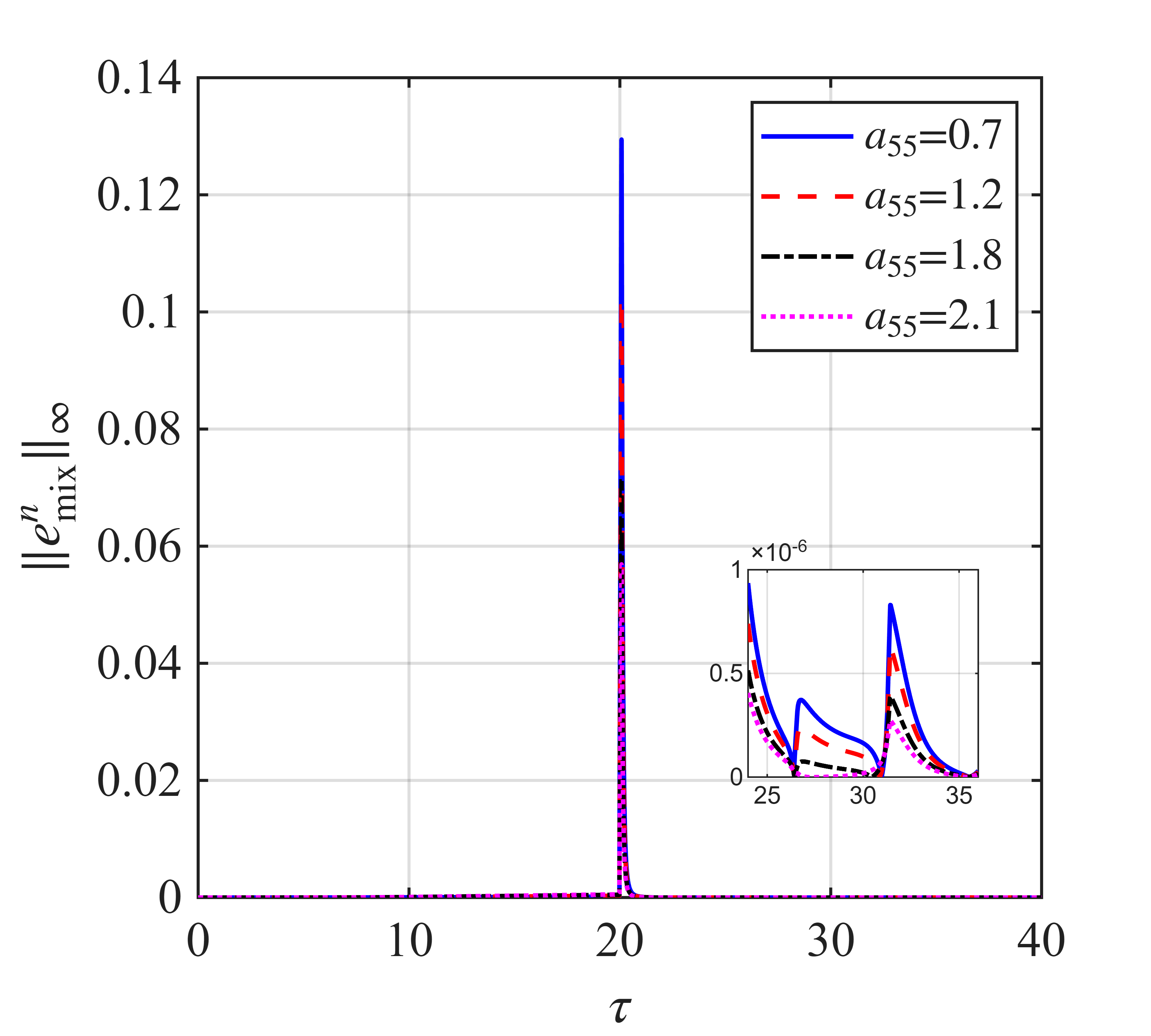}
		}
		\subfigure[ $\tau_{\max}=1$]{
			\includegraphics[width=0.3\textwidth]{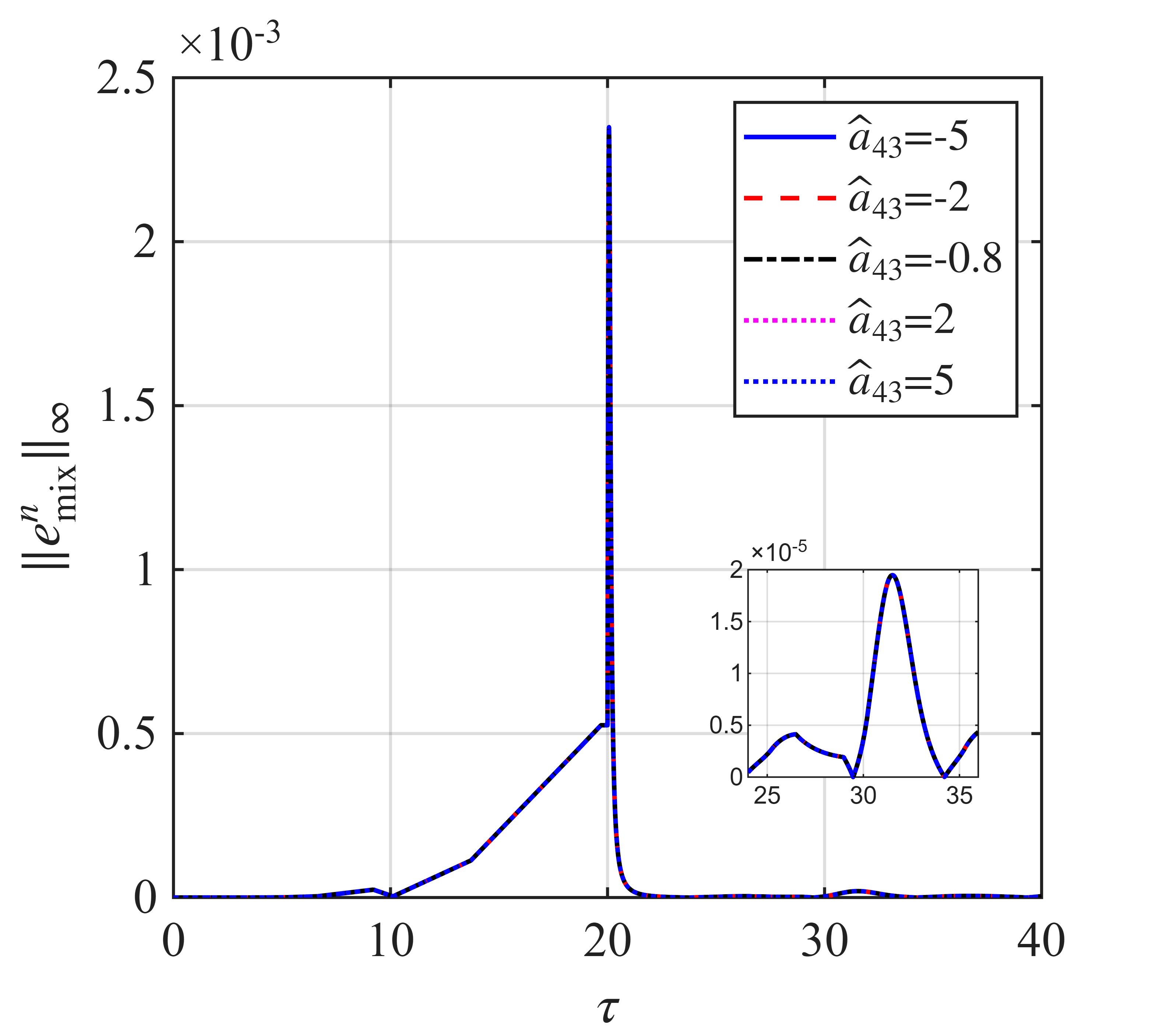}
		} 
		\subfigure[$\tau_{\max}=0.5$]{
			\includegraphics[width=0.3\textwidth]{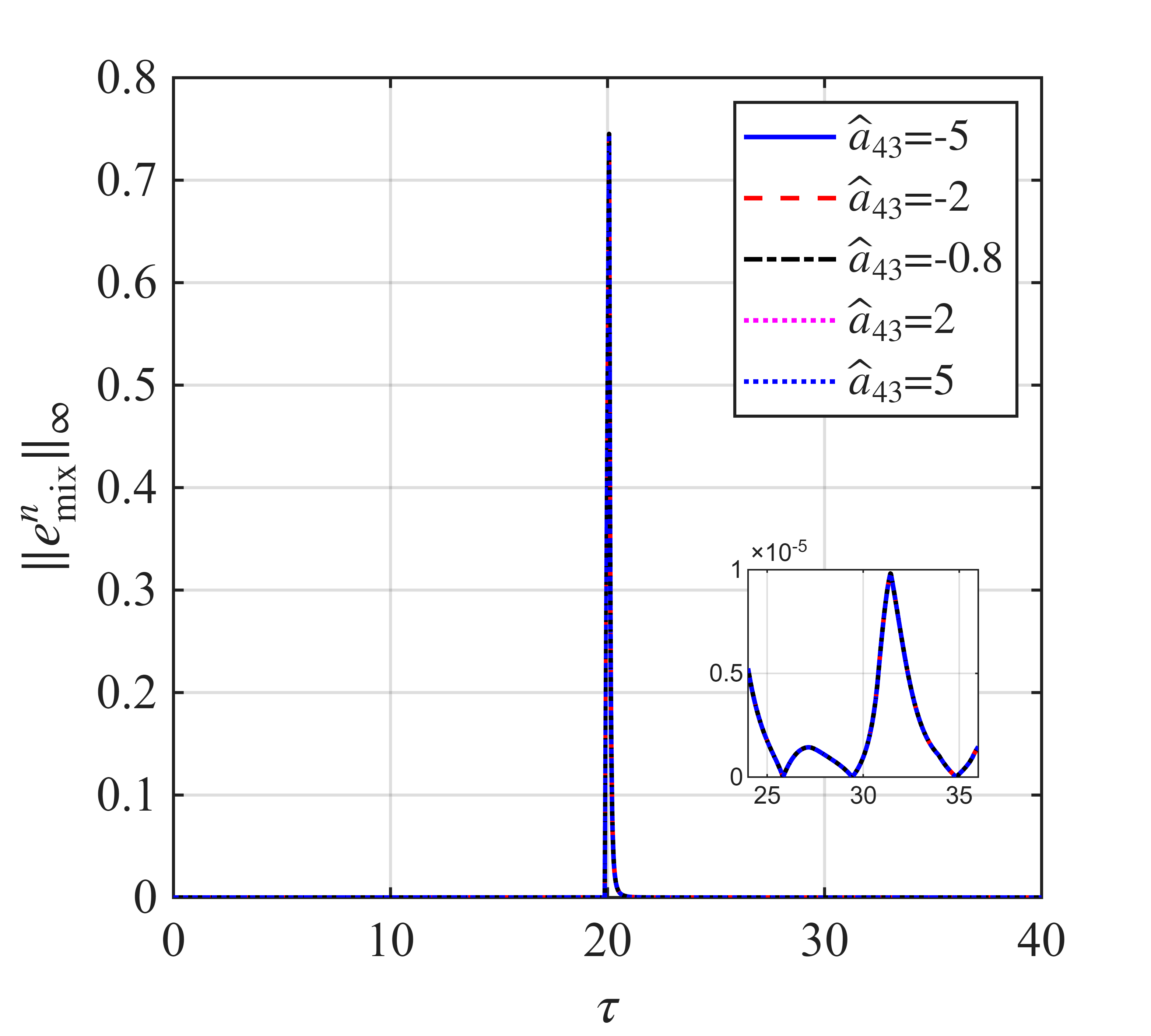}
		} 
		\subfigure[ $\tau_{\max}=0.1$]{
			\includegraphics[width=0.3\textwidth]{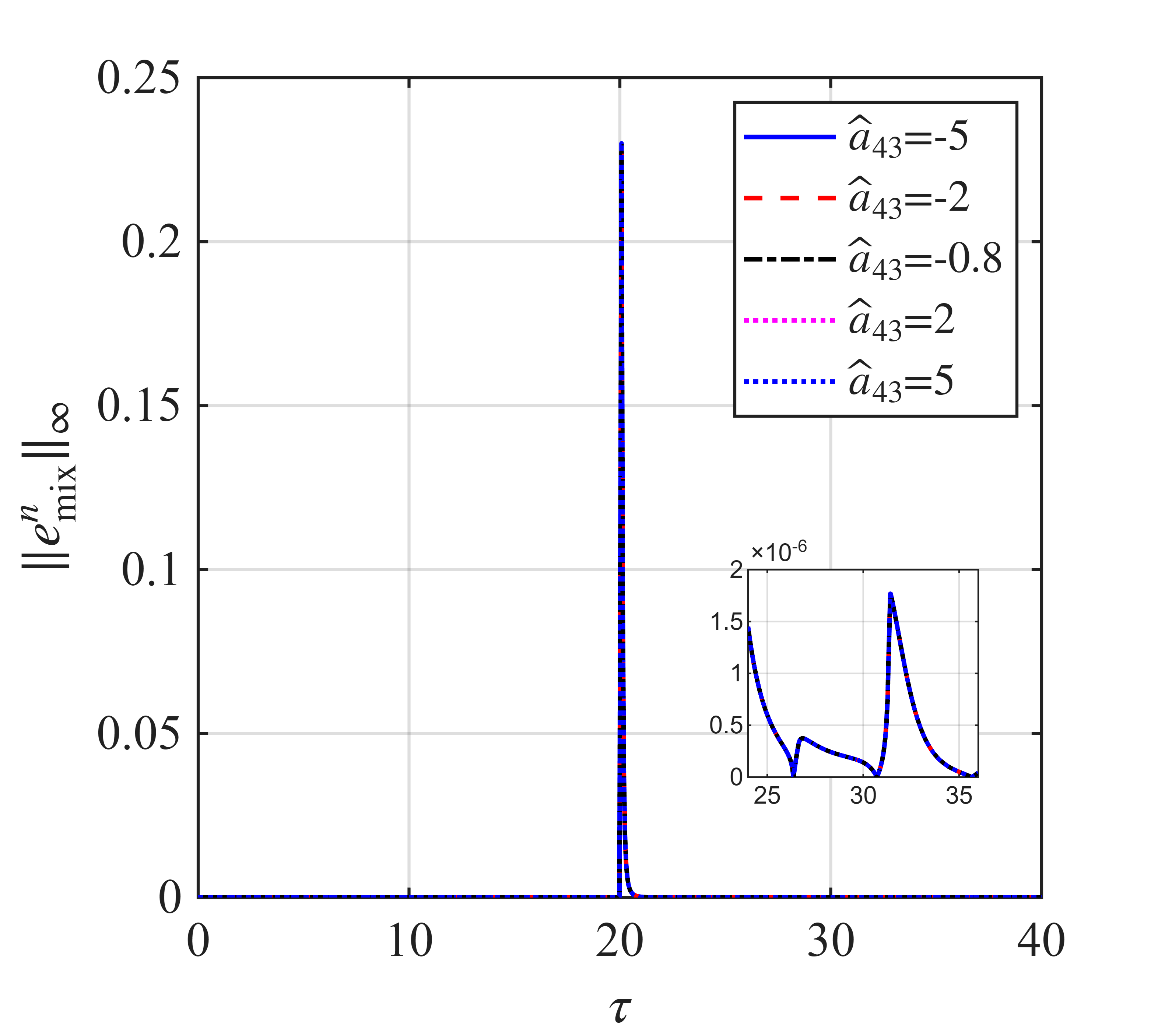}
		}
		\caption{The mixed errors $\|e^{n}_{\text{mix}}\|_{\infty}$ of IERK(3,5;$a_{55}$) (top) and IERK(4,7;$\hat{a}_{43}$) (bottom) schemes with the ATS-LDLB strategy for Example \ref{ex: ATS3comparison}.}
		\label{fig: ex2-Radau3&4-para}
	\end{figure}

	\begin{table}[htb!]
		\centering
		\begin{threeparttable}
			\centering 
			\renewcommand\arraystretch{1.3}
			\belowrulesep=0pt\aboverulesep=0pt
			\caption{Maximum errors $\max_{n} \|e^{n}_{\text{mix}}\|_{\infty}$ of IERK(3,5;$a_{55}$) for Example \ref{ex: ATS3comparison}.}
			\label{table: Radau3-max-error}
			\vspace{2mm}
			\begin{tabularx}{0.95\textwidth}{c@{\extracolsep{\fill}}cccc}
				\toprule 
				& $a_{55}=0.7$ & $a_{55}=1.2$ & $a_{55}=1.8$ & $a_{55}=2.1$ \\
				\midrule 
				$\tau_{\max}=1$ & 5.040e-2 & 9.204e-2 & 1.294e-1 & 1.446e-1 \\[2pt]
				\midrule 
				$\tau_{\max}=0.5$ & 7.871e-2 & 2.692e-1 & 4.755e-1 & 3.337e-2 \\[2pt]
				\midrule 
				$\tau_{\max}=0.1$ & 1.294e-1 & 1.024e-1 & 7.165e-2 & 5.689e-2 \\
				\bottomrule
			\end{tabularx}
		\end{threeparttable}
	\end{table}
	
	Inside the oscillatory interval $t\in(T_1, T]$, both the third- and fourth-order schemes consistently exhibit low and comparable errors for all tested $\tau_{\max}$. However, notable distinctions again emerge near the on-off point $T_1$. The IERK(4,7;$\hat{a}_{43}$) schemes have remarkable robustness, with the error magnitude remaining nearly insensitive to the variations of $\hat{a}_{43}$. In contrast, the IERK(3,5;$a_{55}$) schemes show a pronounced sensitivity to $a_{55}$, see Figure \ref{fig: ex2-Radau3&4-para}. Recall that the cases $a_{55}=0.7$ and $\hat{a}_{43}=2$ arrive at the minimum errors for different time-steps in the convergence tests with smooth solutions, cf. Figure \ref{fig: accuracy of other-IERK}. While these values remain highly effective in the smooth oscillatory regions, the non-smooth transition at $T_1$ introduces a different numerical dynamics. Specifically, for $\tau_{\max}=1$, the case $a_{55}=0.7$ still provides the highest precision. However, see Table~\ref{table: Radau3-max-error}, for smaller settings $\tau_{\max}=0.5$ and 0.1, the IERK(3,5;$a_{55}$) method with $a_{55}=2.1$ becomes the most accurate. This remains mysterious to us.
	
	More surprisingly, the maximum errors of both methods do not always decrease as the maximum step setting $\tau_{\max}$ decreases. For IERK(3,5;$a_{55}$), the peak error near $T_1$ for $\tau_{\max}=0.5$ is actually larger than that for $\tau_{\max}=1$, while further reducing the setting $\tau_{\max}=0.1$ only leads to a slight improvement over the case of $\tau_{\max}=1$. A similar behavior is observed for IERK(4,7;$\hat{a}_{43}$): the smallest error appear for the case $\tau_{\max}=1$, followed by $\tau_{\max}=0.1$, while the largest error occurs for $\tau_{\max}=0.5$. The underlying mechanisms for this phenomenon require further investigations.
	
	\subsection{Robustness of ATS-LDLB strategy}
	
	Here we compare the ATS with ATS-LDLB strategies using the IERK(4,7;$\hat{a}_{43}$) scheme with $\hat{a}_{43}=2$. By introducing two different groups of high-frequency oscillations, we will test whether the ATS-LDLB strategy can generate accurate solutions when the forcing changes rapidly.

	\begin{example}\label{ex: MultiFreqTest}
		We extend the configuration of Example \ref{ex: ATS3comparison} by constructing a temporal weight function $g_2(t)$ consisting of three pulses with different frequencies:
		\[
		g_{2} (t) = \begin{cases}
			\sin^{2}\bra{\tfrac{2\pi l_{t, j} (t-T_{2j-1}) }{T_{2j}-T_{2j-1}}}, & t \in (T_{2j-1}, T_{2j}], \quad j=1,2,3, \\
			0, & \text{otherwise.}
		\end{cases}
		\]
		The pulse intervals are defined by the points $T_0=0, T_1=20, T_2=40, T_3=70, T_4=80, T_5=90$, and $T=T_6=120$. The solution has the form $\omega (x, y, t) = f (t) \sin(l_x x)$, where the temporal evolution $f(t)$ satisfies $f'(t)+\nu l_x^2 f(t) = g_2(t)$. The piecewise analytical expression of $f(t)$ is given by
		\[
		f (t) = \begin{cases}
			\displaystyle f (T_{2j-2}) \ze^{-\nu l_x^2 (t-T_{2j-2})}, & t \in (T_{2j-2},T_{2j-1}], \\
			\displaystyle f (T_{2j-1}) \ze^{-\nu l_x^2 (t-T_{2j-1}) } + \ze^{-\nu l_x^2 (t-T_{2j-1}) } \mathcal{J}_{2,j}, & t \in (T_{2j-1},T_{2j}],
		\end{cases}
		\]
		for $j=1,2,3$, where $\mathcal{J}_{2,j}:=\mathcal{J}_{2}\braB{ \nu l_x^2, \tfrac{2\pi l_{t, j}}{T_{2j}-T_{2j-1}}, t-T_{2j-1}}$.
		To ensure time-periodicity such that $f(0) = f(T)$, the initial value is prescribed as 
		\[
		f (0) = \frac{\sum_{j=1}^{3} \ze^{-\nu l_x^2 (T_{6} - T_{2j-1}) } \mathcal{J}_2 (\nu l_x^2, \frac{2\pi l_{t, j}}{T_{2j}-T_{2j-1}}, T_{2j}-T_{2j-1}) }{1 - \ze^{-\nu l_x^2 T_{6}}}.
		\]
		With the spatial wavenumber $l_x=1$, the pulse frequencies $\vec{l_t}:=(l_{t,1},l_{t,2},l_{t,3})$ will be determined later.  This problem features multiple on-off points of exterior force and different scales of oscillations, providing a challenging test for the practical effectiveness of adaptive time-stepping algorithms.
	\end{example}
	
	\begin{figure}[htb!]
		\centering
		\subfigure[$\|\omega^{n}\|^{2}$ for Freq-A]{
			\includegraphics[width=0.3\textwidth]{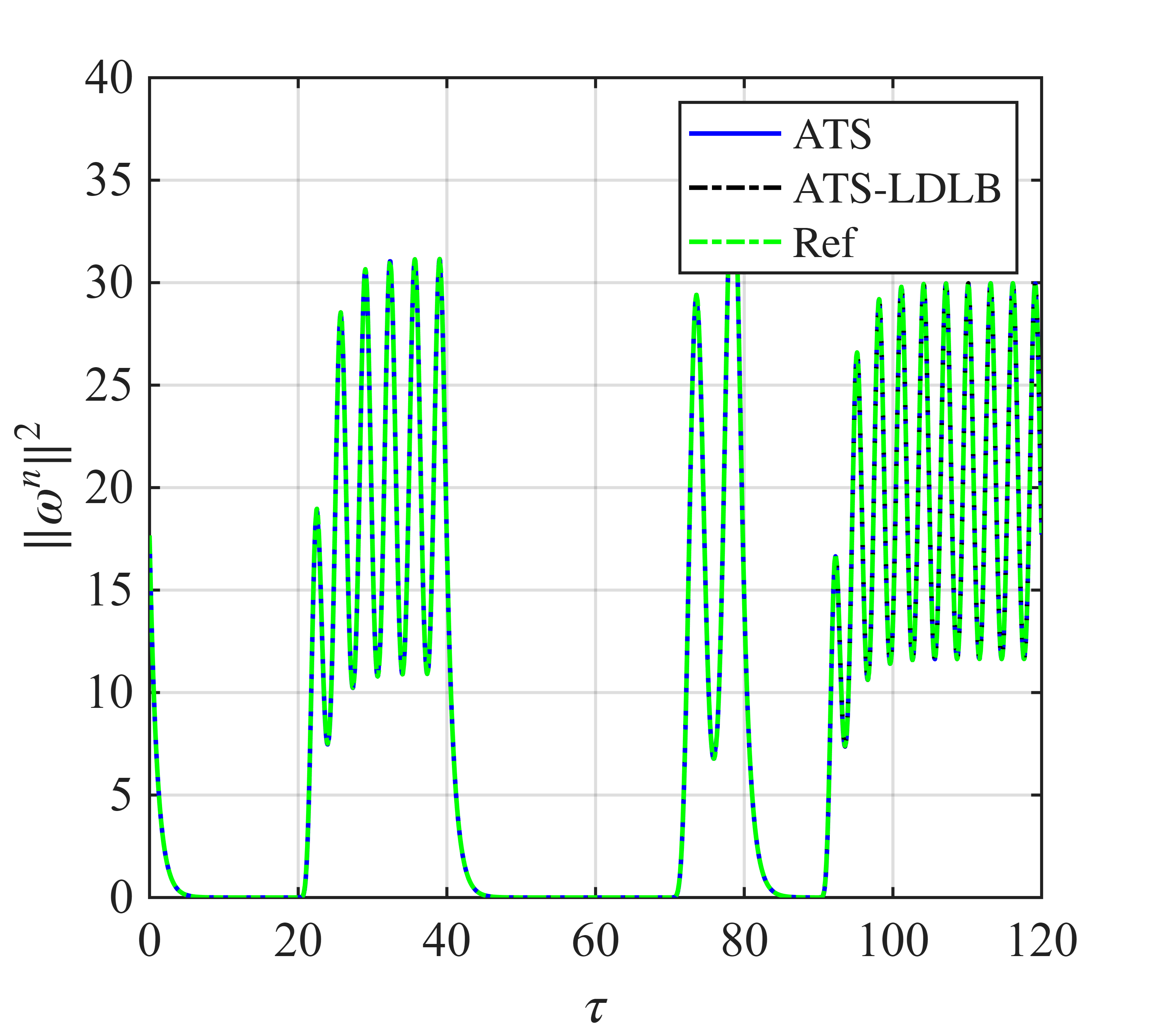}
		} 
		\subfigure[$\tau_{n}$ for Freq-A]{
			\includegraphics[width=0.3\textwidth]{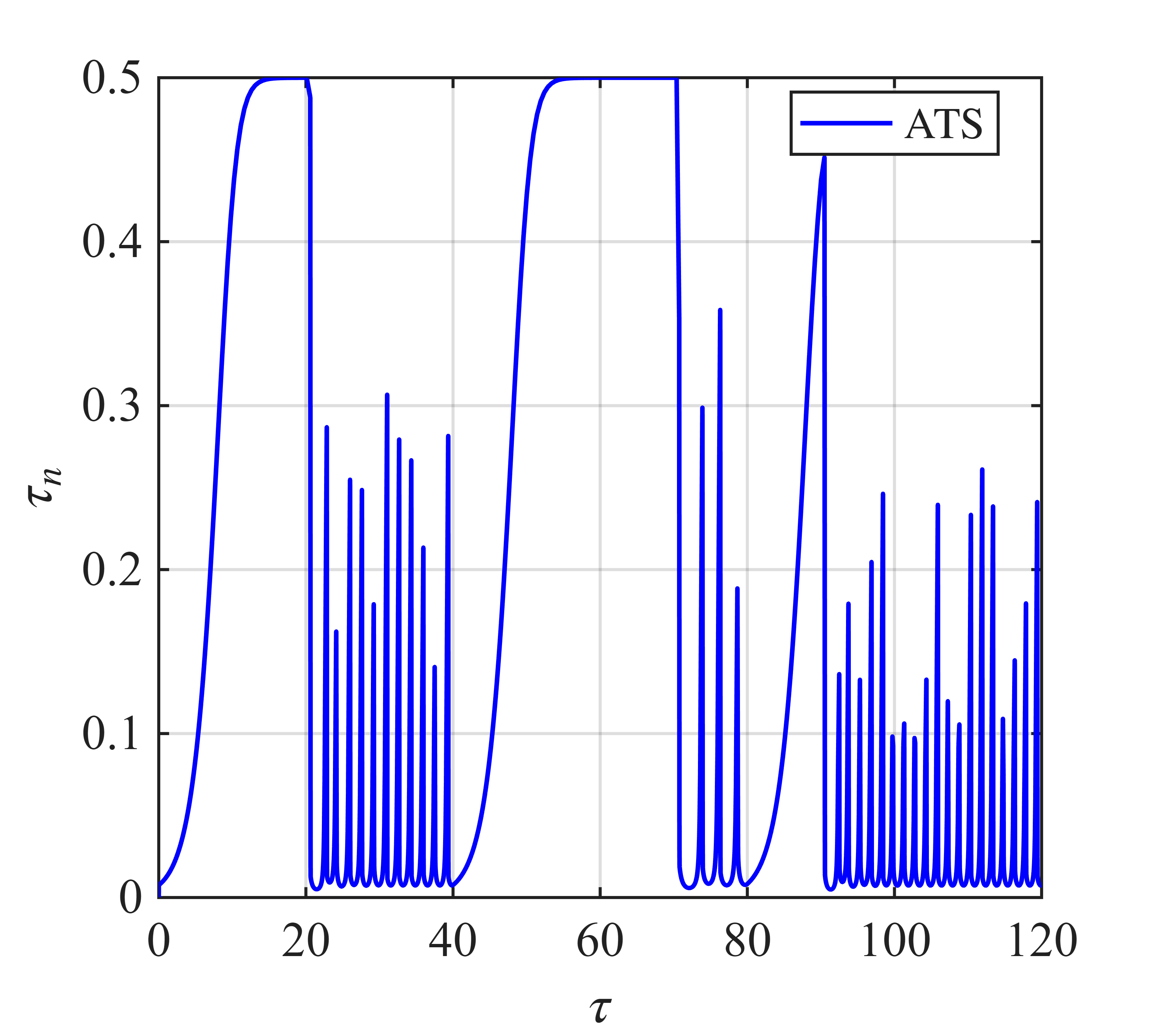}
		} 
		\subfigure[$\tau_{n}$ for Freq-A]{
			\includegraphics[width=0.3\textwidth]{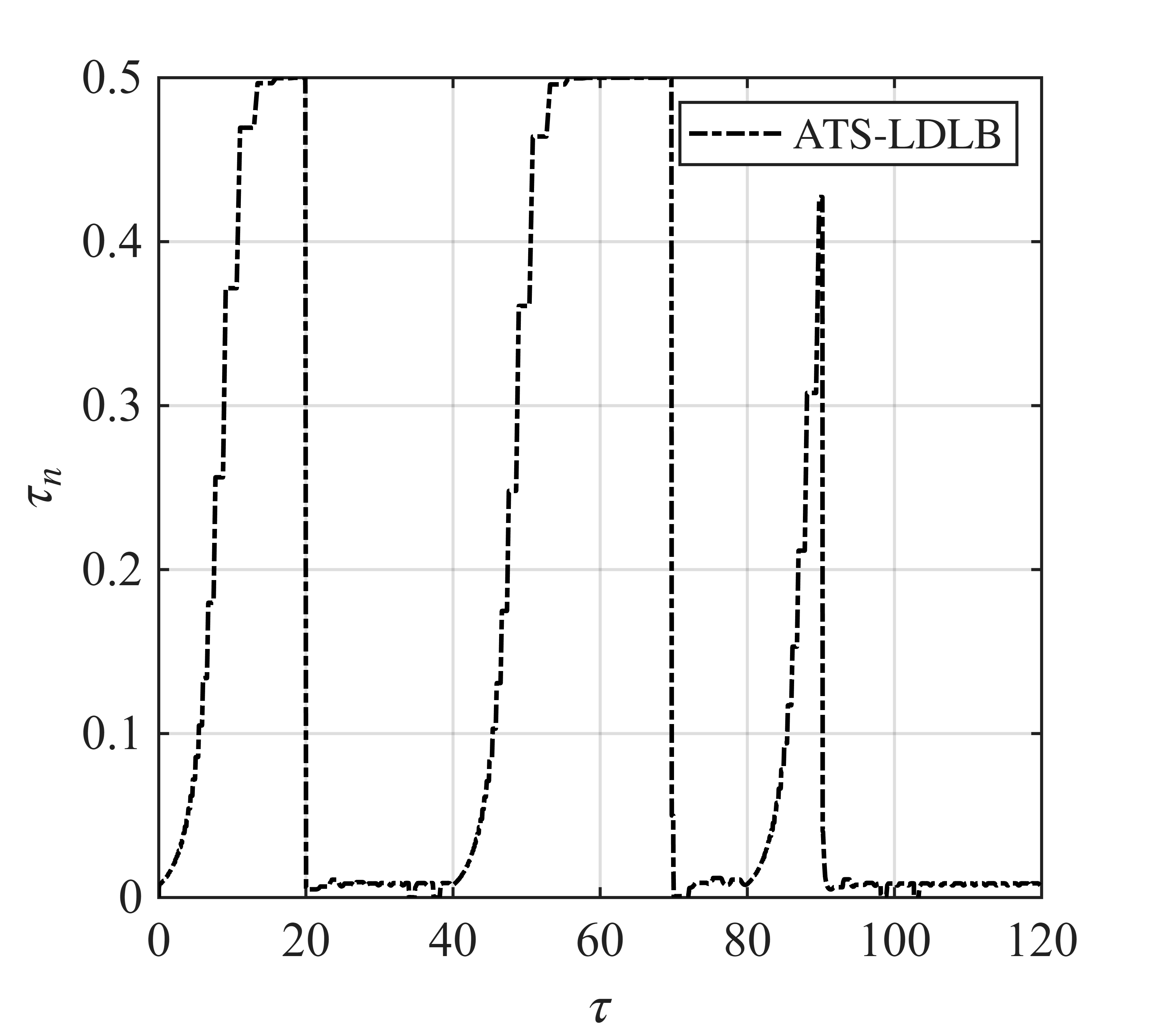}
		}
		\subfigure[$\|\omega^{n}\|^{2}$ for Freq-B]{
			\includegraphics[width=0.3\textwidth]{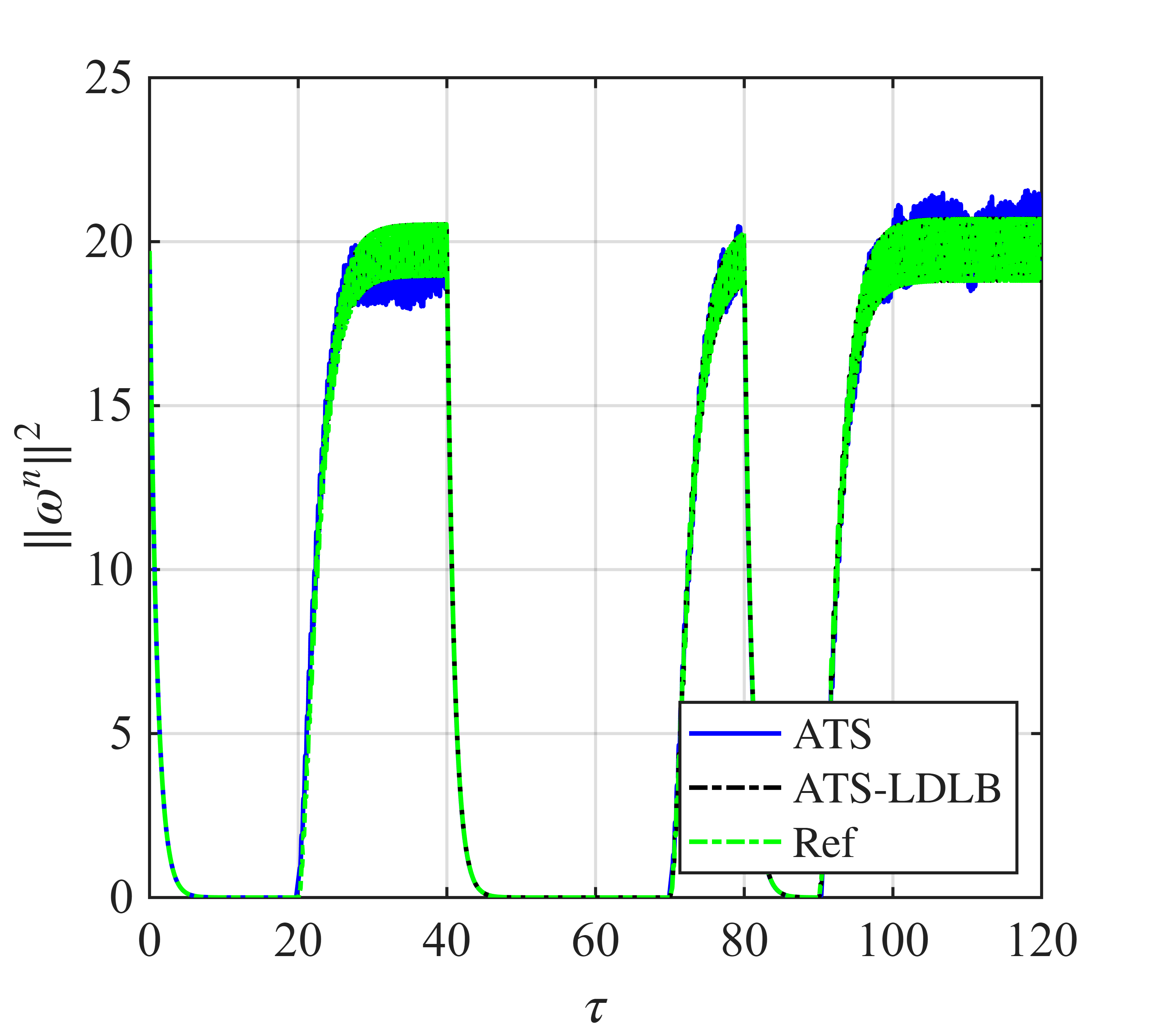}
		} 
		\subfigure[$\tau_{n}$ for Freq-B]{
			\includegraphics[width=0.3\textwidth]{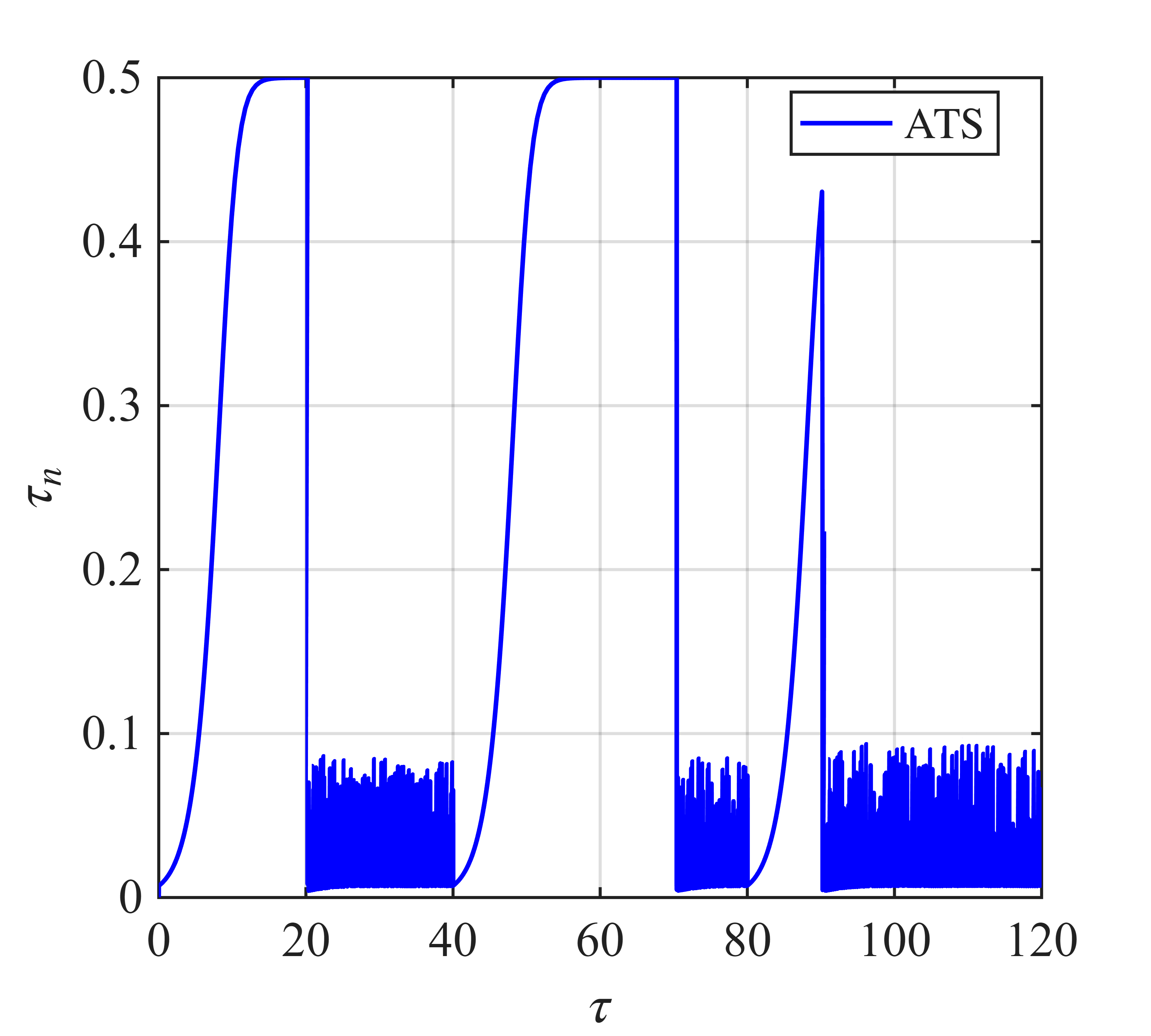}
		} 
		\subfigure[$\tau_{n}$ for Freq-B]{
			\includegraphics[width=0.3\textwidth]{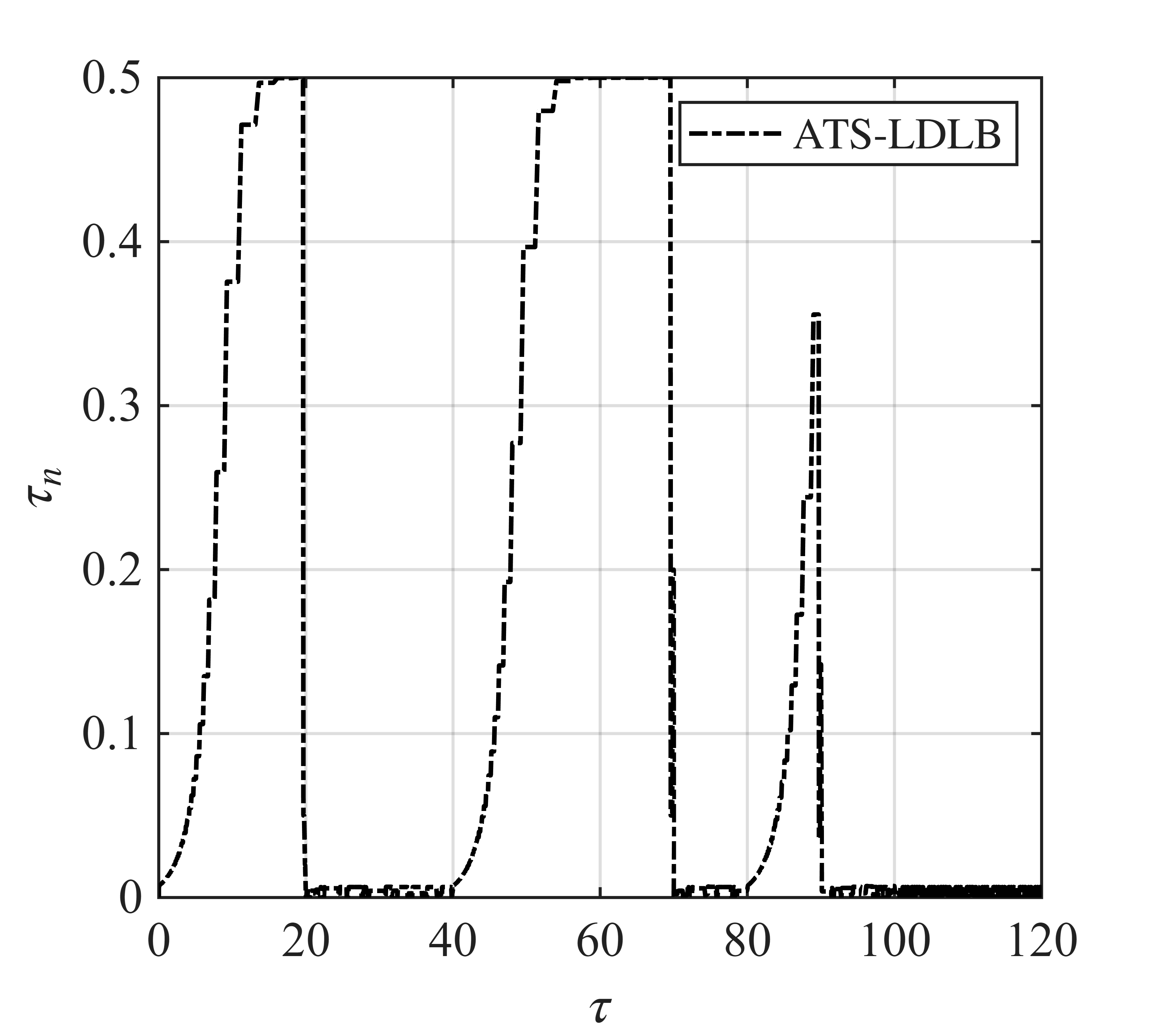}
		}
		\caption{Performance of IERK(4,7;2) method with ATS and ATS-LDLB strategies ($\tau_{\max}=0.5$) for two frequency settings: Freq-A (top) and Freq-B (bottom).}
		\label{fig: ex3-Radau4-sin2T6-dt=0.5}
	\end{figure}

		\begin{figure}[htb!]
		\centering
		\subfigure[$\|\omega^{n}\|^{2}$]{
			\includegraphics[width=0.3\textwidth]{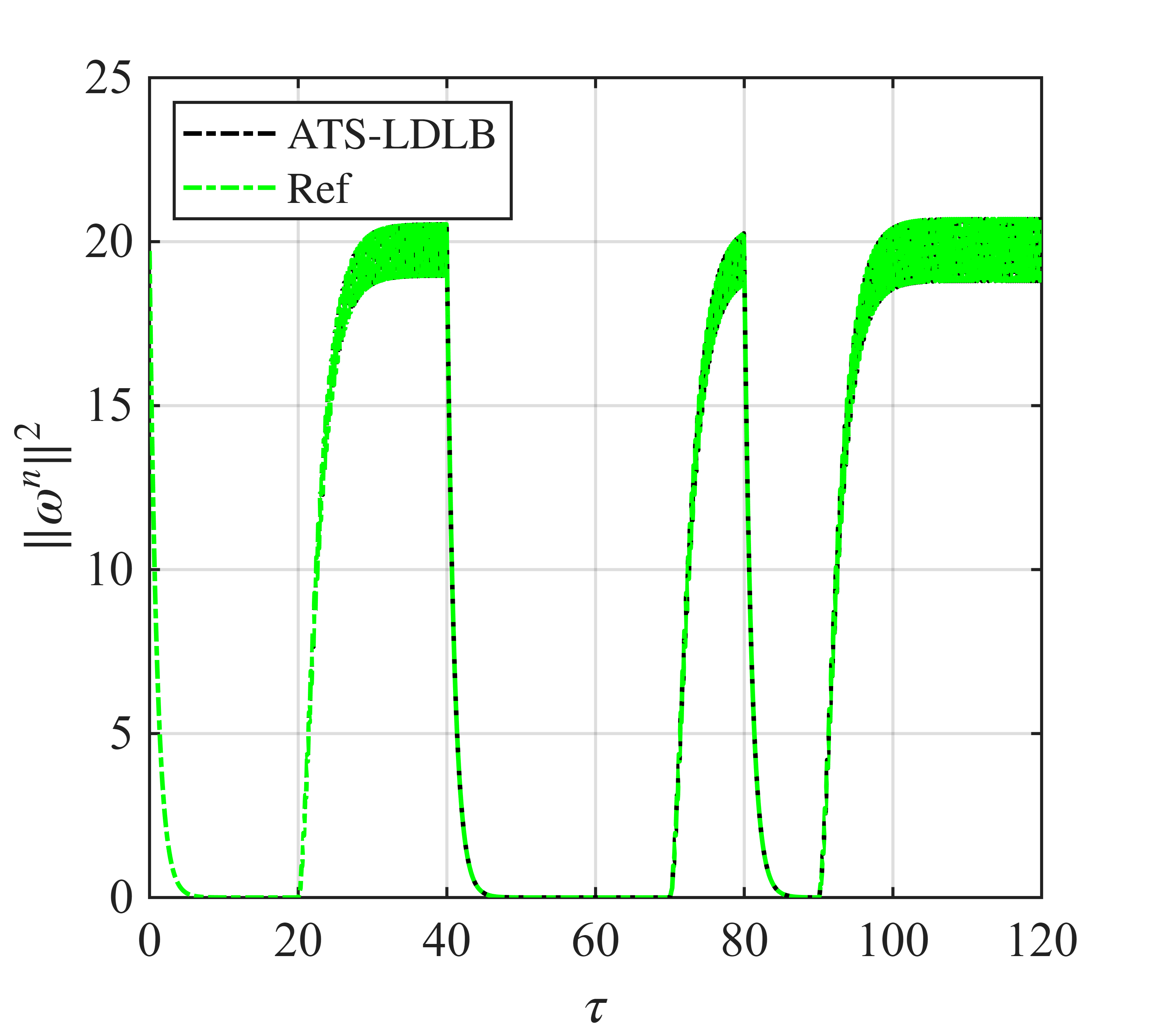}
		}\quad
		\subfigure[$\tau_{n}$]{
			\includegraphics[width=0.3\textwidth]{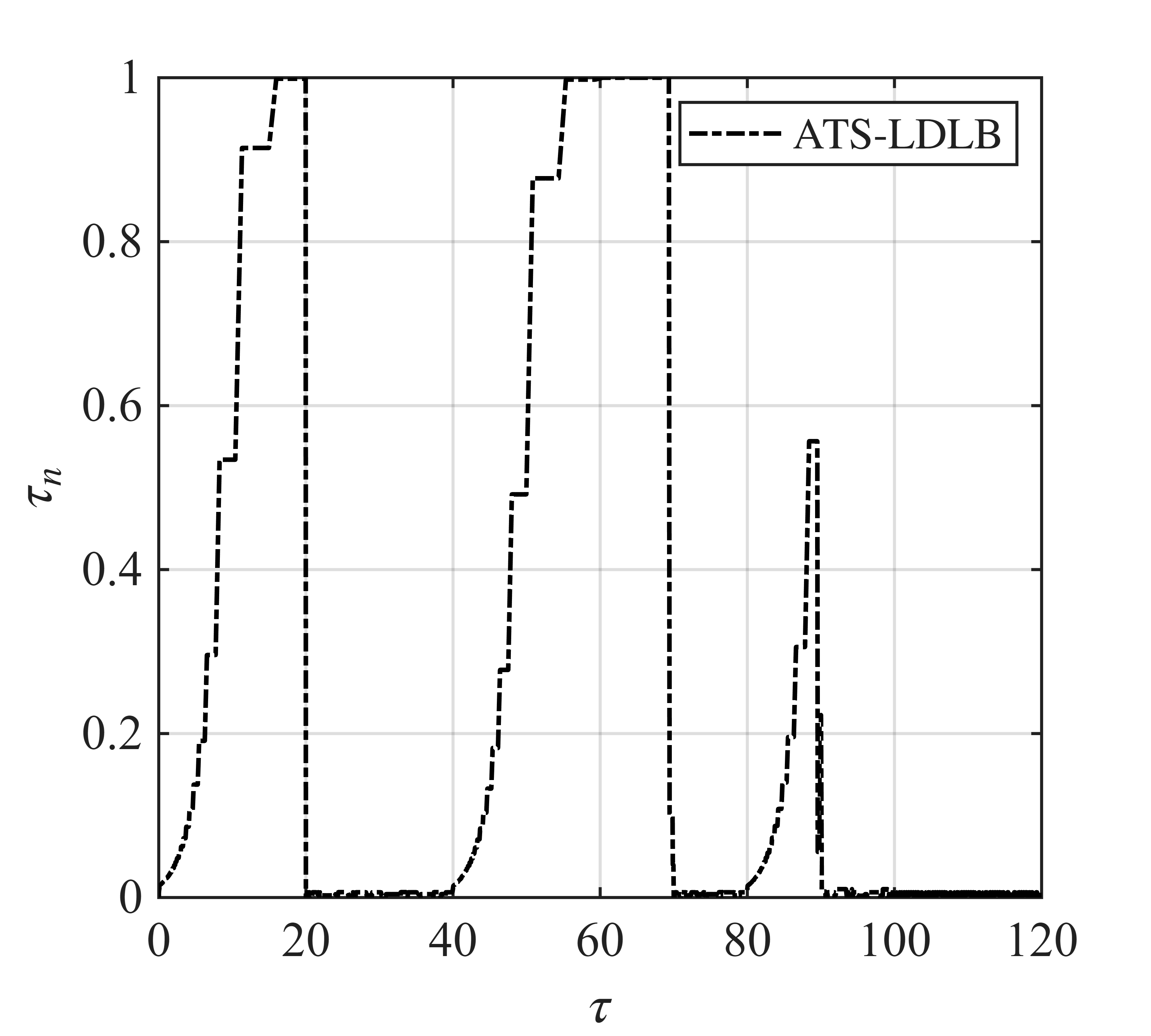}
		} 
		\subfigure[$\|e^{n}_{\text{mix}}\|_{\infty}$]{
			\includegraphics[width=0.3\textwidth]{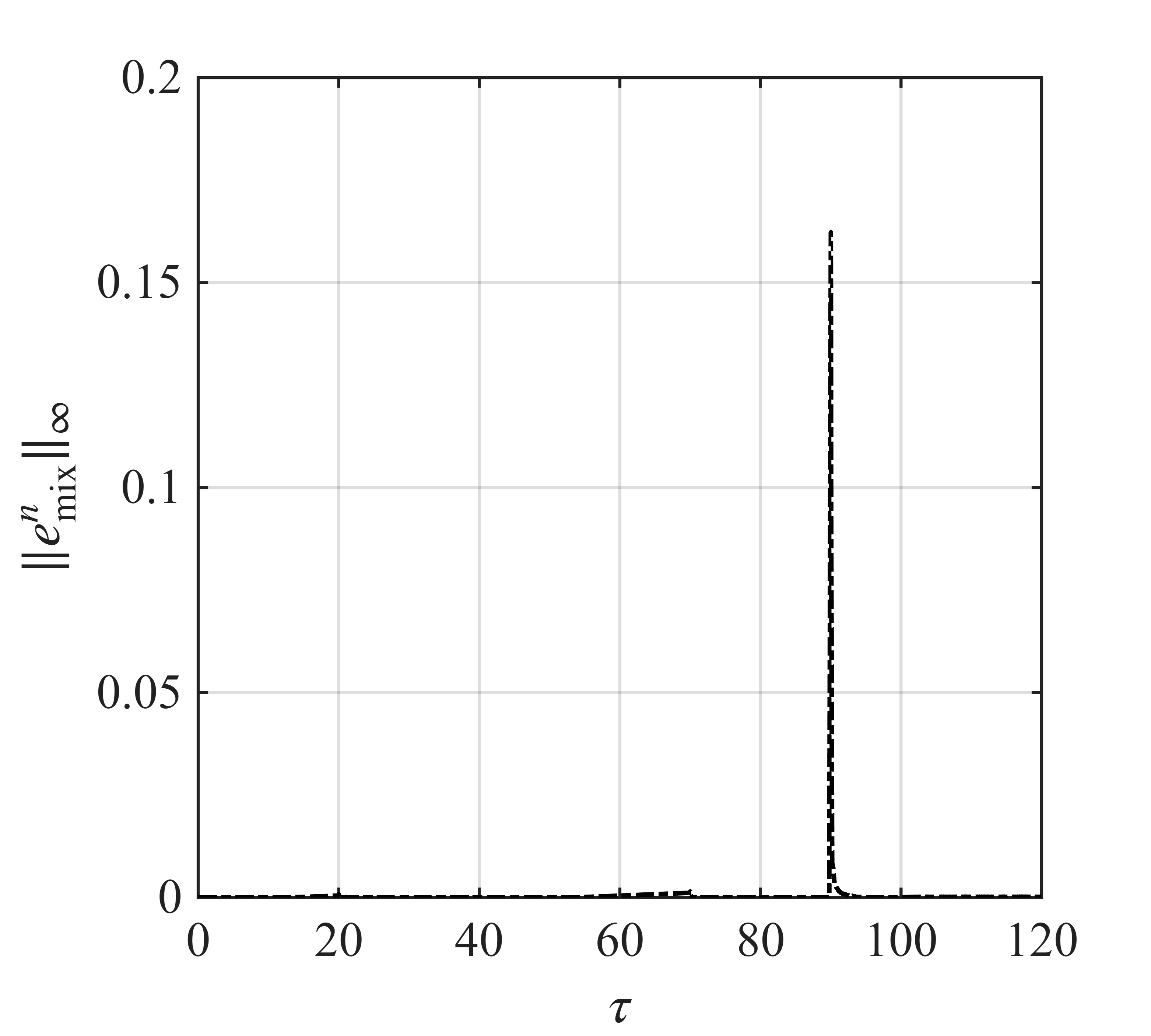}
		}
		\caption{Performance of IERK(4,7;2) with ATS-LDLB strategy ($\tau_{\max}=1$) for Freq-B.}
		\label{fig: ex3-Radau4-sin2T6-dt=1}
	\end{figure}
	
	We run IERK(4,7;2) with the ATS and ATS-LDLB strategies for two groups of frequency settings: (1) Freq-A with $\vec{l_t}=(3,1,5)$ and  (2) Freq-B with $\vec{l_t}=(40,20,50)$.
	As shown in Figure \ref{fig: ex3-Radau4-sin2T6-dt=0.5} with $\tau_{\max}=0.5$, the enstrophy curves generated by both strategies match the exact enstrophy well for Freq-A. However, their step-size behaviors differ significantly. The step sizes for the ATS strategy exhibit severe oscillations, particularly within the three oscillation intervals, cf. Figure \ref{fig: ex3-Radau4-sin2T6-dt=0.5} (b). In contrast, the ATS-LDLB strategy produces much smoother  variations of step-size. It allows the step size $\tau_n$ to increase gradually during the zero-forcing periods and maintains  appropriately small step sizes inside the oscillation intervals, cf. Figure \ref{fig: ex3-Radau4-sin2T6-dt=0.5} (c). For the increased frequencies in Freq-B, the original ATS strategy fails to resolve the rapid dynamics. As shown in Figures \ref{fig: ex3-Radau4-sin2T6-dt=0.5} (d)--(e), the enstrophy curve from ATS exhibits obvious discrepancies from the exact solution because the adaptive step cannot keep up with the fast oscillations. Conversely, the ATS-LDLB strategy performs well. It uses large steps when the forcing is zero and small steps for high-frequency oscillations, maintaining a steady step-size profile without the erratic switching observed in the ATS strategy, cf. Figure \ref{fig: ex3-Radau4-sin2T6-dt=0.5} (f).

	Furthermore, even when the maximum step size is increased to $\tau_{\max}=1$, as shown in Figure \ref{fig: ex3-Radau4-sin2T6-dt=1}, the ATS-LDLB strategy remains capable of accurately simulating the high-frequency oscillations in Freq-B. These results demonstrate that the ATS-LDLB strategy is much more robust and reliable than the original ATS strategy for problems with high-frequency disturbances.

	\newpage
	

	
	\def\siamprelabel{SM}
	\renewcommand{\thesection}{SM\arabic{section}}
	\renewcommand{\thetable}{S\arabic{table}}   
	\renewcommand{\thefigure}{S\arabic{figure}} 
	\renewcommand{\theequation}{S\arabic{section}.\arabic{equation}} 
	\renewcommand{\thealgorithm}{S\arabic{section}.\arabic{algorithm}} 
	
	\setcounter{section}{0}
	\setcounter{page}{1}
	
	
	\begin{center}
			{\bfseries{\MakeUppercase{Supplementary Materials: Long-time stability of implicit-explicit Runge-Kutta 
									methods for two-dimensional incompressible Navier-Stokes equations}}}
		
		\vspace{1em}
		{\small HONG-LIN LIAO, XIAOMING WANG, XUPING WANG {\scriptsize AND} CAO WEN}
		\vspace{2em}	
	\end{center}
	
	
	CONTENT: This supplementary material includes detail derivations of simplified order conditions for IERK methods, the technical proofs of Lemmas \ref{lemma: nonlinear convection bound} and \ref{lemma: nonlinear convection H2 bound}, some new parameterized IERK methods up to fourth-order time accuracy, and the detail descriptions of two time adaptive algorithms: ATS-LD and ATS-LDLB strategies. 
	
	\section{Simplified order conditions for IERK methods}\label{appendix: order conditions}
	\setcounter{equation}{0}
	
	For the ordinary differential equation $y^\prime(t)=f(y)+g(y,t)$ with the linear setting $f(y)=Ly$,  we consider the following IERK scheme 
	\begin{align}
		y_{n + 1} = &\, y_n + \tau \sum_{i = 1}^s b_i f (Y_i) + \tau \sum_{i = 1}^s
		\hat{b}_i g (Y_i, t_n + \tau c_i),  \label{scheme:lin+nonl-nzz}\\
		Y_i = &\, y_n + \tau \omega_i, \quad\text{with}\quad \omega_i = \sum_{j = 1}^s a_{i,j} f (Y_j)
		+ \sum_{j = 1}^s \hat{a}_{i,j} g (Y_j, t_n + \tau c_j), 
		\label{stage:lin+nonl-nzz}
	\end{align}
	where we impose the canopy node condition, that is, $$\sum_{j = 1}^s a_{i, j} = c_i = \sum_{j = 1}^s \hat{a}_{i, j}.$$
	
	\textbf{(Expansion of exact solution)} For the solution $y(t_{n+1})$, the Taylor expansion at $t_n$ is
	\begin{equation}\label{ln-zz:exact solution}
		y (t_{n + 1}) = y (t_n) + \tau y' (t_n) + \frac{\tau^2}{2} y'' (t_n) +
		\frac{\tau^3}{6} y''' (t_n) + \frac{\tau^4}{24} y^{(4)} (t_n) + O (\tau^5),
	\end{equation}
	where, by using the equation $y'=f(y)+g(y,t)$, the time derivatives of $y$ can be expressed in terms of the functions $f$ and $g$ and their derivatives, 
	\begin{align*}
		y' (t) = &\, f + g,\qquad y'' (t) = (f_y + g_y) (f + g) + g_t,\\
		y''' (t) = &\, g_{y y} (f + g)^2 + 2 g_{y t} (f + g) + (f_y + g_y)^2 (f + g)  + g_t (f_y + g_y) + g_{t t},\\
		y^{(4)} (t) = &\, g_{y y y} (f + g)^3 + 3 g_{y y t} (f + g)^2 + 4 g_{y y}  (f_y + g_y) (f + g)^2\\
		&\, + (f_y + g_y)^3 (f + g) + 3 g_{y t t} (f + g) + 5 g_{y t} (f_y + g_y) (f + g) + 3 g_{y y} g_t (f + g)\\
		&\, + g_{t t} (f_y + g_y) + (f_y + g_y)^2 g_t + 3 g_{y t} g_t + g_{t t t}.
	\end{align*}
	
	\textbf{(First-order conditions)} The first-order approximations of $f(Y_i)$ and $g(Y_i, t_{n,i})$, and the increment $\omega_i$ are given by $f (Y_i) = f + O (\tau)$, $g (Y_i, t_{n,i}) = g + O (\tau)$ and
	\begin{align}
		\omega_i = &\, \sum_{j = 1}^s a_{i,j} f + \sum_{j = 1}^s \hat{a}_{i,j} g + O  (\tau) = c_i (f + g) + O (\tau) .  \label{ln-zz:1st-omega}
	\end{align}
	The numerical solution given by \eqref{scheme:lin+nonl-nzz} reduces to
	\[ y_{n + 1} = y_n + \tau \sum_{i = 1}^s b_i f + \tau \sum_{i = 1}^s \hat{b}_i   g + O (\tau) . \]
	Matching the coefficients in the expansions of numerical and exact solutions gives the first-order conditions 
	\begin{equation}
		\sum_{i = 1}^s b_i = 1, \qquad \sum_{i = 1}^s \hat{b}_i = 1.  \label{1st-order condition}
	\end{equation}
	
	\textbf{(Second-order conditions)} Substituting the first-order approximation of the increments \eqref{ln-zz:1st-omega} into the Taylor expansions of  $f(Y_i)$, $g(Y_i, t_{n,i})$ and  $\omega_i$  yields the second-order approximations $f (Y_i) = f + \tau c_i f_y (f + g) + O (\tau^2),$
	\begin{align}
		  &\,	  
		  g (Y_i, t_{n + i}) = g + \tau c_i g_y (f + g) + \tau c_i g_t + O(\tau^2),  \label{ln-nzz:2nd-fg}\\
		&\,\omega_i =  c_i (f + g) + \tau \braB{\sum_{j = 1}^s a_{i,j} c_j f_y +  \sum_{j = 1}^s \hat{a}_{i,j} c_j g_y} (f + g)  \label{ln-nzz:2nd-omega}\\
		&\,\hspace{2cm}+ \tau \sum_{j = 1}^s  \hat{a}_{i,j} c_j g_t + O (\tau^2) . \notag
	\end{align}
	Substituting the second-order approximations \eqref{ln-nzz:2nd-fg} into  \eqref{scheme:lin+nonl-nzz} gives the following expansion
	\begin{align*}
	y_{n + 1} =&\, y_n + \tau \sum_{i = 1}^s b_i f + \tau \sum_{i = 1}^s  \hat{b}_i g + \tau^2 \braB{\sum_{i = 1}^s b_i c_i f_y + \sum_{i = 1}^s  \hat{b}_i c_i g_y} (f + g) \\&\,+ \tau^2 \sum_{i = 1}^s \hat{b}_i c_i g_t  + O (\tau^3),
	\end{align*}
	while the expansion \eqref{ln-zz:exact solution} of the exact solution gives
	\[ y (t_{n + 1}) = y (t_n) + \tau (f + g) + \frac{\tau^2}{2} \kbrab{(f_y + g_y) (f + g) + g_t} + O (\tau^3) . \]
	Matching the coefficients  of $\tau$ and $\tau^2$ terms in the expansions of numerical and exact solutions yields  the first-order conditions \eqref{1st-order condition} and the following conditions 
	\begin{equation}
		\sum_{i = 1}^s b_i c_i = \frac{1}{2}, \qquad \sum_{i = 1}^s \hat{b}_i c_i =  \frac{1}{2} .  
		\label{2nd-order condition}
	\end{equation}
	
	As expected, the linearity assumption of $f$ will not lead to any simplifications of order conditions. Actually, the first and second order conditions \eqref{1st-order condition} and \eqref{2nd-order condition} are derived from the zeroth- and first-order derivative terms of  $f$ and $g$ but do not involve any higher order time derivatives.

	\textbf{(Third-order conditions)}  Substituting the second-order approximation of the increments \eqref{ln-nzz:2nd-omega} into the Taylor expansions of $f(Y_i)$, $g(Y_i, t_{n,i})$ and $\omega_i$ yields the third-order approximations
	\begin{align}
		f (Y_i) = &\, f + \tau c_i f_y (f + g) + \tau^2 \braB{\sum_{j = 1}^s a_{i,j} c_j f_y f_y + \sum_{j = 1}^s \hat{a}_{i,j} c_j f_y g_y} (f + g)\notag\\
		&\, +  \tau^2 \sum_{j = 1}^s \hat{a}_{i,j} c_j f_y g_t + O (\tau^3), \nonumber\\
		g (Y_i, t_{n, i}) = &\, g + \tau c_i g_y (f + g) + \tau^2 \braB{\sum_{j =  1}^s a_{i, j} c_j f_y g_y + \sum_{j = 1}^s \hat{a}_{i,j} c_j g_y g_y}  (f + g)  \label{ln-nzz:3rd-fg}\\
		&\,+ \tau^2 \sum_{j = 1}^s \hat{a}_{i, j} c_j g_y g_t + \tau c_i g_t \notag\\&\,+ \frac{\tau^2}{2} c_i^2 [(f + g)^2 g_{y y} + 2 (f + g)  g_{y t} + g_{t t}] + O (\tau^3), \notag 
\\
		\omega_i = &\, c_i (f + g) + \tau \sum_{j = 1}^s \hat{a}_{i, j} c_j g_t +  \tau \braB{\sum_{j = 1}^s a_{i, j} c_j f_y + \sum_{j = 1}^s \hat{a}_{i, j}  c_j g_y} (f + g)\label{ln-nzz:3rd-omega}\\
		&\, + \tau^2 \braB{\sum_{j = 1}^s a_{i, j} \sum_{q = 1}^s a_{j, q} c_q f_y  f_y + \sum_{j = 1}^s \hat{a}_{i, j} \sum_{q = 1}^s \hat{a}_{j, q} c_j g_y g_y} (f + g) \nonumber\\
		&\, + \tau^2 \braB{\sum_{j = 1}^s \hat{a}_{i, j} \sum_{q = 1}^s a_{j, q}  c_q + \sum_{j = 1}^s a_{i, j} \sum_{q = 1}^s \hat{a}_{j, q} c_q} f_y  g_y (f + g) \nonumber\\
		&\, + \tau^2 \braB{\sum_{j = 1}^s a_{i, j} \sum_{q = 1}^s \hat{a}_{j, q}  c_q f_y + \sum_{j = 1}^s \hat{a}_{i, j} \sum_{q = 1}^s \hat{a}_{j, q} c_q g_y} g_t \nonumber\\
		&\, + \frac{\tau^2}{2} \sum_{j = 1}^s \hat{a}_{i, j} c_j^2 [(f + g)^2 g_{y  y} + 2 (f + g) g_{y t} + g_{t t}] + O (\tau^3) .  \notag
	\end{align}
	Substituting the third-order approximations \eqref{ln-nzz:3rd-fg} into the numerical scheme \eqref{scheme:lin+nonl-nzz} yields the coefficients of  $\tau^3$ term,
	\begin{align*}
		& \braB{\sum_{i = 1}^s b_i \sum_{j = 1}^s a_{i, j} c_j f_y f_y + \sum_{i = 1}^s \hat{b}_i \sum_{j = 1}^s \hat{a}_{i, j} c_j g_y g_y} (f + g)\\
		&\, + \braB{\sum_{i = 1}^s \hat{b}_i \sum_{j = 1}^s a_{i, j} c_j + \sum_{i = 1}^s b_i \sum_{j = 1}^s \hat{a}_{i, j} c_j} f_y g_y (f + g)\\
		& + \braB{\sum_{i = 1}^s b_i \sum_{j = 1}^s \hat{a}_{i, j} c_j f_y +  \sum_{i = 1}^s \hat{b}_i \sum_{j = 1}^s \hat{a}_{i, j} c_j g_y} g_t \\&\,+ \frac{1}{2} \sum_{i = 1}^s \hat{b}_i c_i^2 [(f + g)^2 g_{y y} + 2 (f + g) g_{y t} + g_{t t}],
	\end{align*}
	while the coefficients of $\tau^3$ term in the expansion of exact solution are
	\[ \frac{1}{6} [(f_y + g_y)^2 (f + g) + g_t (f_y + g_y) + g_{y y} (f + g)^2 +  2 g_{y t} (f + g) + g_{t t}]. \]
	Matching the coefficients for $\tau^3$ terms gives the third-order conditions
	\begin{align}
		& \sum_{i = 1}^s \hat{b}_i c_i^2 = \frac{1}{3}, && \sum_{i = 1}^s b_i \sum_{j = 1}^s a_{i, j} c_j = \frac{1}{6},  \label{3rd-order condition} \\
		& \sum_{i = 1}^s \hat{b}_i \sum_{j = 1}^s \hat{a}_{i, j} c_j = \frac{1}{6},&& \sum_{i = 1}^s \hat{b}_i \sum_{j = 1}^s a_{i, j} c_j = \frac{1}{6}, && \sum_{i = 1}^s b_i \sum_{j = 1}^s \hat{a}_{i, j} c_j = \frac{1}{6}.\notag
	\end{align}
	
	In the general nonlinear case, the condition $\sum_{i = 1}^s b_i c_i^2 = \frac{1}{3}$ will be required from the terms involving $f_{yy}$. Under the linear setting of $f$, we have $f_{yy} = 0$ and the entire term 
	$$\frac{1}{2} \sum_{i} b_i c_i^2 [(f + g)^2 f_{y y} + \cdots]$$ vanishes from the Taylor expansion of numerical solution such that the corresponding order condition becomes redundant. It is important to note that while the self-interaction terms of $f$ vanish, the first-order interactions between $f$ and $g$ (and their time derivatives) persist, giving rise to the coupling conditions in \eqref{3rd-order condition}.
	
	\textbf{(Fourth-order conditions)} 	Similarly, substituting the third-order approximation of the increments \eqref{ln-nzz:3rd-omega} into the Taylor expansions of $f(Y_i)$ and $g(Y_i,t_{n,i})$ yields $\tau^3$ terms. Substituting these into the numerical scheme \eqref{scheme:lin+nonl-nzz} and performing the lengthy but straightforward algebra yields the coefficients of $\tau^4$ terms in the expansion of numerical solution. 
	Matching the coefficients for $\tau^4$ terms in the expansions (the lengthy formulas are omitted here) of numerical and exact solutions gives the fourth-order conditions
	\begin{align*}
		& \sum_{i = 1}^s \hat{b}_i c_i^3 = \frac{1}{4}, \\
		&\sum_{i = 1}^s \hat{b}_i c_i \sum_{j = 1}^s \hat{a}_{i, j} c_j = \frac{1}{8}, \qquad\;\; \sum_{i = 1}^s \hat{b}_i c_i \sum_{j = 1}^s a_{i, j} c_j = \frac{1}{8}, \notag\\
		& \sum_{i = 1}^s \hat{b}_i \sum_{j = 1}^s \hat{a}_{i, j} c_j^2 = \frac{1}{12}, \qquad\; \sum_{i = 1}^s b_i \sum_{j = 1}^s \hat{a}_{i, j} c_j^2 = \frac{1}{12},\notag\\
		& \sum_{i = 1}^s \hat{b}_i \sum_{j = 1}^s a_{i, j} \sum_{q = 1}^s \hat{a}_{j, q} c_q + \sum_{i = 1}^s b_i \sum_{j = 1}^s \hat{a}_{i, j} \sum_{q = 1}^s \hat{a}_{j, q} c_q = \frac{1}{12},
		\\& \sum_{i = 1}^s b_i \sum_{j = 1}^s a_{i, j} \sum_{q = 1}^s a_{j, q} c_q = \frac{1}{24}, \notag\\
		& \sum_{i = 1}^s b_i \sum_{j = 1}^s \hat{a}_{i j} \sum_{q = 1}^s a_{j, q} c_q + \sum_{i = 1}^s \hat{b}_i \sum_{j = 1}^s a_{i, j} \sum_{q = 1}^s a_{j, q} c_q = \frac{1}{12},
		\\& \sum_{i = 1}^s b_i \sum_{j = 1}^s a_{i, j} \sum_{q = 1}^s \hat{a}_{j, q} c_q = \frac{1}{24}, \notag\\
		& \sum_{i = 1}^s \hat{b}_i \sum_{j = 1}^s \hat{a}_{i, j} \sum_{q = 1}^s a_{j, q} c_q = \frac{1}{24}, \quad \sum_{i = 1}^s \hat{b}_i \sum_{j = 1}^s \hat{a}_{i, j} \sum_{q = 1}^s \hat{a}_{j, q} c_q = \frac{1}{24}. 
	\end{align*}
	The key simplification stems directly from the linearity assumption of $f$. All terms involving higher-order derivatives of $f$ are removed so that the terms requiring the coupling conditions in the general nonlinear case either vanish or can be re-expressed using lower-order conditions.
	
	In summary, the order conditions for an IERK method applied to the system $y'=L y + g(y,t)$, where the linear term $L y$ is treated implicitly and the nonlinear term $g(y,t)$ is treated explicitly, are given in Table \ref{table: order condition}. This set of order conditions is a subset of classical order conditions in \cite{AscherRuuthSpiteri:1997-SM,IzzoJackiewicz:2017-SM}. They are significantly less restrictive, thereby greatly facilitating the design of fourth-order methods.
	
	\section{Proof of Lemma \ref{lemma: nonlinear convection bound}}\label{appendix: lemma3.1}
	\setcounter{equation}{0}
	
	\begin{proof} The definition \eqref{discrete: skew symmetric form} of nonlinear term $\mathtt{c}_h(\myvec{u},\omega)$ and the skew symmetric property \eqref{discrete: skew symmetric orthogonality} give
		\begin{align*} 
			H_{j}^{i+1}:=&\,2\myinnerb{\mathtt{c}_h(\myvec{u}^{n,j},\omega^{n,j}),\omega^{n,i+1}}
			\notag\\
			=&\,2\myinnerb{\mathtt{c}_h(\myvec{u}^{n,j},\omega^{n,j}-\omega^{n,i+1}),\omega^{n,i+1}}\notag\\
			=&\,\myinnerb{\myvec{u}^{n,j}\cdot\nabla_h (\omega^{n,j}-\omega^{n,i+1}),\omega^{n,i+1}}\notag\\&\,+\myinnerb{\nabla_h \cdot(\myvec{u}^{n,j}(\omega^{n,j}-\omega^{n,i+1})),\omega^{n,i+1}}\notag\\
			=&\,\myinnerb{\omega^{n,i+1}-\omega^{n,j},\nabla_h\cdot (\myvec{u}^{n,j}\omega^{n,i+1})}
			\notag\\&\,+\myinnerb{ \omega^{n,i+1}-\omega^{n,j},\myvec{u}^{n,j}\cdot\nabla_h\omega^{n,i+1}},
		\end{align*}
		where the discrete Green's formula has been used in the last equality. 
		
		It is to note that for the first  term $\nabla_h\cdot (\myvec{u}^{n,j}\omega^{n,i+1})$, it cannot be expanded as $\myvec{u}^{n,j}\cdot\nabla_h \omega^{n,i+1}$
		as in the Fourier Galerkin approximation, even though $\myvec{u}^{n,j}$ is divergence-free at the
		discrete level \eqref{discrete: divergence-free}. By following the derivations in \cite[(4.29)-(4.35)]{GottliebToneWangWangWirosoetisno:2012-SM}, one can apply Lemma \ref{lemma: collocation interpolation}, the H\"{o}lder's inequality, the Sobolev embedding and  the elliptic regularity to obtain that $$\mynormb{\nabla_h\cdot (\myvec{u}^{n,j}\omega^{n,i+1})}\le\ck_{\Omega}\mynormb{\myvec{\omega}_c^{n,j}}_{H^{\delta}}\mynormb{\nabla_h\omega^{n,i+1}}.$$ Moreover, the second term  $\myvec{u}^{n,j}\cdot\nabla_h\omega^{n,i+1}$ can be controlled
		in a similar way, $$\mynormb{\myvec{u}^{n,j}\cdot\nabla_h\omega^{n,i+1}}\le\ck_{\Omega}\mynormb{\myvec{\omega}_c^{n,j}}_{H^{\delta}}\mynormb{\nabla_h\omega^{n,i+1}}.$$ Through the above analysis, it can be concluded that
		
		\begin{align*} 
			H_{j}^{i+1}\le&\,\ck_{\Omega}\mynormb{\myvec{\omega}_c^{n,j}}_{H^{\delta}}
			\mynormb{\omega^{n,i+1}-\omega^{n,j}}\mynormb{\nabla_h\omega^{n,i+1}}\\
			\le&\, \ck_{\Omega}\cm_0\mynormb{\omega^{n,i+1}-\omega^{n,j}}\mynormb{\nabla_h\omega^{n,i+1}}\\
			\le&\, \ck_{\Omega}\cm_0\sum_{\ell=1}^{i}\mynormb{\delta_{\tau}\omega^{n,\ell+1}}\mynormb{\nabla_h\omega^{n,i+1}},
		\end{align*}
		where $\ck_{\Omega}$ is a positive constant  dependent on $\Omega$, but independent of the priori  bound $\cm_0$.
		Then we apply Lemma \ref{lemma: bound quadratic form} (iii) with $\alpha=\beta=2$ and the Cauchy-Schwarz inequality to get
		\begin{align*} 
			\tau_n\sum_{i=1}^k\sum_{j=1}^{i}\underline{\hat{a}}_{i+1,j}H_{j}^{i+1}
			\le&\, \ck_{\Omega}\cm_0\sigma_{\mathrm{E}}\tau_n\sqrt{\sum_{i=1}^ki\sum_{\ell=1}^{i}\mynormb{\delta_{\tau}\omega^{n,\ell+1}}^2}\sqrt{\sum_{i=1}^k\mynormb{\nabla_h\omega^{n,i+1}}^2}\\
			\le&\, \ck_{\Omega}\cm_0\sigma_{\mathrm{E}}\tau_n\sqrt{\sum_{\ell=1}^{k}\mynormb{\delta_{\tau}\omega^{n,\ell+1}}^2\sum_{i=1}^ki}\sqrt{\sum_{i=1}^k\mynormb{\nabla_h\omega^{n,i+1}}^2}\\
			\le&\, \frac{k+1}{\sqrt{2}}\ck_{\Omega}\cm_0\sigma_{\mathrm{E}}\tau_n\sqrt{\sum_{i=1}^{k}\mynormb{\delta_{\tau}\omega^{n,i+1}}^2}\sqrt{\sum_{i=1}^k\mynormb{\nabla_h\omega^{n,i+1}}^2}.
		\end{align*}
		Since $k+1\le s$, the Young's inequality arrives at the claimed result.
	\end{proof}
	
	\section{Proof of Lemma \ref{lemma: nonlinear convection H2 bound}}\label{appendix: lemma3.2}	
	\begin{proof} To handle the nonlinear convection term $\mathtt{c}_h(\myvec{u}^{n,j},\omega^{n,j})$, we make the following decompositions 
		\begin{align*} 
			&\myvec{u}^{n,j}\cdot\nabla_h\omega^{n,j}
			=\myvec{u}^{n,i+1}\cdot\nabla_h\omega^{n,i+1}
			-\myvec{u}^{n,j}\cdot\nabla_h(\omega^{n,i+1}-\omega^{n,j})\\
			&\,\hspace{3cm}
			-(\myvec{u}^{n,i+1}-\myvec{u}^{n,j})\cdot\nabla_h\omega^{n,i+1},\notag\\
			&\nabla_h\cdot(\myvec{u}^{n,j}\omega^{n,j})
			=\nabla_h\cdot\kbra{\myvec{u}^{n,i+1}\omega^{n,i+1}
				-\myvec{u}^{n,j}(\omega^{n,i+1}-\omega^{n,j})
				-(\myvec{u}^{n,i+1}-\myvec{u}^{n,j})\omega^{n,i+1}}.
		\end{align*}
		Thus, according to the definition \eqref{discrete: skew symmetric form},
		the term $\widehat{H}_{j}^{i+1}:=2\myinnerb{\mathtt{c}_h(\myvec{u}^{n,j},\omega^{n,j}),\Delta_h\omega^{n,i+1}}$
		can be handled by decomposing it into three parts
		\begin{align}\label{lemmaProof H2 bound-1}
			&\,\widehat{H}_{j}^{i+1}=\myinnerb{\myvec{u}^{n,j}\cdot\nabla_h \omega^{n,j},\Delta_h\omega^{n,i+1}}+\myinnerb{\nabla_h \cdot(\myvec{u}^{n,j}\omega^{n,j}),\Delta_h\omega^{n,i+1}}\\
		&\,\hspace{0.18cm}	=-\myinnerb{\nabla_h\cdot\kbra{\myvec{u}^{n,j}(\omega^{n,i+1}-\omega^{n,j})}+\myvec{u}^{n,j}\cdot\nabla_h(\omega^{n,i+1}-\omega^{n,j}),\Delta_h\omega^{n,i+1}}\notag
			\\
			&\,\hspace{0.18cm}-\myinnerb{(\myvec{u}^{n,i+1}-\myvec{u}^{n,j})\cdot\nabla_h\omega^{n,i+1}
				+\nabla_h\cdot\kbra{(\myvec{u}^{n,i+1}-\myvec{u}^{n,j})\omega^{n,i+1}},\Delta_h\omega^{n,i+1}}\notag
			\\
			&\,\hspace{0.18cm}+\myinnerb{\myvec{u}^{n,i+1}\cdot\nabla_h\omega^{n,i+1}+\nabla_h\cdot\bra{\myvec{u}^{n,i+1}\omega^{n,i+1}},\Delta_h\omega^{n,i+1}}\notag\\
			&\,\hspace{0.18cm}=-2\myinnerb{\mathtt{c}_h(\myvec{u}^{n,j},\omega^{n,i+1}-\omega^{n,j}),\Delta_h\omega^{n,i+1}}	\notag\\
			&\,\hspace{0.18cm}-2\myinnerb{\mathtt{c}_h(\myvec{u}^{n,i+1}-\myvec{u}^{n,j},\omega^{n,i+1}),\Delta_h\omega^{n,i+1}}\notag\\
			&\,\hspace{0.18cm}+2\myinnerb{\mathtt{c}_h(\myvec{u}^{n,i+1},\omega^{n,i+1}),\Delta_h\omega^{n,i+1}}
		=:\widehat{H}_{j,1}^{i+1}+\widehat{H}_{j,2}^{i+1}+\widehat{H}_{j,3}^{i+1}.\notag
		\end{align}
		By following the derivations in \cite[(4.49)-(4.54)]{GottliebToneWangWangWirosoetisno:2012-SM}, one can apply Lemma \ref{lemma: collocation interpolation}, the H\"{o}lder's inequality, the Sobolev embedding and  the elliptic regularity to obtain that
		\begin{align*} 
			\widehat{H}_{j,1}^{i+1}\le&\, {\ck}_{\Omega1}\cm_0\mynormb{\nabla_h(\omega^{n,i+1}-\omega^{n,j})}\mynormb{\Delta_h\omega^{n,i+1}}\\
			\le&\, {\ck}_{\Omega1}\cm_0\sum_{\ell=1}^{i}\mynormb{\delta_{\tau}\nabla_h\omega^{n,\ell+1}}\mynormb{\Delta_h\omega^{n,i+1}},
		\end{align*}
		where ${\ck}_{\Omega1}$ is a positive constant  dependent on $\Omega$, but independent of the priori  bound $\cm_0$. Then one applies Lemma \ref{lemma: bound quadratic form} (iii) with $\alpha=\beta=2$ and the Cauchy-Schwarz inequality to get
		\begin{align} \label{lemmaProof H2 bound-2}
			&\,\tau_n\sum_{i=1}^k\sum_{j=1}^{i}\underline{\hat{a}}_{i+1,j}\widehat{H}_{j,1}^{i+1}\\
			&\,\hspace{0.8cm}\le {\ck}_{\Omega1}\cm_0\sigma_{\mathrm{E}}\tau_n
			\sqrt{\sum_{i=1}^ki\sum_{\ell=1}^{i}\mynormb{\delta_{\tau}\nabla_h\omega^{n,\ell+1}}^2}
			\sqrt{\sum_{i=1}^k\mynormb{\Delta_h\omega^{n,i+1}}^2}\notag\\
			&\,\hspace{0.8cm}\le {\ck}_{\Omega1}\cm_0\sigma_{\mathrm{E}}\tau_n\sqrt{\sum_{\ell=1}^{k}\mynormb{\delta_{\tau}\nabla_h\omega^{n,\ell+1}}^2\sum_{i=1}^ki}\sqrt{\sum_{i=1}^k\mynormb{\Delta_h\omega^{n,i+1}}^2}\notag\\
					&\,\hspace{0.8cm}\le \frac{k+1}{\sqrt{2}}{\ck}_{\Omega1}\cm_0\sigma_{\mathrm{E}}\tau_n\sqrt{\sum_{i=1}^{k}\mynormb{\delta_{\tau}\nabla_h\omega^{n,i+1}}^2}\sqrt{\sum_{i=1}^k\mynormb{\Delta_h\omega^{n,i+1}}^2}\notag\\
			&\,\hspace{0.8cm}\le \frac{3s^2}{8\epsilon}{\ck}_{\Omega1}^2\cm_0^2\sigma_{\mathrm{E}}^2\tau_n\sum_{i=1}^{k}\mynormb{\delta_{\tau}\nabla_h\omega^{n,i+1}}^2+\frac{\epsilon\tau_n}{3}\sum_{i=1}^k\mynormb{\Delta_h\omega^{n,i+1}}^2.\notag
		\end{align}
		By following the derivations in \cite[(4.56)-(4.64)]{GottliebToneWangWangWirosoetisno:2012-SM}, one can apply the Poincar\'{e} inequality, the Sobolev embedding and  the elliptic regularity to obtain that
		\begin{align*} 
			\widehat{H}_{j,2}^{i+1}\le&\, {\ck}_{\Omega2}\cm_1^{\frac{3}2}\mynormb{\nabla_h(\omega^{n,i+1}-\omega^{n,j})}^{\delta}\mynormb{\Delta_h\omega^{n,i+1}}^{\frac{3}2}\\
			\le&\, {\ck}_{\Omega2}\cm_1^{\frac{3}2}\sum_{\ell=1}^{i}\mynormb{\delta_{\tau}\nabla_h\omega^{n,\ell+1}}^{\delta}\mynormb{\Delta_h\omega^{n,i+1}}^{\frac{3}2},
		\end{align*}
		where ${\ck}_{\Omega2}$ is a positive constant  dependent on $\Omega$, but independent of the priori  bound $\cm_0$. Then one applies Lemma \ref{lemma: bound quadratic form} (iii) with $\alpha=4,\;\beta=\tfrac43$  to get
		\begin{align*} 
			&\,\tau_n\sum_{i=1}^k\sum_{j=1}^{i} \underline{\hat{a}}_{i+1,j}\widehat{H}_{j,2}^{i+1}\\
				&\,\hspace{0.8cm}\le{\ck}_{\Omega2}\cm_1^{\frac{3}2}k^{\frac{1}{4}}\sigma_{\mathrm{E}}\tau_n
			\kbra{\sum_{i=1}^k\braB{\sum_{\ell=1}^{i}\mynormb{\delta_{\tau}\nabla_h\omega^{n,\ell+1}}^{\delta}}^4}^{\frac1{4}}
			\braB{\sum_{i=1}^k\mynormb{\Delta_h\omega^{n,i+1}}^{2}}^{\frac{3}{4}}\\
				&\,\hspace{0.8cm}\le
			\frac{9k}{64\epsilon}{\ck}_{\Omega2}^4\cm_1^{6}\sigma_{\mathrm{E}}^4\tau_n
			\sum_{i=1}^k\braB{\sum_{\ell=1}^{i}\mynormb{\delta_{\tau}\nabla_h\omega^{n,\ell+1}}^{\delta}}^4+
			\frac{\epsilon\tau_n}{3}\sum_{i=1}^k\mynormb{\Delta_h\omega^{n,i+1}}^{2},
		\end{align*}
		where the Young inequality, $ab\le \frac{1}{4p\tilde{\epsilon}}a^p+\frac{\tilde{\epsilon}}{q}b^q$ for $p=4$, $q=\frac{4}{3}$ and $\tilde{\epsilon}=\frac{4\epsilon}{9}$, was used in the second inequality. Moreover,
		\begin{align*} 
			\sum_{i=1}^k\braB{\sum_{\ell=1}^{i}\mynormb{\delta_{\tau}\nabla_h\omega^{n,\ell+1}}^{\delta}}^4
			\le&\,\sum_{i=1}^ki^2\braB{\sum_{\ell=1}^{i}\mynormb{\delta_{\tau}\nabla_h\omega^{n,\ell+1}}^{2\delta}}^2\\
			\le&\,\sum_{i=1}^ki^3\sum_{\ell=1}^{i}\mynormb{\delta_{\tau}\nabla_h\omega^{n,\ell+1}}^{4\delta}
			\le k^4\sum_{\ell=1}^{k}\mynormb{\delta_{\tau}\nabla_h\omega^{n,\ell+1}}^{4\delta}\\
			\le&\, k^4\sum_{\ell=1}^{k}\brat{1-2\delta}+2k^4\delta\sum_{\ell=1}^{k}\mynormb{\delta_{\tau}\nabla_h\omega^{n,\ell+1}}^{2}.
		\end{align*}
		Thus it follows that
		\begin{align}  \label{lemmaProof H2 bound-3}
			&\,\tau_n\sum_{i=1}^k\sum_{j=1}^{i} \underline{\hat{a}}_{i+1,j}\widehat{H}_{j,2}^{i+1}
			\le
			\frac{9\delta }{32\epsilon}{\ck}_{\Omega2}^4s^5\cm_1^{6}\sigma_{\mathrm{E}}^4\tau_n
			\sum_{\ell=1}^{k}\mynormb{\delta_{\tau}\nabla_h\omega^{n,\ell+1}}^{2}\\
		&\,\hspace{2.2cm}	+
			\frac{\epsilon\tau_n}{3}\sum_{i=1}^k\mynormb{\Delta_h\omega^{n,i+1}}^{2}+\frac{9}{64\epsilon}\brat{1-2\delta}{\ck}_{\Omega2}^4s^6\cm_1^{6}\sigma_{\mathrm{E}}^4\tau_n.\notag
		\end{align}
		For the term $\widehat{H}_{j,3}^{i+1}$ in \eqref{lemmaProof H2 bound-1}, one can follow the derivations in \cite[(4.68)-(4.73)]{GottliebToneWangWangWirosoetisno:2012-SM} to get
		\begin{align*} 
			\widehat{H}_{j,3}^{i+1}\le&\, {\ck}_{\Omega3}\cm_1^{\frac{3}2}\mynormb{\Delta_h\omega^{n,i+1}}^{\frac{3+\delta}2},
		\end{align*}
		where ${\ck}_{\Omega3}$ is a positive constant  dependent on $\Omega$, but independent of the priori bound $\cm_0$. Then one applies Lemma \ref{lemma: bound quadratic form} (iii) with $\alpha=1$ and $\beta=\infty$ to get
		\begin{align}  \label{lemmaProof H2 bound-4}
			\tau_n\sum_{i=1}^k\sum_{j=1}^{i}\underline{\hat{a}}_{i+1,j}\widehat{H}_{j,3}^{i+1}
			\le&\, {\ck}_{\Omega3}\cm_1^{\frac{3}2}k^{\frac1{2}}\sigma_{\mathrm{E}}\tau_n\sum_{i=1}^k\mynormb{\Delta_h\omega^{n,i+1}}^{\frac{3+\delta}2}\\
			\le&\, \frac{(1-\delta)}{4\epsilon}{\ck}_{\Omega3}^2\cm_1^{3}s^{2}\sigma_{\mathrm{E}}^2
			\tau_n
			+	\frac{\epsilon\tau_n}{3}\sum_{i=1}^k\mynormb{\Delta_h\omega^{n,i+1}}^{2},\notag
		\end{align}
		where the Young inequality, $ab\le \frac{1}{4p\tilde{\epsilon}}a^p+\frac{\tilde{\epsilon}}{q}b^q$ with 
		$$
		p=\frac{4}{1-\delta},\quad q=\frac{4}{3+\delta}\quad\text{and}\quad
		\tilde{\epsilon}=\frac{4\epsilon}{3(3+\delta){\ck}_{\Omega3}\sigma_{\mathrm{E}}}\cm_1^{-\frac{3}2}s^{-\frac1{2}},$$ was used in the second inequality. By using the decomposition \eqref{lemmaProof H2 bound-1}, we collect the above three estimates \eqref{lemmaProof H2 bound-2}-\eqref{lemmaProof H2 bound-4} to obtain the claimed result by taking the constants
		$$\hat{\ck}_{\Omega}:=6{\ck}_{\Omega1}^2\cm_0^2
		+\frac{9 }{2}\delta{\ck}_{\Omega2}^4\cm_1^{6}s^3\sigma_{\mathrm{E}}^2,$$
		and $$ \bar{\ck}_{\Omega}:=\frac{9}{4}\brat{1-2\delta}{\ck}_{\Omega2}^4\cm_1^{6}s^4\sigma_{\mathrm{E}}^2
		+4(1-\delta){\ck}_{\Omega3}^2\cm_1^{3}.$$
		Here, $\bar{\ck}_{\Omega}$ is dependent on  the upper bound $\cm_1$ of $L^2$ norm for the vorticity,  but always independent of the priori  bound $\cm_0$.
		The proof is completed.
	\end{proof}
	
	\section{New parameterized IERK methods}\label{sec: IERK methods}
	\setcounter{equation}{0}

	Second-order methods require at least two implicit stages ($s_{\mathrm{I}}=2$). Consider the 3-stage IERK methods that satisfy the canopy node condition and the two order conditions for first-order accuracy,
	\begin{equation*}
		\begin{array}{c|c}
			\mathbf{c} & A \\
			\hline\\[-10pt]   & \mathbf{b}^T
		\end{array}
		= \begin{array}{c|ccc}
			0 & 0 &  &      \\
			c_{2} & 0 & c_{2} &   \\
			1 & 0 & 1-a_{33} & a_{33}    \\
			\hline\\[-10pt]   &  0 & 1-a_{33} & a_{33}
		\end{array}\;,\quad
		\begin{array}{c|c}
			\hat{\mathbf{c}} & \widehat{A} \\
			\hline\\[-10pt]   & \hat{\mathbf{b}}^T
		\end{array}
		=\begin{array}{c|ccc}
			0 & 0 &  &     \\
			c_{2} & c_2 & 0 &   \\
			1 & 1-\hat{a}_{32} &\hat{a}_{32} & 0   \\
			\hline\\[-10pt]   & 1-\hat{a}_{32} &\hat{a}_{32} & 0
		\end{array}\;.
	\end{equation*}
	We determine the three independent coefficients $c_2$, $a_{33}$ and $\hat{a}_{32}$ using the two order conditions for second-order accuracy (see the second line in Table \ref{table: order condition}), which yields $$\hat{a}_{32}=\tfrac{1}{2c_2}\quad\text{and}\quad a_{33}=\frac{2 c_2-1}{2 (c_2-1)}. $$
	Thus, we obtain the parameterized IERK(2,3;$c_2$) methods with the following Butcher tableaux
	\begin{equation*}
		\begin{array}{c|c}
			\mathbf{c} & A \\
			\hline\\[-10pt]   & \mathbf{b}^T
		\end{array}
		= \begin{array}{c|ccc}
			0 & 0 &  &      \\
			c_{2} & 0 & c_{2} &   \\[3pt]
			1 & 0 & \frac{1}{2-2c_2} & \frac{1-2 c_2}{2-2c_2}    \\[3pt]
			\hline\\[-10pt]   &  0 & \frac{1}{2-2c_2} & \frac{1-2 c_2}{2-2c_2} 
		\end{array}\;,\qquad
		\begin{array}{c|c}
			\hat{\mathbf{c}} & \widehat{A} \\
			\hline\\[-10pt]   & \hat{\mathbf{b}}^T
		\end{array}
		=\begin{array}{c|ccc}
			0 & 0 &  &     \\
			c_{2} & c_2 & 0 &   \\[3pt]
			1 & 1-\tfrac{1}{2c_2} &\tfrac{1}{2c_2} & 0   \\[3pt]
			\hline\\[-10pt]   & 1-\tfrac{1}{2c_2} &\tfrac{1}{2c_2} & 0
		\end{array}\;.
	\end{equation*}
	The matrix $E_{s_{\mathrm{I}}}^{-1}A_{\mathrm{I}}$ of the difference coefficients in the implicit part  is given by
	\begin{align*}
		E_{s_{\mathrm{I}}}^{-1}A_{\mathrm{I}}=\begin{pmatrix}
				c_2 & 0 \\
				\frac{2 c_2^2-2 c_2+1}{2-2 c_2} & \frac{1-2 c_2}{2-2 c_2} 
			\end{pmatrix},\quad
	\end{align*}
	Note that, the matrix $\mathcal{S}(E_{s_{\mathrm{I}}}^{-1}A_{\mathrm{I}})$ is positive definite if $0.116337<c_2<0.434174$ or $1.15161 < c_2 < 4.29788$, which ensures that the parameterized IERK(2,3;$c_2$) methods satisfy the condition of Theorem \ref{thm: IERK long-time stability}.

	Third-order methods require at least three implicit stages ($s_{\mathrm{I}}=3$). However, we cannot find any four-stage IERK methods such that the associated matrix $E_{s_{\mathrm{I}}}^{-1}A_{\mathrm{I}}$ is positive definite; see the analysis in \cite[Subsection 4.1]{LiaoWangWen:2024IERK-SM}. We remark that the IERK methods in \cite[Subsection 2.8]{AscherRuuthSpiteri:1997-SM} do not fulfill Theorem \ref{thm: IERK long-time stability} since the associated matrices $E_{s_{\mathrm{I}}}^{-1}A_{\mathrm{I}}$ are not positive definite.
	
	We present the following parameterized IERK(3,5;$a_{55}$) methods with the following Butcher tableaux
	\begin{align*}
		&\begin{array}{c|c}
			\mathbf{c} & A \\
			\hline\\[-10pt]   & \mathbf{b}^T
		\end{array}
		= \begin{array}{c|ccccc}
			0 & 0 &  &  &  &  \\
			1 & 0 & 1 &  &  &  \\[3pt]
			\frac{1}{2} & 0 & -\frac{3}{10} & \frac{4}{5} &  &  \\[3pt]
			\frac{9}{10} & 0 & -\frac{367}{250} & \frac{196}{125} & \frac{4}{5} &  \\[3pt]
			1 & 0 & -\frac{2(36 a_{55}-5)}{147} & \frac{75 a_{55}+598}{588} & -\frac{25(15 a_{55}+2)}{588} & a_{55}  \\[3pt]
			\hline\\[-10pt]
			& 0 & -\frac{2(36 a_{55}-5)}{147}  & \frac{75 a_{55}+598}{588}  & -\frac{25(15 a_{55}+2)}{588} & a_{55} 
		\end{array},\\
		&\begin{array}{c|c}
			\hat{\mathbf{c}} & \widehat{A} \\
			\hline\\[-10pt]   & \hat{\mathbf{b}}^T
		\end{array}
		=\begin{array}{c|ccccc}
			0 & 0 &  &  &  &  \\
			1 & 1 & 0 &  &  &  \\[3pt]
			\frac{1}{2} & \frac{939 a_{55}+282}{2640 a_{55}+940} & \frac{381 a_{55}+188}{2640 a_{55}+940} & 0 &  &  \\[3pt]
			\frac{9}{10} & \frac{9 (639 a_{55}+1222)}{250 (132 a_{55}+47)} & \frac{9}{10} & -\frac{9 (639 a_{55}+1222)}{250 (132 a_{55}+47)} & 0 &  \\[4pt]
			1 & \frac{47}{270} & \frac{1}{10} & \frac{19}{30} & \frac{5}{54} & 0 \\[3pt]
			\hline\\[-10pt]
			& \frac{47}{270} & \frac{1}{10} & \frac{19}{30} & \frac{5}{54} & 0
		\end{array}\ .
	\end{align*}
	Simple calculations show that the matrix $\mathcal{S}(E_{s_{\mathrm{I}}}^{-1}A_{\mathrm{I}})$ is positive definite if $$0.626214<a_{55}<2.10996,$$ which ensures that the IERK(3,5;$a_{55}$) methods satisfy the  condition of Theorem \ref{thm: IERK long-time stability}.

	Fourth-order methods require at least four implicit stages ($s_{\mathrm{I}}=4$). Unfortunately, we were unable to construct any IERK methods with fewer  than seven stages such that the associated matrix $E_{s_{\mathrm{I}}}^{-1}A_{\mathrm{I}}$ is positive definite. 
	Under the new order condition in Table \ref{table: order condition}, we present the following parameterized IERK(4,7;$\hat{a}_{43}$) methods with the Butcher tableaux
	\begin{align*}
		\begin{array}{c|c}
			\mathbf{c} & A \\
			\hline\\[-10pt]   & \mathbf{b}^T
		\end{array}
		= \begin{array}{c|ccccccc}
			0 & 0 &  &  &  &  &  &  \\
			\frac{3}{4} & 0 & \frac{3}{4} &  &  &  &  &  \\[3pt]
			1 & 0 & -\frac{1}{2} & \frac{3}{2} &  &  &  &  \\[3pt]
			\frac{9}{20} & 0 & -\frac{169}{800} & \frac{129}{800} & \frac{1}{2} &  &  &  \\[4pt]
			\frac{3}{4} & 0 & a_{42} & a_{43}& \frac{4021588899578801}{4257206032921875} & \frac{144648284471}{278085937500} &  &  \\[4pt]
			\frac{3}{4} & 0 & a_{52} & a_{53}& \frac{2547104330002710487}{10159241669472656250} & -\frac{3921377950657453}{7299755859375000} & \frac{4}{5} &  \\[4pt]
			1 & 0 & \frac{94181}{262500} & -\frac{53}{100} & \frac{3}{5} & -\frac{125681}{262500} & \frac{4}{5} & \frac{1}{4} \\[3pt]
			\hline\\[-10pt]
			& 0 & \frac{94181}{262500} & -\frac{53}{100} & \frac{3}{5} & -\frac{125681}{262500} & \frac{4}{5} & \frac{1}{4}
		\end{array},
	\end{align*}	
	\begin{align*}
		\begin{array}{c|c}
			\hat{\mathbf{c}} & \widehat{A} \\
			\hline\\[-10pt]   & \hat{\mathbf{b}}^T
		\end{array}
		=\begin{array}{c|cccccccc}
			0 & 0 &  &  &  &  &  &  \\
			\frac{3}{4} & \frac{3}{4} & 0 &  &  &  &  &  \\[3pt]
			1 & \frac{7}{10} & \frac{3}{10} & 0 &  &  &  &  \\[3pt]
			\frac{9}{20} & \frac{1}{3} \hat{a}_{43}+\frac{1557}{4000} & \frac{243}{4000}-\frac{4}{3} \hat{a}_{43} & \hat{a}_{43} & 0 &  &  &  \\[3pt]
			\frac{3}{4} & \hat{a}_{51} & \hat{a}_{52} & \hat{a}_{53} & \frac{8}{5} & 0 &  &  \\[3pt]
			\frac{3}{4} & \hat{a}_{61} & \hat{a}_{62} & \hat{a}_{63} & -\frac{201267778}{267906771} & \frac{3}{10} & 0 &  \\[3pt]
			1 & \frac{25}{162} & -\frac{811}{540} & \frac{3}{22} & \frac{500}{891} & \frac{3}{4} & \frac{9}{10} & 0 \\[3pt]
			\hline\\[-10pt]
			& \frac{25}{162} & -\frac{811}{540} & \frac{3}{22} & \frac{500}{891} & \frac{3}{4} & \frac{9}{10} & 0
		\end{array}\ ,
	\end{align*}	
	with the coefficients
	\begin{align*}
		& a_{42} = -\frac{11099846794473413537}{13545655559296875000}, \quad
		a_{43} =\frac{5938991227245191}{56762747105625000}  \\
		&a_{52} = -\frac{15012700453574148059759}{355573458431542968750000}, 
		\quad
		a_{53} =\frac{37751222339857820917}{135456555592968750000} \\
		& \hat{a}_{51} = \frac{70997500000 \hat{a}_{43}+1042842334347}{2411160939000},   \\
		&\hat{a}_{61} = \frac{1913150328903-672359500000 \hat{a}_{43}}{2893393126800}, \\
		& \hat{a}_{52} = \frac{-283990000000 \hat{a}_{43}-5126845621293}{2411160939000},   \\
		&\hat{a}_{53} = \frac{70997500 \hat{a}_{43}}{803720313}+\frac{570851989}{676532250}, \\
		& \hat{a}_{62} = \frac{2689438000000 \hat{a}_{43}+3223449241353}{2893393126800}, \\
		&\hat{a}_{63} = \frac{-168089875000 \hat{a}_{43}-138406721733}{241116093900},
	\end{align*}
	Note that $\mathcal{S}(E_{s_{\mathrm{I}}}^{-1}A_{\mathrm{I}})$ is positive definite for any real value $\hat{a}_{43}$, ensuring that the parameterized IERK(4,7;$\hat{a}_{43}$) methods satisfy 
	the sufficient condition of Theorem \ref{thm: IERK long-time stability}.	
	
	\section{The ATS-LD and ATS-LDLB strategies}\label{appendix: ATS-LDLB strategy}	
	\setcounter{equation}{0}
	
	This section presents the detail descriptions of two adaptive time-stepping algorithms, including the ATS-LD (Algorithm \ref{Adaptive-Time-Strategy-Delay}) and ATS-LDLB (Algorithm \ref{Adaptive-Time-Strategy-Delay-Recompute}) strategies.

	Our idea for the ATS-LD strategy is rather simple: in addition to using large time steps to effectively accelerate the simulation of large-scale, low-frequency periodic motion, the ATS-LD algorithm can maintain a relatively small time step (rather than frequently changing the step size) to accurately capture small-scale chaotic or high-frequency quasi-periodic  behaviors. Specifically, we incorporate a \textbf{local delay mechanism} that fixes the step size over a duration governed by the maximum delay steps $d_{\max}$, as detailed in Lines 11--22 of Algorithm \ref{Adaptive-Time-Strategy-Delay}. This mechanism suppresses the increase of step size even when the $L^2$ norm of the discrete time derivative, $\mynormt{\partial_{\tau}\omega^{n}}$, is non-increasing. Conversely, the algorithm remains responsive to rapid variations, allowing for immediate reduction of step size if $\mynormt{\partial_{\tau}\omega^{n}}$ exceeds  $\mynormt{\partial_{\tau}\omega^{n-1}}$ during the delay period (see Lines 7--9). The local delay process is exited only when the solution exhibits smooth variation--when the standard variance $\gamma$ of the $L^2$ norm sequence $\mynormt{\partial_{\tau}\omega^{n+1-j}}$ for $1\le j\le d_{\max}$, stored in the vector $\myvec{a}$, falls below the given tolerance $\gamma_{\mathrm{tol}}$, see Lines 15--20 of Algorithm \ref{Adaptive-Time-Strategy-Delay}.
	
	\begin{algorithm}[htb!]
		\caption{Adaptive time-stepping with local delay (ATS-LD) }
		\label{Adaptive-Time-Strategy-Delay}
		\begin{algorithmic}[1]
			\Require{Final time $T$; Minimum time step $\tau_{\min}$; Maximum time step $\tau_{\max}$; Adaptation sensitivity $\beta$;  Maximum step ratio $r^{*}$; Maximum delay steps $d_{\max}$; Variance tolerance $\gamma_{\mathrm{tol}}$ of $L^2$ norm.}
			\State Set delay counter $d \gets 1$, store vector $\myvec{a} \gets \myvec{0} \in \mathbb{R}^{d_{\max}}$
			\State Set initial state $\omega^0$; Set step index $n \gets 1$, step size $\tau_n \gets \tau_{\min}$, and current time $t_{n-1} \gets 0$;
			\State Set $L^2$ norm $\|\partial_{\tau}\omega^{n-1}\| \gets \infty$;
			
			\While  {$t_{n-1} < T$}
			\State Update the time $t_{n}\gets t_{n-1}+\tau_{n}$.
			\State  Run the IERK method to compute the numerical solution $\omega^{n}$; Compute $\|\partial_{\tau}\omega^n\|$; 
			\If{$\|\partial_{\tau}\omega^n\| > \|\partial_{\tau}\omega^{n-1}\|$}
			\State Update time step $\tau_{n+1}\gets \tau_{\text{ada}}$ by \eqref{def: adaptive step size};
			\State Set delay counter $d \gets 1$; Reset store vector $\myvec{a} \gets \myvec{0}$.
			\Else
			\State Record the $L^2$ norm by $\myvec{a} (d) \gets \mynormt{\partial_{\tau}\omega^{n}}$;
			\If{$d<d_{\max}$}
			\State Update time step $\tau_{n+1} \gets \tau_n$ and the delay counter $d \gets d+1$.
			\Else
			\State Compute the standard variance of time series $\myvec{a}$, $\gamma \gets \mathrm{Var} (\myvec{a})$.
			\If{$\gamma<\gamma_{\mathrm{tol}}$} 
			\State Compute $\tau_{\text{ada}}$ by \eqref{def: adaptive step size}; Update time step $\tau_{n+1}\gets \tau_{\text{ada}}$; 	
			\Else 
			\State  Update time step $\tau_{n+1}\gets \tau_n$.		
			\EndIf
			\State Reset the delay counter $d\gets 1$ and store vector $\myvec{a} \gets \myvec{0}$.					
			\EndIf
			\EndIf	
			\If	{$t_n+\tau_{n+1}>T$}
			\State Update time step $\tau_{n+1}\gets T-t_n$.
			\EndIf
			\State Update step index $n \gets n + 1$.
			\EndWhile
		\end{algorithmic}
	\end{algorithm}

	Note that the ATS-LD algorithm degenerates into the traditional adaptive stepping strategy by setting the maximum delay parameter $d_{\max}=1$. Not surprisingly, due to the local delay mechanism ($d_{\max}>1$), the adaptive step size would avoid abrupt increases when the solution changes steadily, instead exhibiting a staircase-like increment, where the plateau width is controlled by the maximum delay steps $d_{\max}$. In this sense, although the maximum delay $d_{\max}$ is primarily chosen  to accurately capture the high-frequency quasi-periodic motions, an excessively  large value of $d_{\max}$ would not be desirable for accelerating the simulation of low-frequency solution. 
	
	\begin{algorithm}[htb!]
		\caption{Adaptive time-stepping with local delay and local backtrack}
		\label{Adaptive-Time-Strategy-Delay-Recompute}
		\begin{algorithmic}[1]
			\Require{Final time $T$; Minimum time step $\tau_{\min}$; Maximum time step $\tau_{\max}$; Adaptation sensitivity $\beta$; Maximum step ratio $r^{*}$;  Maximum delay steps $d_{\max}$; Variance tolerance $\gamma_{\mathrm{{tol}}}$ of $L^2$ norm; Variation threshold $\beta_{\text{thr}}$ for local backtrack.}
			\State Set delay counter $d \gets 1$; Initialize store vector $\myvec{a} \gets \myvec{0} \in \mathbb{R}^{d_{\max}}$; Set change rate threshold $\|\partial_{\tau}\omega^{n-1}\| \gets \infty$;
			\State Set initial state $\omega^0$; Set step index $n \gets 1$, step size $\tau_n \gets \tau_{\min}$, and current time $t_{n-1} \gets 0$;
			
			\While  {$t_{n-1} < T$}
			\State Update the time $t_{n}\gets t_{n-1}+\tau_{n}$;
			\State  Run the IERK method to compute the numerical solution $\omega^{n}$; Compute $\|\partial_{\tau}\omega^n\|$;  
			\If{$\|\partial_{\tau}\omega^n\| > \|\partial_{\tau}\omega^{n-1}\|$}
			\State Reset the delay counter $d\gets 1$ and store vector $\myvec{a} \gets \myvec{0}$;
			\If{$\|\partial_{\tau}\omega^n\| > \beta_{\text{thr}}\|\partial_{\tau}\omega^{n-1}\|$ and $\tau_{n} > \tau_{\min}$}
			\State Update time step $\tau_{n} \gets \max\{\tau_{\min}, \tau_{n}/\beta_{\text{thr}}\}$;
			\State Continue;
			\EndIf
			\State Compute the adaptive time step $\tau_{\text{ada}}$ by \eqref{def: adaptive step size};
			\State Update time step $\tau_{n+1}\gets \tau_{\text{ada}}$;
			\Else
			\State Record the $L^2$ norm by $\myvec{a} (d) \gets \mynormt{\partial_{\tau}\omega^{n}}$;
			\If{$d<d_{\max}$}
			\State Update time step $\tau_{n+1} \gets \tau_n$ and the delay counter $d \gets d+1$.
			\Else
			\State Compute the variance $\gamma \gets \mathrm{Var} (\myvec{a}) $ of time series $\myvec{a}$.
			\If{$\gamma<\gamma_{\mathrm{{tol}}}$} 
			\State Compute $\tau_{\text{ada}}$ by \eqref{def: adaptive step size}; Update time step $\tau_{n+1}\gets \tau_{\text{ada}}$; 	
			\Else 
			\State  Update time step $\tau_{n+1}\gets \tau_n$.		
			\EndIf
			\State Reset the delay counter $d\gets 1$ and store vector $\myvec{a} \gets \myvec{0}$.					
			\EndIf
			\EndIf	
			\If	{$t_n+\tau_{n+1}>T$}
			\State Update time step $\tau_{n+1}\gets T-t_n$.
			\EndIf
			\State Update step index $n \gets n + 1$.
			\EndWhile
		\end{algorithmic}
	\end{algorithm}
	
	Our idea for the refinement in ATS-LDLB (Algorithm \ref{Adaptive-Time-Strategy-Delay-Recompute}) is to mitigate the potential lags of step reduction at the interface between steady and oscillatory regimes--most notably near the on-off points of external source term. It is to detect the abrupt changes of solution in the long-time dynamics and introduce backtrack process before the accumulation of large local errors. If the variation rate of vorticity at the current step significantly exceeds that at the previous step, $\|\partial_{\tau}\omega^n\| > \beta_{\text{thr}}\|\partial_{\tau}\omega^{n-1}\|$ for a prescribed threshold $\beta_{\text{thr}}$, the current step is deemed unreliable and rejected. The current step size $\tau_n$ is scaled down ($\tau_n \gets \tau_n /\beta_{\text{thr}}$), and the solution is recomputed using this refined step size, cf. Lines 8--11 of Algorithm \ref{Adaptive-Time-Strategy-Delay-Recompute}. This \textbf{local backtrack mechanism} ensures that the adaptive procedure can maintain highly responsive to sudden changes of external forcing or high-frequency fluctuations, thereby balancing computational efficiency and improving accuracy at critical transition points.


\begin{thebibliography}{99}
		
		
		
		
		\bibitem{ArchillaNovo:2022} 	
		{\sc B. Archilla and J. Novo}, 
		Robust error bounds for the Navier-Stokes equations using implicit-explicit second order BDF method with variable steps. 
		{\em IMA J. Numer. Anal.}, 43:5 (2023), pp. 2892--2933.
		
		\bibitem{AscherRuuthSpiteri:1997}
		{\sc U. Ascher, S. Ruuth and R.J. Spiteri},
		Implicit-explicit Runge-Kutta methods for time dependent partial differential equations,
		{\em Appl. Numer. Math.}, 25:2~3 (1997), pp. 151--167.
		
		
		
		
		\bibitem{BoscarinoPareschiRusso:2017}
		{\sc S. Boscarino, L. Pareschi and G. Russo},
		A unified IMEX Runge-Kutta approach for hyperbolic systems with multiscale relaxation,
		{\em SIAM J. Numer. Anal.}, 55:4 (2017), pp. 2085--2109.
		
		\bibitem{BoisneaultDubuisPicasso:2023}
		{\sc A. Boisneault, S. Dubuis and M. Picasso},
		An adaptive space-time algorithm for the incompressible Navier-Stokes equations,
		{\em J. Comput. Phys.}, 493 (2023), 112457.
		
		
		\bibitem{BoscarinoRusso:2009}
		{\sc S. Boscarino and G. Russo},
		On a class of uniformly accurate IMEX Runge-Kutta schemes and applications to hyperbolic systems with relaxation,
		{\em SIAM J. Sci. Comput.}, 31:3 (2009), pp. 1926--1945.
		
		
		
		\bibitem{CalvoFrutosNovo:2001}
		{\sc M.P. Calvo, J. De Frutos and J. Novo},
		Linearly implicit Runge-Kutta methods for advection-reaction-diffusion equations,
		{\em Appl. Numer. Math.}, 37:4 (2001), pp. 535--549.
		
		
		
		\bibitem{CardoneJackiewiczSanduZhang:2014MMA}
		{\sc A. Cardone, Z. Jackiewicz, A. Sandu and H. Zhang},
		Extrapolated implicit-explicit Runge-Kutta methods,
		{\em Math. Model. Anal.}, 19:1 (2014), pp. 18--43.
		
		
		
		\bibitem{ChenGunzburgerSunWang:2013}
		{\sc W.~Chen, M.~Gunzburger, D.~Sun, X.-M.~Wang},
		Efficient and long-time accurate second-order methods for the Stokes-Darcy system, 
		{\em SIAM J. Numer. Anal.}, 51:5 (2013), pp. 2563--2584.
		
		\bibitem{ChenGunzburgerSunWang:2016}
		{\sc W.~Chen, M.~Gunzburger, D.~Sun, X.-M.~Wang},
		An efficient and long-time accurate third-order algorithm for the {Stokes--Darcy system},
		{\em Numer. Math.}, 134:4 (2016), pp. 857--879.
		
		\bibitem{ChenHuangLiaoYi:2026cicp}
		{\sc Y. Chen, Y. Huang, H.-L. Liao and N. Yi},
		Energy behaviors of implicit-explicit Runge-Kutta methods in adaptive finite element simulations for the Cahn-Hilliard model,
		{\em Commun. Comput. Phys.}, 2026, to appear.
		
		\bibitem{ChengWangxiaoming:2008}
		{\sc W. Cheng and X.-M.~Wang},
		A semi-implicit scheme for stationary statistical properties of the infinite Prandtl number model,
		{\em SIAM J. Numer. Anal.}, 47:1 (2008), pp. 250--270.
		
		\bibitem{ChengWang:2016}
		{\sc K.~Cheng, C.~Wang},
		Long time stability of high order multistep numerical schemes for two-dimensional incompressible Navier-Stokes equations,
		{\em SIAM J. Numer. Anal.}, 54:5 (2016), pp. 3123--3144.
		
		\bibitem{ChengWangWiseYue:2016Weakly}
		{\sc K.~Cheng, C.~Wang, S.~Wise and X.~Yue},
		A second-order, weakly energy-stable pseudo-spectral scheme for the Cahn-Hilliard equation and its solution by the homogeneous linear iteration method,
		{\em J. Sci. Comput.}, 69:3 (2016), pp. 1083--1114.
		
		\bibitem{ChuWangWangZhang: 2023}
		{\sc T. Chu, J. Wang, N. Wang and Z. Zhang},
		Optimal-order convergence of a two-step BDF method for the Navier-Stokes equations with $H^1$ initial data,
		{\em J. Sci. Comput.}, 96:2 (2023), 62.
		
		\bibitem{ConstantinFoias:1988}
		{\sc P. Constantin and C. Foias},
		{Navier-Stokes equations},
		University of Chicago Press, New York, 1988.
		
		
		
		
		%
		%
		%
		%
		
		\bibitem{DeCariaSchneier:2021}
		{\sc V. DeCaria and  M. Schneier},
		An embedded variable step IMEX scheme for the incompressible Navier-Stokes equations,
		{\em Comput. Methods Appl. Mech. Eng.}, 376 (2021), 113661.
		
		
		\bibitem{DimarcoPareschi:2013}
		{\sc G. Dimarco and L. Pareschi},
		Asymptotic preserving implicit-explicit Runge-Kutta methods for nonlinear kinetic equations,
		{\em SIAM J. Numer. Anal.}, 51:2 (2013), pp. 1064--1087.
		
		
		
		
		
		
		
		
		
		
		
		
		\bibitem{FoiasManleyRosaTemam:2001}
		{\sc C. Foias, O. Manley, R. Rosa, and R. Temam},
		{Navier-Stokes equations and turbulence}, 
		Cambridge University Press, Cambridge, 2001.
		
		\bibitem{FoiasTemam:1998JFA}
		{\sc C. Foias and R. Temam}, 
		Gevrey class regularity for the solutions of the Navier-Stokes equations,
		{\em J. Funct. Anal.}, 87:2 (1989), pp. 359--369.
		
		
		
		
		
		\bibitem{FuTangYang:2024}
		{\sc Z. Fu, T. Tang and J. Yang}, 
		Energy diminishing implicit-explicit Runge-Kutta methods for gradient flows, 
		{\em Math. Comput.}, 93:350 (2024), pp. 2745--2767.
		
		
		
		
		
		
		\bibitem{GottliebToneWangWangWirosoetisno:2012}
		{\sc S.~Gottlieb, F.~Tone, C.~Wang, X.-M.~Wang, D.~Wirosoetisno},
		Long time stability of a classical efficient scheme for two-dimensional Navier--Stokes equations,
		{\em SIAM J. Numer. Anal.}, 50:1 (2012), pp. 126--150.
		
		\bibitem{GottliebWang:2012}
		{\sc S.~Gottlieb and C.~Wang},
		Stability and convergence analysis of fully discrete {Fourier} collocation spectral method for 3-D viscous Burgers' equation,
		{\em J. Sci. Comput.}, 53:2 (2012), pp. 102--128.
		
		
		
		
		\bibitem{GuesmiGrotteschiStiller:2023}
		{\sc M. Guesmi, M. Grotteschi and J. Stiller},
		Assessment of high-order IMEX methods for incompressible flow,
		{\em Int. J. Numer. Meth. Fluids}, 95:6 (2023), pp. 954--978.
		
		
		
		
		
		
		
		
		
		\bibitem{IzzoJackiewicz:2017}
		{\sc G. Izzo and Z. Jackiewicz},
		Highly stable implicit-explicit Runge-Kutta methods,
		{\em Appl. Numer. Math.}, 113 (2017), pp. 71-92.
		
		
		
		
		
		
		\bibitem{KanevskyCarpenterGottliebHesthaven:2007}
		{\sc A. Kanevsky, M.H. Carpenter, D. Gottlieb and J.S. Hesthaven},
		Application of implicit-explicit high order Runge-Kutta methods to discontinuous-Galerkin schemes,
		{\em J. Comput. Phys.}, 225:2 (2007), pp. 1753--1781.
		
		
		
		
		\bibitem{HillSuli:2000}
		{\sc A. T. Hill and E. S\"{u}li}, 
		Approximation of the global attractor for the incompressible Navier-Stokes
		equations, {\em IMA J. Numer. Anal.}, 20:4 (2000), pp. 633--667.
		
		%
		
		\bibitem{HornJohnson:2013book}
		{\sc R.A. Horn and C.R. Johnson}, Matrix Analysis, Second Edition, 
		Cambridge University Press, New York, 2013.
		
		
		%
		%
		%
		%
		
		
		\bibitem{KennedyCarpenter:2003}
		{\sc C.A. Kennedy and M.H. Carpenter}, 
		Additive Runge-Kutta schemes for convection-diffusion-reaction
		equations, {\em Appl. Numer. Math.}, 44:1~2 (2003), pp. 139--181.
		
		
		
		
		%
		
		\bibitem{LiLiao:2022}
		{\sc Z. Li and H.-L. Liao},
		{Stability of variable-step BDF2 and BDF3 methods},
		{\em SIAM J. Numer. Anal.}, 60:4 (2022), pp. 2253--2272.
		%
		%
		%
		
		
		
		\bibitem{LiaoTangWangZhou:2025R-IERK}
		{\sc H.-L. Liao, T. Tang, X.-P. Wang and T. Zhou},
		{A class of refined implicit-explicit Runge-Kutta methods with 
			robust time adaptability and unconditional convergence for the Cahn-Hilliard model}, 
		{\em Math. Comput.}, 2025, doi: 10.1090/mcom/4090.
		
		
		\bibitem{LiaoWangWen:2024JCP}
		{H.-L. Liao,  X.-P. Wang and C. Wen},
		{Original energy dissipation preserving corrections of 
			integrating factor Runge-Kutta
			methods for gradient flow problems}, 
		{\em J. Comput. Phys.},  2024, 519: 113456.
		
		\bibitem{LiaoWangWen:2024IERK}
		{\sc H.-L. Liao, X.-P. Wang and C. Wen},
		Average energy dissipation rates of additive implicit-explicit 
		Runge-Kutta methods for gradient flow problems, 	
		{\em CSIAM Trans. Appl. Math.},
		2026, to appear.
		
		\bibitem{LiaoZhang:2021}
		{\sc H.-L. Liao and Z. Zhang},
		{Analysis of adaptive BDF2 scheme for diffusion equations},
		{\em Math. Comput.}, 90:329 (2021), pp. 1207--1226.
		
		\bibitem{LiuZou:2006}
		{\sc H. Liu and J. Zou},
		{Some new additive Runge-Kutta methods and their applications},
		{\em J. Comput. Appl. Math.}, 190:1~2 (2006), pp. 74--98.
		
		\bibitem{LiuShu:2000}
		{\sc J.-G. Liu and C.-W. Shu}, A high-order discontinuous Galerkin method for 2D incompressible flows, {\em J. Comput. Phys.}, 160:2 (2000), pp. 577--596.
		
		
		
		
		
		%
		
		
		\bibitem{NguyenPeraireCockburn: 2011}
		{\sc N. C. Nguyen, J. Peraire and B. Cockburn}, 
		An implicit high-order hybridizable discontinuous Galerkin method for the incompressible Navier-Stokes equations, {\em J. Comput. Phys.}, 230:4 (2011), pp. 1147--1170.
		
		
		\bibitem{PareschiRusso:2005}
		{\sc L. Pareschi and G. Russo},
		Implicit-explicit Runge-Kutta schemes and applications to hyperbolic systems with relaxation,
		{\em J. Sci. Comput.}, 25:1~2 (2005), pp. 129--155.
		
		
		
		
		
		
		\bibitem{ShenTangWang:2011Spectral}
		{J.~Shen, T.~Tang and L.~Wang},
		\textit{Spectral methods: Algorithms, analysis and applications},
		Springer-Verlag, Berlin Heidelberg, 2011.
		
		
		
		
		
		
		\bibitem{StuartHumphries:1998}
		{\sc A. M. Stuart and A. R. Humphries},
		{Dynamical systems and numerical analysis},
		Cambridge University Press, New York, 1998.
		
		\bibitem{Tachim-Medjo:1996}
		{\sc T. Tachim-Medjo}, Navier-Stokes equations in the vorticity-velocity formulation: The two-dimensional
		case, {\em Appl. Numer. Math.}, 21:2 (1996), pp. 185--206.
		
		\bibitem{Temam:1997}
		{\sc R. Temam},
		{Infinite-dimensional dynamical systems in mechanics and physics}, 2nd ed., Springer-Verlag, New York, 1997.
		
		\bibitem{ToneWangWirosoetisno:2015}
		{\sc F. Tone, X.-M.~Wang, and D. Wirosoetisno},
		Long-time dynamics of 2d double-diffusive convection: analysis and/of numerics.
		{\em Numer. Math.}, 130:3 (2015), pp. 541--566.
		
		\bibitem{ToneWirosoetisno:2006}
		{\sc F. Tone and D. Wirosoetisno}, On the long-time stability of the implicit Euler scheme for the
		two-dimensional Navier-Stokes equations, {\em SIAM J. Numer. Anal.}, 44:1 (2006), pp. 29--40.
		
		
		\bibitem{WangWangMao:2022}
		{\sc X.-M.~Wang, Z. Wang and M. Mao}, Linearly implicit variable step-size BDF schemes with Fourier pseudospectral approximation for incompressible Navier-Stokes equations, {\em Appl. Numer. Math.}, 172 (2022), pp. 393--412.
		
		\bibitem{Wang:2010}
		{\sc X.-M.~Wang}, Approximation of stationary statistical properties of dissipative
		dynamical systems: time discretization, {\em Math. Comput.}, 79:269
		(2010), pp. 259--280.
		
		\bibitem{Wang:2012}
		{\sc X.-M.~Wang}, An efficient second order in time scheme for approximating long time statistical
		properties of the two dimensional Navier-Stokes equations, {\em Numer. Math.}, 121:4 (2012),
		pp. 753--779.
		
		\bibitem{Wang:2016}
		{\sc X.-M.~Wang}, Numerical algorithms for stationary statistical properties of
		dissipative dynamical systems, {\em Dis. Contin. Dyn. Sys.},
		36:8 (2016), pp. 4599--4618.
		
		
		
		\bibitem{WangZhaoLiao:2025}
		{\sc X.-P. Wang, X. Zhao and H.-L. Liao}, A unified framework on the original energy laws of three effective classes of Runge-Kutta methods for phase field crystal type models, {\em SIAM J. Numer. Anal.}, 63:4 (2025), pp. 1808--1832.
		
		
		
		%
		
		
		
		
		
		
		
		
		
		
	\end{thebibliography}

\begin{thebibliography}{99}
	
	\bibitem{AscherRuuthSpiteri:1997-SM}
	{\sc U. Ascher, S. Ruuth and R.J. Spiteri},
	Implicit-explicit Runge-Kutta methods for time dependent partial differential equations,
	{\em Appl. Numer. Math.}, 25:2~3 (1997), pp. 151--167.
	

	
\bibitem{GottliebToneWangWangWirosoetisno:2012-SM}
{\sc S.~Gottlieb, F.~Tone, C.~Wang, X.-M.~Wang, D.~Wirosoetisno},
Long time stability of a classical efficient scheme for two-dimensional Navier--Stokes equations,
{\em SIAM J. Numer. Anal.}, 50:1 (2012), pp. 126--150.

	\bibitem{IzzoJackiewicz:2017-SM}
{\sc G. Izzo and Z. Jackiewicz},
Highly stable implicit-explicit Runge-Kutta methods,
{\em Appl. Numer. Math.}, 113 (2017), pp. 71-92.

	\bibitem{LiaoWangWen:2024IERK-SM}
{\sc H.-L. Liao, X.-P. Wang and C. Wen},
Average energy dissipation rates of additive implicit-explicit 
Runge-Kutta methods for gradient flow problems, 	
{\em CSIAM Trans. Appl. Math.},
2026, to appear.
	
\end{thebibliography}
\end{document}